\documentclass[reqno,11pt]{amsart}

\usepackage{amssymb,amsthm}
\usepackage{amsmath}
\usepackage{amsfonts}
\usepackage{dsfont}
\usepackage{orcidlink}

\usepackage[a4paper,  margin=3.3cm]{geometry}
\usepackage[active]{srcltx}
\makeatletter\@addtoreset{equation}{section}\makeatother

\newtheorem{theorem}{Theorem}[section]

\newtheorem{lemma}[theorem]{Lemma}

\newtheorem{proposition}[theorem]{Proposition}
\newtheorem{assumption}[theorem]{Assumption}
\newtheorem{definition}[theorem]{Definition}
\newtheorem{remark}[theorem]{Remark}
\numberwithin{equation}{section}

\title[The stochastic evolution of an infinite population]
{The stochastic evolution of an infinite population with logistic-type interaction}
\author{Yuri Kozitsky}

\address{Instytut Informatyki i Matematyki, Uniwersytet Marii Curie-Sk{\l}odowskiej, Lublin, Poland}
\email{jurij.kozicki@mail.umcs.pl}

\author{Michael R\"ockner}

\address{Fakult\"at f\"ur Mathematik, Universit\"at Bielefeld, Bielefeld, Germany}
\email{roeckner@math.uni-bielefeld.de}


\begin{document}


\begin{abstract}
An infinite population of point entities dwelling in
the habitat $X=\mathds{R}^d$ is studied. Its members
arrive in and depart from $X$ at random. The departure rate has a term corresponding to a logistic-type
interaction between the entities. Thereby, the corresponding Kolmogorov operator
$L$  has an additive quadratic term, which usually
produces essential difficulties in its study. The population
pure states are locally finite counting measures defined on $X$. The
set of such states $\Gamma$ is equipped with the vague
topology, which allows one to use probability measures defined thereon.  
The population evolution is described at two levels. At the first level, one deals 
with the Fokker-Planck equation for $(L,\mathcal{F},\mu_0)$ where
$\mathcal{F}$ is an appropriate set of bounded 
test functions $F:\Gamma\to \mathds{R}$ (domain of $L$) and $\mu_0$ is an initial state, which is supposed to belong to the set $\mathcal{P}_{\rm exp}$ of
sub-Poissonian probability measures on $\Gamma$. 
We prove that the Fokker-Planck
equation  has a unique solution $t\mapsto\mu_t$, which belongs to
$\mathcal{P}_{\rm exp}$.  Some of the properties of this solution are also described. The second-level description yields a Markov process with cadlag paths such that its one-dimensional marginals coincide with the mentioned states $\mu_t$. The process is obtained as the unique solution of the corresponding martingale problem.  The results obtained are discussed and compared with those known for similar models with logistic-type interactions.
\end{abstract}


\maketitle
\tableofcontents

\section{Introduction}
More than 200 years ago, T. R. Malthus  suggested describing the
evolution of a population by the differential equation
$\dot{N}_t = (\lambda - \mu)N_t$, where $N_t$ is the size of the population
 (number of entities) at time $t$, $\dot{N}_t$
 denotes the time derivative, and $\lambda$ and $\mu$ are
the rates of fertility and mortality, respectively. Later, P. F.
Verhulst modified this equation by making the mortality rate
state-dependent in the form $\mu= \mu_0 +\mu_1 N$ that takes into
account the increase of $\mu$ due to the competition between the
entities for the resources available at the habitat. This
modification leads to the following evolution equation $\dot{N}_t =
(\lambda -\mu_0) N_t - \mu_1 N_t^2$, now known as the logistic
growth equation \cite{Ver}. The appearance of a quadratic term
makes the theory more complex. In
particular, rough methods that do not take into account the sign of
the quadratic term are no longer adequate. At the same time,
this change in the model essentially alters the dependence of
$N_t$ on $t$. In particular, it gets globally bounded in time in 
contrast to the unbounded growth possible in Malthus' theory.

In individual-based (i.e., microscopic) modeling, population members are
assigned traits, typically spatial locations $x\in X$. Then,
the counting measures $\gamma$ defined on the trait space $X$ are
naturally used as pure population states. That is, for a
suitable $\Delta \subset X$, $\gamma(\Delta)$ is the size of the
subpopulation contained in $\Delta$ if the state of the whole
population is $\gamma$. Returning to the aforementioned modeling
amounts to restricting the theory to $N=\gamma(X)$. Here, however,
one obtains the possibility of studying also infinite populations,
which corresponds to admitting that $\gamma(\Delta)$ is finite not for all
subsets, e.g., only for bounded ones if $X=\mathds{R}^d$. The
set of all such states $\Gamma$ is equipped with a suitable topology
and hence with the corresponding Borel $\sigma$-field
$\mathcal{B}(\Gamma)$, which allows one to employ
probability measures on $(\Gamma,\mathcal{B}(\Gamma))$ as population
states. By $\mathcal{P}(\Gamma)$ we shall denote the set of all such
measures. In this setting, pure states $\gamma$ appear as the
corresponding Dirac measures. In the case of random events, the map $\omega \mapsto \gamma$, a \emph{random measure}, cf. \cite{Dawson,Li}, serves as a random 
`variable', whose distribution is given by a certain $\mu \in \mathcal{P}(\Gamma)$.    
Another advantage of the individual-based modeling
is that the evolution equation may now appear in its `dual' form
$\dot{F}_t = L F_t$, called the backward Kolmogorov equation. Here $F:\Gamma \to \mathds{R}$ is a suitable
test function, whereas $L$ is the Kolmogorov operator (generator)
which contains complete information concerning the elementary acts
of population dynamics.

In this paper, we consider the model in which $X=\mathds{R}^d$,
$d\geq 1$, and the Kolmogorov operator reads 
\begin{eqnarray}
  \label{L}
L & = & L^{+} + L^{-}, \\[.2cm] \nonumber
(L^{+}F)(\gamma) & = & \int_X b(x)[ F(\gamma\cup x ) - F(\gamma)] dx, \\[.2cm] \nonumber (L^{-}F)(\gamma) &
= & - \int_X \left(m(x) + \int_X a (x-y) (\gamma\setminus x)(d y)
\right) [  F(\gamma) -  F(\gamma\setminus x )] \gamma ( d x),
\end{eqnarray}
where we use the notations: $\gamma\cup x = \gamma +\delta_x$,
$\gamma\setminus x = \gamma -\delta_x$, $\delta_x$ being Dirac's
measure with atom at $x$. In this model, point entities 
randomly arrive and depart from the habitat $X=\mathds{R}^d$.  The
departure part $L^{-}$ is taken in logistic form: the 
departure rate from $\Delta \subset X$ is
\[
\int_\Delta m(x) \gamma (d x) + \int_{\Delta} \left( \int_X a(x-y)
(\gamma\setminus x) (d y)\right) \gamma(d x),
\]
where the second term corresponds to the departure due to the
influence (competition) of the rest of the population described by
\[
\int_X a(x-y) (\gamma\setminus x) (d y).
\]
The arrival
rate in a given $\Delta$ is $\int_\Delta b(x) dx$. It may be infinite for non-compact $\Delta$  and is
state-independent. In view of this, the model 
defined in \eqref{L} 
only partially corresponds to the Verhulst model 
mentioned above, in which the newcomers emerge from the already existing population. The latter aspect of the dynamics 
is taken into account in the Bolker-Pacala (or Bolker-Pacala-Dieckmann-Law) model, introduced in \cite{BP,BP1,Law,MDL} and studied in \cite{Ether,Dima,Four,KK2,KK1,Mail} and in several other publications. The generator corresponding to the Bolker-Pacala model has the same departure part $L^{-}$, 
whereas the arrival part now reads 
\begin{equation}
\label{BPLp}
(L^{+}F)(\gamma)  =  \int_X \left( \int_X a^{+}(x-y) \gamma (d y) \right)[ F(\gamma\cup x ) - F(\gamma)] dx,
\end{equation}
where $a^{+}$ is a dispersal kernel. The presence of the integral over $\gamma$ in \eqref{BPLp} complicates the theory. In a sense, the current study, combined with the results of \cite{KK1}, lays the foundation for constructing a similar Markov process for the Bolker-Pacala model.  The importance of such constructions stems from the fact that   
individual-based models of this kind, supported by appropriate simulation methods,
find numerous applications in various fields of knowledge; see, e.g.
\cite{Law,Mi,MW,Omel,Sad,WY} and the literature quoted in these works. In Sect. 3.3 below, we discuss in detail our present approach and results, also in light of their application to the Bolker-Pacala model, 
and compare them with those
obtained for such and similar models.

In relatively simple situations, the stochastic evolution of a given model is described by solving the backward Kolmogorov equation $\dot{F}_t = LF$, performed by constructing a $C_0$-semigroup acting in suitable Banach spaces of test functions, see, e.g., \cite[Chapt. II]{EN}. However, in our case, this direct way is rather impossible in view of the complex nature of $L$ given in \eqref{L}. In particular, for nontrivial test functions, the very point-wise finiteness of $L^{-}F$, and hence of $LF$,  is not obvious in advance. Instead, we obtain the   
 evolution of states as a map $[0,+\infty) \ni t \mapsto \mu_t\in
\mathcal{P}(\Gamma)$, which solves the Fokker-Planck equation
\begin{equation}
  \label{FPE}
\mu_t (F) = \mu_0 (F) +\int_0^t \mu_u(LF) du, \quad \mu|_{t=0} = \mu_0, \quad \mu(F):= \int F
d\mu,
\end{equation}
see \cite{FKP} for the general theory of such and similar equations.
Here $\mu_0$ is an initial state
 and $F$  is supposed to belong to a sufficiently representative class of functions,
$\mathcal{F}$, considered as the domain of $L$.
To stress this, we shall speak of the Fokker-Planck equation for
$(L,\mathcal{F}, \mu_0)$. Note that now $LF$ need not be point-wise finite, which a priori imposes the $\mu_t$-integrability 
condition on the solution of \eqref{FPE}.

Let $Z$ be an integer-valued random variable, and $\varphi_Z (\zeta ) = \mathds{E}\zeta^Z$ be its probability generating function. 
The $n$-th derivative at $\zeta=1$ (if exists) is the corresponding factorial moment of $Z$, i.e., $\phi_n(Z) = \mathds{E}Z (Z-1) \cdots (Z-n+1)$, see, e.g., \cite[Sect. 5.2, page 112]{DVJ}.  If $Z$ is Poissonian with parameter $\lambda$, then 
$\varphi_Z (\zeta ) = e^{\lambda(\zeta-1)}$ and hence $\phi_n(Z) = \lambda^n$. 
For a compact $\Lambda \subset X$, we set
\begin{equation}
 \label{Lambd}
 \Gamma^{(n)}_\Lambda = \{\gamma \in \Gamma: \gamma(\Lambda ) = n\}, \qquad n\in \mathds{N}_0.
\end{equation}
Then each such $\Lambda$ and $\mu\in \mathcal{P}(\Gamma)$ determine the distribution $p_{\Lambda, \mu}(n) = \mu(\Gamma^{(n)}_\Lambda)$ of a random variable, $Z_{\mu,\Lambda}$, which is Poissonian with parameter $\kappa(\Lambda)$ if $\mu=\pi_\kappa$, where the latter is the Poisson measure on $\Gamma$ with intensity measure $\kappa$. 
In dealing with states of infinite `particle' systems, one often
tries to confine consideration to a suitable subset of
$\mathcal{P}(\Gamma)$, which may bring additional
technical possibilities and also shed light on the properties of possible solutions. In particular, the appearance of a `heavy tail' of  $\{p_{\Lambda, \mu}(n)\}$ may indicate the possibility of `clumping' in $\Lambda$ in the state $\mu$. See subsect. 3.3 for more on this issue.   
As in \cite{KK1,KK,KR,KR1}, we shall use here the
set of \emph{sub-Poissonian} measures $\mathcal{P}_{\rm exp}$.
By definition, each such measure $\mu$ has the property 
\begin{equation}
 \label{Lambd1}
 \phi_n (Z_{\mu, \Lambda}) = \int_{\Lambda^n} k_\mu^{(n)} (x_1, \dots , x_n) dx_1 \cdots d x_n, \qquad n\in \mathds{N},
\end{equation}
with $k_\mu^{(n)}$ being positive symmetric elements of $L^\infty(X^n)$ satisfying Ruelle's bound, see Definition \ref{X3df} below. These $k_\mu^{(n)}$ are called \emph{correlation functions}, cf. \cite{KK,Kuna,Lenard,Ruelle,Zessin}, which completely characterize the corresponding state.

Our aim in this work is to construct a unique
Markov process with cadlag paths corresponding to (generated by)
$L$ given in (\ref{L}). Here we are going to follow the scheme
elaborated in our previous (rather lengthy) works \cite{KR,KR1}
based on solving a restricted martingale problem, see \cite[Chapter
5]{Dawson}. Here, however, the generator \eqref{L} is unbounded, which forced us  
to develop essentially new technical tools to implement this scheme, in particular, to include the $\mu_t$ -integrability condition in the very definition of a solution of \eqref{FPE}. 
The key feature of our scheme is that the one-dimensional
marginals of the corresponding path measures, which solve this
problem, belong to $\mathcal{P}_{\rm exp}$ and solve (\ref{FPE}). Along with evident technical benefits, the latter fact indicates that clumping does not appear in the course of the evolution described in this way.  
The uniqueness of path measure solutions means that all finite-dimensional marginals of two such path measures coincide. The latter is obtained by the fact that their
one-dimensional marginals  coincide, which
in turn is obtained by showing that the Fokker-Planck equation for
$(L,\mathcal{F},\mu_0)$ has a unique solution whenever $\mu_0$ is in
$\mathcal{P}_{\rm exp}$.  It should be pointed out here that $L$ as in
(\ref{L}) is a particular case of the generator studied in
\cite{GK}, where the corresponding Markov processes were obtained by
solving a stochastic equation involving $L$. However, uniqueness in \cite{GK} was obtained only for constant departure (mortality) rates, which in our notations corresponds to $m(x) \equiv m>0$ and $a(x) \equiv 0$, and hence excludes the logistic part. 
A similar model containing the logistic part was studied in \cite{Ether}, where uniqueness was proved only for processes with values in the set of finite configurations.  
 A more detailed comparison and related discussions are given in subsect. 3.3 below.

The remainder of this article consists of three parts. In the first part,  Sections 2 and 3, we introduce technicalities (Sect. 2) and formulate the results as Theorems \ref{1tm} and \ref{2tm} (Sect. 3). In the second part, Sections 4 and 5,  we
prove Theorem \ref{1tm},
which states that for each $\mu_0\in \mathcal{P}_{\rm exp}$, the Fokker-Planck equation \eqref{FPE} for $(L, \mathcal{F},\mu_0)$
has  a unique solution  $\mu_t\in \mathcal{P}_{\rm exp}$. Certain properties of this solution are also described. Here, uniqueness is meant in the class of all measures for which the very
solution of this equation can be defined, see Definition \ref{1df}. Among the key ingredients of the proof we mention: (a) the proper choice of the domain $\mathcal{F}$, see \eqref{X18}; (b) the proof that every solution of \eqref{FPE} for $(L, \mathcal{F},\mu_0)$ lies in $\mathcal{P}_{\rm exp}$, see Lemma \ref{01lm}; (c) the result of \cite{KK} where the evolution of states $t\mapsto \mu_t\in
\mathcal{P}_{\rm exp}$ describing the stochastic dynamics governed by
(\ref{L}) was obtained with the help of correlation functions. 
Typically, the domain $\mathcal{F}$ is taken as a subset of the set of all bounded
continuous functions $C_{\rm b}(\Gamma)$, and $L$ is supposed to
have the property $L:\mathcal{F}\to C_{\rm b}(\Gamma)$, or at least
$L:\mathcal{F}\to B_{\rm b}(\Gamma)$, where the latter is the set of
all bounded measurable functions. However, the presence of the
quadratic term in (\ref{L}) makes such a property barely possible
since in proving that $LF$ is bounded, the sign of this quadratic
term cannot be taken into account. In our approach, we take
$\mathcal{F}\subset C_{\rm b}(\Gamma)$ and define solutions of
(\ref{FPE}) as maps $t\mapsto \mu_t$ for which $LF$ is absolutely
$\mu_t$-integrable for (Lebesgue) almost all $t>0$, see Definition
\ref{1df}. The construction of $\mathcal{F}$ is made in such a way
that $\pm L^{\pm} F \geq 0$ for each $F\in \mathcal{F}$, see
 \eqref{0d}. This allows one to keep track of the
sign of the quadratic term in $L^{-}$. 

A significant property of measures $\mu\in \mathcal{P}_{\rm exp}$ is $\mu(\Gamma_*)=1$, where 
$\Gamma_* \subset \Gamma$ consists of those $\gamma$ for which $\psi \gamma$ is a finite measure on $X$, where $\psi(x) = (1+|x|^{d+1})^{-1}$. With the help of this property, $\Gamma_*$ can 
be endowed with the Polish topology induced by the weak topology of the set of all finite measures on $X$. 
Then each $\mu \in \mathcal{P}(\Gamma)$ with the property $\mu(\Gamma_*)=1$ can be redefined as a probability measure on $\Gamma_*$, see Proposition \ref{Gas1pn} and Remark \ref{Gasrk}.  We use this fact in the third part of the article, Sections 6 and 7, where we construct probability measures on the space $\mathfrak{D}_{[0,+\infty)}(\Gamma_*)$ of all cadlag paths with values in $\Gamma_*$. Here we mostly follow the scheme elaborated in our previous works \cite{KR,KR1}. In particular, the measures in question are obtained as unique solutions of the corresponding restricted martingale problems; see Definition \ref{Cad1df} and Theorem \ref{2tm}. In more detail, our approach is presented and commented in subsect. 3.3 below.

\section{Preliminaries}
\label{Psec}

By $\mathds{N}_0= \mathds{N}\cup \{0\}$ we denote the set of all nonnegative integers $0,1,2, \dots$, $\Lambda$ will always denote a compact subset of $X=\mathds{R}^d$. A Polish space, $E$ in general, is a separable space the topology of which is consistent with a complete metric, see, e.g., \cite{Cohn}. By $\mathcal{B}(E)$, $B_{\rm b}(E)$, $C_{\rm b}(E)$, $C_{\rm cs}(E)$ we denote the corresponding Borel $\sigma$-field, the sets of all bounded measurable, bounded  continuous,  and  continuous compactly supported functions, respectively. By $B^{+}_{\rm b}(E)$, $C^{+}_{\rm b}(E)$, $C^{+}_{\rm cs}(E)$ we mean the corresponding cones of positive elements, defined in a usual way.     
For a suitable set $\Delta$, by $\mathds{1}_\Delta$ we denote the corresponding indicator function.

\subsection{Configuration spaces and measures}

As mentioned above, by $\Gamma$ we denote the standard set of Radon
counting measures on $X=\mathds{R}^d$, which in the sequel are
called \emph{configurations}.  For $x\in X$ and $\gamma\in \Gamma$, we set
$n_\gamma(x) = \gamma(\{x\})$ and $p(\gamma) = \{x\in X:
n_\gamma(x)
>0\}$. The set $p(\gamma)$ is called the ground configuration
for $\gamma$, whereas $\gamma$ itself is the multiset $(p(\gamma), n)$, see, e.g., \cite{Bi}. 
The latter interpretation is consistent with the notations
\begin{equation*}
\int_X g(x) \gamma(d x) =\sum_{x\in \gamma} g(x)  = \sum_{x\in
p(\gamma)} n_\gamma(x) g(x),
\end{equation*}
where $g$ is a suitable numerical function. The weak-hash (vague)
topology of $\Gamma$ is defined as the weakest topology that makes
continuous all the maps
\[
\Gamma \ni \gamma \mapsto \sum_{x\in \gamma} g(x), \qquad g\in
C_{\rm cs}(X).
\]
With this topology $\Gamma$ is a Polish space; see, e.g. \cite{DVJ}.
By $\Gamma_{\rm fin}$ we
denote the subset of $\Gamma$ consisting of all finite
configurations, i.e., those that satisfy $\gamma (X) <\infty$. 
Along with the subspace topology induced on $\Gamma_{\rm fin}$ by the vague topology of
$\Gamma$ we define the following one, see \cite[Sect. 2.1]{Pilorz}. For multisets $\xi=\{x_1, \dots, x_n\}$ and 
$\eta=\{y_1, \dots, y_n\}$, define 
\[
\rho_n(\xi,\eta) = \min_{\sigma \in S_n} \sum_{i=1}^n |x_i - y_{\sigma(i)}|, 
\]
where $S_n$ is the symmetric group. Then set
\[
\rho_{\rm fin} (\xi, \eta) = \frac{\rho_{|\xi|}(\xi, \eta)}{1 + \rho_{|\xi|}(\xi, \eta)}, \qquad {\rm if} \quad |\xi | = |\eta|,
\]
and $\rho_{\rm fin} (\xi, \eta) =1$, otherwise.  One can show that $\rho_{\rm fin}$ is a complete metric, and the corresponding Borel $\sigma$-field $\mathcal{B}(\Gamma_{\rm fin})$ coincides with the 
$\sigma$-field $\{\Delta \in \mathcal{B}(\Gamma): \Delta \subset \Gamma_{\rm fin}\}$. Thus,
each measurable $G:\Gamma_{\rm fin} \to \mathds{R}$ is defined
by a sequence of symmetric Borel functions $\{G^{(n)}\}_{n\in
\mathds{N}_0}$ such that
\begin{equation}
  \label{Xa}
 G(\varnothing)= G^{(0)}, \quad {\rm and} \quad  G(\gamma) = G^{(n)} (x_1, \dots , x_n), \quad {\rm for} \ \
  \gamma= \{x_1, \dots , x_n\}.
\end{equation}
Here \emph{symmetric} means that
\begin{equation}
  \label{Xb}
\forall\sigma\in S_n \qquad G^{(n)} (x_1, \dots , x_n) =  G^{(n)}
(x_{\sigma(1)}, \dots , x_{\sigma(n)}).
\end{equation}
 \begin{definition}
  \label{X1df}
A measurable function $G:\Gamma_{\rm fin} \to \mathds{R}$ is said to
have bounded support if there exist $n\in \mathds{N}$ and a compact
$\Lambda \subset X$ such that the following holds: (a) $G^{(n)} =0$
for  $n>N$; (b) $G(\gamma)=0$ whenever $\gamma(\Lambda) <
\gamma(X)$. The set of all such functions which are bounded is denoted by $B_{\rm
bs}$.
\end{definition}
The Lebesgue-Poisson measure $\lambda$ on $\Gamma_{\rm fin}$ is
defined by the following integrals
\begin{equation}
  \label{Xc}
  \int_{\Gamma_{\rm fin}} G(\gamma) \lambda (d \gamma) = G(\varnothing)
  +\sum_{n=1}^{\infty} \frac{1}{n!}\int_{X^n} G^{(n) }(x_1 , \dots ,
  x_n)  dx_1\cdots d x_n.
\end{equation}
It is clear, that each $G\in B_{\rm bs}$ is absolutely
$\lambda$-integrable. The integral in the left-hand side of
(\ref{Xc}) has the following property, see e.g., \cite[Lemma
A.1]{Kuna},
\begin{equation}
  \label{XL}
\int_{\Gamma_{\rm fin}} G(\eta) \sum_{\xi \subset \eta}
H(\eta\setminus \xi, \xi) \lambda (d \eta) = \int_{\Gamma^2_{\rm
fin}} G(\eta\cup \xi )  H(\eta, \xi) \lambda (d \eta)\lambda (d
\xi),
\end{equation}
where $G$ and $H$ are suitable functions. In view of the multiset terminology adopted here, for $x\in p(\gamma)$, we
write $\gamma\setminus x= \gamma -\delta_x$, i.e., $n_{\gamma\setminus
x}(y) = n_\gamma (y)$ for $y\neq x$, and $n_{\gamma\setminus x}(x) =
n_\gamma (x)-1$. Similarly, $\gamma\cup y$, $y\in X$, stands for
$\gamma +\delta_y$.  We will also use notations
\begin{equation*}
  \sum_{x\in \gamma}\sum_{y\in \gamma\setminus x} g(x,y) =
  \int_{X^2} g(x,y) \gamma (dx) \gamma(dy) - \int_X g(x,x) \gamma(d
  x),
\end{equation*}
and their extensions
\begin{eqnarray}
  \label{X2}
& & \sum_{x_1\in \gamma} \sum_{x_2\in \gamma\setminus x_1} \cdots
\sum_{x_n\in \gamma\setminus \{x_1, \dots , x_{n-1}\}} g(x_1 , \dots
,
x_n) \\[.2cm] \nonumber & & = \sum_{\mathbb{G}\subset \mathbb{K}_n}
( -1)^{l_{\mathbb{G}}} \int_{X^{n_{\mathbb{G}}}} g_{\mathbb{G}}(y_1,
\dots , y_{n_{\mathbb{G}}}) \gamma(dy_1) \cdots
\gamma(dy_{n_{\mathbb{G}}}),
\end{eqnarray}
where each $\mathbb{G}$ is a spanning subgraph of the complete graph
$\mathbb{K}_n$ on $\{1, \dots , n\}$, $l_{\mathbb{G}}$ and
$n_{\mathbb{G}}$ are the number of edges and the connected
components of $\mathbb{G}$, respectively; $g_{\mathbb{G}}(y_1, \dots
, y_{n_\mathbb{G}})$ is obtained from $g(x_1 , \dots , x_n) $ be
setting $x_i = y_j$ for all $i$ belonging to $j$-th connected
component of $\mathbb{G}$. For $\gamma\in \Gamma$, by writing
$\gamma' \subset \gamma$ we mean a configuration such that
$p(\gamma')\subset p(\gamma)$ and $n_{\gamma'}(x) \leq n_\gamma(x)$,
$x\in p(\gamma')$. This means that $\gamma'$ is a multisubset of $\gamma$. In this case, we say that $\gamma'$ is a
\emph{sub-configuration} of $\gamma$.

Following \cite{Kuna,Lenard}, we now introduce \emph{correlation
measures}. For a given $n\in \mathds{N}$, a compact $\Delta \subset
X^n$ and $\gamma\in \Gamma$, let us consider
\begin{equation}
\label{X2a}
  {Q}^{(n)}_\gamma (\Delta) = \sum_{x_1\in \gamma}\sum_{x_2\in
  \gamma\setminus x_1}\cdots \sum_{x_n\in \gamma\setminus \{x_1, \dots ,
  x_{n-1}\}} \mathds{1}_\Delta (x_1 , \dots x_n).
\end{equation}
Clearly, ${Q}^{(n)}_\gamma$ is a counting measure: $
{Q}^{(n)}_\gamma(\Delta)$ is the number of tuples $(x_1, \dots ,
x_n)\in \Delta$ in state $\gamma$. In particular, ${Q}^{(1)}_\gamma = \gamma$.
It is known, \cite[Theorem
1]{Lenard}, that the map $\gamma \mapsto {Q}^{(n)}_\gamma (\Delta)$
is measurable for each compact $\Delta$. However, it may be
unbounded.
\begin{definition}
  \label{X2df}
A given $\mu \in \mathcal{P}(\Gamma)$ is said to have all
correlations if all $\gamma\mapsto {Q}^{(n)}_\gamma (\Delta)$, $n\in \mathds{N}$ are
$\mu$-integrable for all compact $\Delta \subset X^n$. By
$\mathcal{P}_{\rm cor} (\Gamma)$ we denote the set of all $\mu\in
\mathcal{P}(\Gamma)$ that have all correlations.
\end{definition}
For $\mu\in \mathcal{P}_{\rm cor} (\Gamma)$, one can define
\begin{equation}
  \label{X3}
\chi^{(n)}_\mu (\cdot) = \int_{\Gamma} Q^{(n)}_\gamma (\cdot) \mu
(d\gamma),
\end{equation}
which is called the \emph{correlation measure} of $n$-th order for
$\mu$, cf. \cite{Kuna,Lenard,Zessin}. Note that the factorial moment mentioned in \eqref{Lambd1} and the correlation measure are related to each other as follows
\begin{equation*}
 \phi_n(Z_{\mu,\Lambda}) = \chi^{(n)}_\mu (\Lambda^n).
\end{equation*}
In view of this, correlation measures are also called \emph{factorial moment measures}, cf. \cite[Chapt. 7]{DVJ}. For $G\in B_{\rm bs}$, we write, cf. \cite{Kuna},
\begin{equation}
  \label{X4}
  (K G)(\gamma) =\sum_{\eta\Subset \gamma} G(\eta),
\end{equation}
where $\eta\Subset \gamma$ means that the sum is taken over finite
sub-configurations of $\gamma$, including $\eta = \varnothing$. The advantage of using this $K$-map
can be seen from the following relation
\begin{equation}
  \label{X5}
 \mu (KG) = G(\varnothing) + \sum_{n=1}^\infty \chi_\mu^{(n)}
 (G^{(n)}), \qquad G\in B_{\rm bs},
\end{equation}
see \cite[Corollary 4.1]{Kuna}. We use this fact to introduce the
set of sub-Poissonian measures, which plays the key role in our
constructions. Let $\vartheta$ be a finite (nonempty) collection of
$\theta \in C^{+}_{\rm cs}(X)$, and $|\vartheta|$ stand for its
cardinality. We do not require that the members of $\vartheta$ are
distinct, i.e., $\vartheta$ is a multiset as well. For $n=|\vartheta|$, we set
\begin{equation}
  \label{X7}
  G^\vartheta(\gamma) = \left\{\begin{array}{ll} \frac{1}{n!}\sum_{\sigma\in
  S_n}   \prod_{i=1}^n \theta_i(x_{\sigma(i)}), \quad &{\rm for} \
  \gamma = \{x_1 , \dots , x_n\}, \\[.3cm] 0, \quad &{\rm
  otherwise}.
  \end{array} \right.
\end{equation}
That is, cf.\eqref{Xa},
\begin{gather}
\label{OCT}
(G^{\vartheta})^{(n)} =  g_\vartheta, \qquad  g_\vartheta (x_1 , \dots , x_n) := \frac{1}{n!} \sum_{\sigma \in S_n} \theta_1 (x_{\sigma(1)}) \cdots \theta_n (x_{\sigma(n)}), \\[.2cm] \nonumber  
(G^{\vartheta})^{(m)} = 0, \qquad m\neq n.
\end{gather}
For this function, we have
\begin{equation}
  \label{X6}
  \left( KG^\vartheta\right) (\gamma) = \frac{1}{n!}\sum_{x_1\in \gamma} \sum_{x_2\in
  \gamma\setminus x_1} \cdots \sum_{x_n \in \gamma\setminus \{x_1 , \dots ,
  x_{n-1}\}} \theta_1(x_1) \cdots \theta_n(x_n).
\end{equation}
Let $G_n^\theta$ be as in (\ref{X7}) with $|\vartheta|=n$ and all the
members of $\vartheta$
equal to a given $\theta\in C^{+}_{\rm cs}(X)$. Then the function
\begin{equation}
  \label{X6a}
   F^\theta(\gamma) = \prod_{x\in \gamma} (1+\theta (x)) = \exp\left(
 \sum_{x\in \gamma} \log (1+\theta (x) )\right),
\end{equation}
that possibly takes value $+\infty$, can be written in the form
\begin{equation}
  \label{X6b}
 F^\theta(\gamma) = 1+ \sum_{n=1}^\infty (K G_n^\theta)(\gamma).
\end{equation}
Among all $\mu\in \mathcal{P}(\Gamma)$ we distinguish Poisson
measures. Let $\kappa$ be a positive Radon measure on $X$.
Then the Poisson measure $\pi_\kappa$, for which $\kappa$ is the
\emph{intensity measure}, is defined as such that its correlation
measures satisfy $\chi^{(n)}_{\pi_\kappa} = \kappa^{\otimes n}$.
Then by (\ref{X5}), (\ref{X7}) and (\ref{X6b}) one gets
\begin{gather}
  \label{X6c}
  \pi_\kappa (F^\theta) = 1 + \sum_{n=1}^\infty \frac{1}{n!} \kappa^{\otimes n}
  (G^\theta_n) = \sum_{n=0}^\infty \frac{1}{n!} \left[\kappa(\theta)
  \right]^n = e^{\kappa(\theta)}.
\end{gather}
If $\kappa$ is absolutely continuous with respect to Lebesgue's
measure on $X=\mathds{R}^d$, then the Radon-Nikodym derivative $\rho
= d \kappa/dx$ may be an element of $L^\infty(X)$. A particular case
is a constant $\rho$, i.e., $\rho(x) \equiv \varkappa$ for some
$\varkappa>0$. The corresponding Poisson measure is called
\emph{homogeneous}. For simplicity, we denote it by $\pi_\varkappa$
and call $\varkappa$ the \emph{intensity} of $\pi_\varkappa$. In this
case,
\begin{equation*}
  \pi_\varkappa (F^\theta) = \exp\left( \varkappa \int_X \theta (x) d x \right).
\end{equation*}

\subsection{Sub-Poissonian measures}
In this article, the following class of measures on $\Gamma$ will 
be crucially used.
\begin{definition}
  \label{X3df}
$\mu \in \mathcal{P}(\Gamma)$ is said to be sub-Poissonian if it has
all correlations; see Definition \ref{X2df}, and for all $n\in \mathds{N}$ and $\vartheta=\{\theta_1, \dots , \theta_n\}$, $\theta_i\in C^{+}_{\rm
cs}(X)$, and thus for $G^\vartheta$ as in (\ref{X7}) -- (\ref{X6}), the following holds
\begin{equation}
  \label{X8}
n!  \mu (KG^\vartheta)= \chi^{(n)}_\mu (g_\vartheta) \leq
\varkappa^n \langle \theta_1
  \rangle \cdots \langle \theta_n \rangle, \qquad   \langle \theta_i
  \rangle := \int_X \theta_i (x) d x,
\end{equation}
for one and the same $\varkappa>0$. The least $\varkappa$ satisfying \eqref{X8} will
be called the type of $\mu$; $\mathcal{P}_{\rm exp}$ will denote the
set of all sub-Poissonian measures, while $\mathcal{P}^\alpha_{\rm exp}$, $\alpha \in \mathds{R}$, 
will represent the set of all $\mu \in \mathcal{P}_{\rm exp}$ whose type does not exceed $e^\alpha$.
\end{definition}
\begin{remark}
  \label{X1rk}
By Definition \ref{X3df} and (\ref{X6}), each $\mu\in \mathcal{P}_{\rm
exp}$ has the following properties:
\begin{itemize}
  \item[(a)] For every $n\in \mathds{N}$, the map $(\theta_1, \dots , \theta_n)  \mapsto
  \mu(KG^\vartheta)$ can be continued to a continuous symmetric $n$-linear functional on
  the real Banach space $L^1(X)$.
  \item[(b)] For $\theta\in C^{+}_{\rm cs}(X)$,
 the map $\theta \mapsto
  \mu(F^\theta)$, see (\ref{X6a}), can be continued to a  real exponential entire
  function of normal type defined on $L^1(X)$.
  \item[(c)] For each $H\in B^{+}_{\rm b}(\Gamma)$ such that $\mu(H)=: C_H>0$, the measure $\mu_H := C^{-1}_H H \mu$,  lies in $\mathcal{P}_{\rm exp}$ and the types of $\mu$ and $\mu_H$ satisfy \[\varkappa_{\mu_H} \leq \varkappa_{\mu} \max\{1; C^{-1}_H \sup H \}.\]    
\end{itemize}
\end{remark}
\begin{proof}
We begin the proof of (a) by recalling that $X^n$ is locally compact. Let $C_0(X^n)$ denote the real Banach algebra of all continuous functions on $X^n$ that vanish at infinity, equipped with the supremum norm. Then the linear span of the family of functions
\[
(\theta_1 \otimes \cdots \otimes \theta_n)(x_1, \dots , x_n) = \theta_1 (x_1) \cdots \theta_n(x_n), \quad \theta_i \in C^{+}_{\rm cs}(X),  
\]
is a sub-algebra of $C_0(X^n)$, which readily has the following properties: (i) it separates points of $X^n$; (ii) for  each $(x_1, \dots , x_n)$, one finds $(\theta_1, \dots , \theta_n)$ such that $(\theta_1 \otimes \cdots \otimes \theta_n)(x_1, \dots , x_n)\neq 0$. By the Stone-Weierstrass theorem for locally compact spaces, see  \cite{LdB}, 
it follows that the aforementioned family is dense in the Banach algebra $C_0(X^n)$. Then by 
\cite[Theorem 4.3, page 90]{Brezis}, it is also dense in $L^1 (X^n)$, which by the
estimate in  (\ref{X8}) yields the continuity in question. The stated linearity and symmetricity are immediate.

To prove claim (b), we first note that the
function
\begin{equation*}
  F^\theta_N (\gamma) := 1+ \sum_{n=1}^N (K G_n^\theta)(\gamma)
\end{equation*}
is certainly $\mu$-integrable, and then by (\ref{X6b}) and
(\ref{X8}), we get
\begin{equation*}
 \mu(F_N^\theta) \leq \exp\left(\varkappa_\mu \langle \theta \rangle
 \right),
\end{equation*}
where $\varkappa_\mu$ is the type of $\mu$. By the Beppo Levi
(monotone convergence) theorem, this yields the proof of claim (b). The proof of (c) is evident. 
\end{proof}
Along with the space of configurations $\Gamma$ introduced in subsection 2.1, one can use the space of simple configurations $\breve{\Gamma}= \{\gamma\in \Gamma: p(\gamma) = \gamma\}$, i.e., consisting of those $\gamma$ that satisfy $n_\gamma (x) \leq 1$. It is known, see e.g., \cite{Zessin}, that $\breve{\Gamma}$ is a $G_\delta$ subset of $\Gamma$. The latter means that each $\mu \in \mathcal{P}(\Gamma)$ with the property $\mu (\breve{\Gamma})=1$ can be redefined as a measure on $\breve{\Gamma}$; see \cite[Remark 2.9, page 11]{KR}. 
\begin{proposition}\cite[Lemma 2.10]{KR}
\label{JunePN}
Every $\mu \in \mathcal{P}_{\rm exp}$ has the property $\mu(\breve{\Gamma})=1$.
\end{proposition}
In view of this statement, we can and will use the results and techniques elaborated in \cite{KK,Kuna}. 

Let $\mu$ have all correlation measures $\chi^{(n)}_\mu$ defined in (\ref{X3}). Its correlation measure 
$\chi_\mu$ is defined as a locally finite positive measure on $\Gamma_{\rm fin}$ by the following formula
\begin{equation}
  \label{X10}
\int_{\Gamma_{\rm fin}} G(\eta) \chi_\mu(d \eta) = G(\varnothing) +
\sum_{n=1}^\infty \frac{1}{n!} \int_{X_n} G^{(n)} (x_1 , \dots ,
x_n) \chi_\mu^{(n)} (d x_1 , \dots , dx_n) ,
\end{equation}
holding for all $G\in B_{\rm bs}$. 
Then, by (\ref{Xc}) it follows that the Lebesgue-Poisson measure $\lambda$ is the correlation measure for $\pi_\varkappa$, $\varkappa=1$.

Recall that $KG$ is defined for all $G\in B_{\rm bs}$, see
(\ref{X4}). Keeping this in mind, we set
\begin{equation}
  \label{Y4}
  B^\star_{\rm bs} = \{G\in B_{\rm bs}: (KG)(\gamma) \geq 0 \quad {\rm for} \ {\rm all} \ \gamma\in
  \Gamma\}.
\end{equation}
Note that the cone of point-wise positive $G\in B_{\rm bs}$ is a
proper subset of $B^\star_{\rm bs}$. Assume that we have a (possibly signed) measure
$\chi$ on $\Gamma_{\rm fin}$. For a compact $\Lambda \subset X$, set $\omega^{(0)} = \chi^{(0)}$ and
\[
\omega^{(n)} (\Lambda^n) = |\chi^{(n)} (\Lambda^n)|/n!, \qquad n\in \mathds{N}.
\]
The next statement, see \cite[Theorem 2]{Lenard} and/or \cite[Theorem 2.2]{Zessin}, provides a tool
for relating such measures to elements of $\mathcal{P}_{\rm cor}(\Gamma)$.
\begin{proposition}
\label{JUNEpn}
A measure $\chi$ on $\Gamma_{\rm fin}$ is the correlation measure of a unique $\mu\in \mathcal{P}_{\rm cor}(\Gamma)$
if it satisfies the following conditions: (a) $\chi^{(0)} = 1 $; (b) $\chi(G) \geq 0$ for all $G \in B^\star_{\rm bs}$,
see (\ref{Y4}); (c) for all $n\in \mathds{N}$ and compact $\Lambda \subset X$, the following holds
\begin{equation}
\label{JUNNY} 
\sup_{n\in \mathds{N}} [\omega^{(n)} (\Lambda^n)]^{1/n}/n < \infty.
\end{equation}
\end{proposition}
The role of the aforementioned conditions is as follows: (a) yields normalization $\mu (\Gamma)=1$; (b) implies positivity of $\mu$; (c) secures uniqueness.   

The effects of the death term of the  generator \eqref{L} 
can be described by using the following notion.
\begin{definition}
  \label{Thdf}
Let $q: X\to (0,1]$ be measurable. For a given $\mu\in
\mathcal{P}_{\rm cor}(\Gamma)$,  its independent $q$-thinning is the measure
$\mu^q\in \mathcal{P}_{\rm cor}(\Gamma)$ defined by the correlation
measure $  \chi_{\mu^q}$ which has the following form
\begin{equation}
  \label{Y0a}
  \chi_{\mu^q} (d \eta) = e(\eta; q) \chi_\mu (d \eta), \qquad
  e(\eta; q) := \prod_{x\in \eta} q(x).
\end{equation}
\end{definition}
One could interpret $\mu^q$ as the law of configurations obtained from those distributed by $\mu$ by independently deleting their points with probability $1-q(x)$, where $x$ is the location of the point to be considered; see \cite[page 290]{DVJ}. If $\mu$ is Poisson, i.e., $\mu = \pi_\kappa$, then $\mu^q= \pi_{q\kappa}$. In the general case, $e(\cdot, q)\chi_\mu$ obviously satisfies conditions (a) and (c) of Proposition \ref{JUNEpn}, while the validity of (b) is not so immediate. In Appendix, we prove the following statement.
\begin{proposition}
\label{JUN1pn}
Let $q: X\to (0,1]$ be measurable and $\mu$ be in
$\mathcal{P}_{\rm cor}(\Gamma)$. Then $e(\cdot, q)\chi_\mu$ satisfies condition (b) of Proposition \ref{JUNEpn}. Hence, $\mu^q$ is also in $\mathcal{P}_{\rm cor}(\Gamma)$.
\end{proposition}
By standard arguments, see e.g., \cite[Theorem 4.14, page
99]{Brezis}, it follows that for each $\mu\in \mathcal{P}_{\rm
exp}$, its correlation measure $\chi_\mu$ is absolutely continuous
with respect to the Lebesgue-Poisson measure $\lambda$. Its
Radon-Nikodym derivative
\begin{equation}
  \label{X11}
  k_\mu  := \frac{d \chi_\mu}{d \lambda}
\end{equation}
is such that $k_\mu(\varnothing)=1$ and, for $\gamma = \{x_1, \dots
, x_n\}$, $n\in \mathds{N}$, the following holds, cf. \eqref{Lambd1},
\begin{equation}
  \label{X11a}
k_\mu(\gamma) = k_\mu^{(n)} (x_1 , \dots , x_n) := \frac{d
\chi^{(n)}_\mu}{dx_1 \cdots d x_n}(x_1 , \dots , x_n),
\end{equation}
where $k_\mu^{(n)}$ is a symmetric element of the corresponding
$L^\infty (X^n)$, see (\ref{Xb}). Then for $\mu\in \mathcal{P}_{\rm
exp}$ and $G\in B_{\rm bs}$, we get, cf. \eqref{Xc},
\begin{equation}
  \label{X12} \mu(KG) = \int_{\Gamma_{\rm fin}} G (\eta) k_\mu(\eta) \lambda
  ( d \eta) =: \langle\!\langle k_\mu , G \rangle\!\rangle.
\end{equation}
Since $\chi_\mu^{(n)}$ is positive, see (\ref{X2a}), (\ref{X3}), Proposition \ref{JUNEpn}, and
in view of (\ref{X8}), we have that
\begin{equation}
  \label{X13}
  0\leq k^{(n)}_\mu (x_1 , \dots , x_n) \leq \varkappa_\mu^n,
\end{equation}
holding for all $n$ and almost all $(x_1, \dots , x_n)$. 
The
upper estimate in (\ref{X13}) is known as Ruelle's bound.
For $G\geq
0$, by (\ref{X12}) and (\ref{X13}) one readily obtains 
\begin{equation}
  \label{X14}
\mu (KG) \leq \langle \! \langle k_{\pi_{\varkappa_\mu}}, G \rangle
\! \rangle = \pi_{\varkappa_\mu} (KG),
`\end{equation}
which may be interpreted as the `sub-Poissonicity' of $\mu$.
The estimate in \eqref{X14} points to the stochastic domination of states $\mu\in\mathcal{P}_{\rm exp}$ by Poisson states, which could be interpreted as the lack of strong correlations in such states. In fact, this is not the case since the states of thermal equilibrium of
interacting physical particles satisfying Ruelle's bound admit phase transitions induced by strong correlations, which are absent in Poisson states; see
\cite{Ruelle}. 

Sub-Poissonian probability distributions are known in various applications, ranging from genetics \cite{Weid} to quantum optics \cite{David}. In these applications, a random variable $Z$ is sub-Poissonian if its  Fano factor $F= {\rm Var}Z/\mathds{E}Z$ satisfies $F\leq 1$. It is a different property from that corresponding to \eqref{X8}. To see this, let us consider $Z$ for which $P(Z=k) = (1-\theta) \theta^{k-1}$, $k\geq 1$, $\theta \in (0,1)$. Its factorial moments $\phi_n(Z) = n! \theta^{n-1}(1-\theta)^{-n}$ fail to satisfy Ruelle's bound $\phi_n(Z)\leq \varkappa^n$. At the same time, this $Z$ is sub-Poissonian in the sense that Fano's factor $F = \theta/(1-\theta)$ is smaller than $1$ for $\theta < 1/2$.
Thus, a more appropriate name for $\mu \in \mathcal{P}_{\rm exp}$ would be \emph{Ruelle's} measures. Nevertheless, we will stick to the version introduced above. More information on sub-Poissonian measures in the sense of Definition \ref{X3df} can be found in \cite[sect. 2.2]{KR}.

\subsection{Tempered configurations and cadlag paths}
The elements of $\mathcal{P}_{\rm exp}$  have one more significant property, which allows one to confine the theory to 
so-called \emph{tempered} configurations.
The set of such configurations $\Gamma_*$ is defined by the condition $\mu(\Gamma_*) =1$, which has to hold for all $\mu \in \mathcal{P}_{\rm exp}$. 
In the present work, this set is defined by means of the function
\begin{equation}
 \label{psi}
 \psi(x) = \frac{1}{1+ |x|^{d+1}}, \qquad x\in X.
\end{equation}
In a similar way, tempered configurations were introduced in \cite[page 41]{Dawson}, see also \cite[page 195]{Ether}.
Obviously,
\begin{equation}
 \label{psi1}
 0<\psi(x) \leq 1, \qquad {\rm and} 
 \qquad \int_X \psi(x) d x =: \langle \psi \rangle < \infty.
\end{equation}
Define
\begin{equation}
\label{Gas} 
 \Gamma_* = \{\gamma \in \Gamma : \varPhi(\gamma):=\gamma (\psi) = \sum_{x\in \gamma}  \psi(x) < \infty   \}.
\end{equation}
By \eqref{X14} we then have 
\begin{equation}
 \label{Gas1}
 \mu(\varPhi) = \int_X k^{(1)}_\mu (x) \psi(x)  d x \leq \varkappa_\mu \langle \psi \rangle < \infty,
\end{equation}
which yields $\mu(\Gamma_*) =1$. Next, we define
\begin{equation}
 \label{Gas2}
 \mathcal{A}_* = \{\mathbb{A}\in \mathcal{B}(\Gamma):\mathbb{A} \subset \Gamma_*\}.
\end{equation}
By \eqref{Gas} it follows that $\Gamma_*$ is the set of all those configurations $\gamma$ for which $\psi\gamma$ is a finite Borel measure on $X$. This fact allows one to define a kind of weak topology on $\Gamma_*$, which we do as follows. Set 
\begin{equation}
 \label{rho}
\rho (\gamma, \gamma') = \min\left\{1; \sup_{g\in C^1_{\rm L} }|\gamma (\psi g) - \gamma' (\psi g)|\right\},  \qquad \gamma, \gamma' \in \Gamma_*,
\end{equation}
where 
\begin{equation*}
C^1_{\rm L} := \left\{ g \in C_{\rm b} (X) : \sup_{x\in X} |g(x) | + \sup_{x, y\in X, \ x\neq y} \frac{|g(x) - g(y)|}{|x-y|} \leq 1\right\} .
\end{equation*}
It is clear that $\rho$ defined in \eqref{rho} is a metric on $\Gamma_*$. 
\begin{proposition}\cite[Lemma 2.7 and Corollary 2.8]{KR}
 \label{Gas1pn}
 The metric space $(\Gamma_*, \rho)$ is complete and separable. Its Borel $\sigma$-field and the collection of sets defined in \eqref{Gas2} coincide, i.e., $\mathcal{B}(\Gamma_*)= \mathcal{A}_*$. 
\end{proposition}
\begin{remark}
 \label{Gasrk}
 The space of tempered configurations defined in \eqref{psi} and \eqref{Gas} is exactly the same as that in \cite{KR}, where one can find more information on the properties of this space. Here we only mention that, in view of the equality $\mathcal{B}(\Gamma_*)= \mathcal{A}_*$, each $\mu\in \mathcal{P}(\Gamma)$ with the property $\mu(\Gamma_*)=1$ can be redefined as an element of $\mathcal{P}(\Gamma_*)$. 
\end{remark}
Similarly as in \cite{KR}, see also \cite[Sect. V]{Dawson} and \cite[Chapter 4]{EK}, we introduce the spaces of cadlag paths with values in $\Gamma_*$. For $s\geq 0$, we let $\mathfrak{D}_{[s,+\infty)} (\Gamma_*)$ define the set of all cadlag maps 
$[s,+\infty) \ni t \mapsto \gamma_t\in \Gamma_*$, where we mean the metric topology of $\Gamma_*$, see Proposition  \ref{Gas1pn}.
For $s=0$, we write $\mathfrak{D}_{\mathds{R}_{+}} (\Gamma_*)$. The elements of $\mathfrak{D}_{[s,+\infty)} (\Gamma_*)$ will be denoted $\bar{\gamma}$. For every $s' > s$, the restriction of a given $\bar{\gamma}\in \mathfrak{D}_{[s,+\infty)} (\Gamma_*)$ is usually considered an element of $\mathfrak{D}_{[s',+\infty)} (\Gamma_*)$. For $t\geq 0$, by $\varpi_t$ we denote the evaluation map, that is, $\varpi_t(\bar{\gamma})$ is the corresponding value $\gamma_t$ of $\bar{\gamma}$. For $t, t'\geq 0$, $t' > t$, by $\mathfrak{F}^0_{t,t'}$ we mean the $\sigma$-field of subsets of $\mathfrak{D}_{\mathds{R}_{+}} (\Gamma_*)$ generated by the collection of maps $\{\varpi_u: u\in [t,t']\}$. Next, set 
\begin{equation}
 \label{Cad}
\mathfrak{F}_{t,t'} = \bigcap_{\epsilon > 0}\mathfrak{F}^0_{t,t'+\epsilon}, \qquad \mathfrak{F}_{t,+\infty} = \bigcup_{n\in \mathds{N}} \mathfrak{F}_{t,t+n}.
\end{equation}
By \cite[Theorem 5.6, page 121]{EK} the Skorohod topology turns each $\mathfrak{D}_{[s,+\infty)} (\Gamma_*)$ into a Polish space, measurably isomorphic to the measurable space $(\mathfrak{D}_{[s,+\infty)} (\Gamma_*), \sigma(\mathfrak{F}_{s,+\infty}))$.

\subsection{Banach spaces of functions}
For a given $\mu\in \mathcal{P}_{\rm exp}$, its correlation
functions $k_\mu$, $k_\mu^{(n)}$, $n\in \mathds{N}_0$, are defined
in (\ref{X11}), (\ref{X11a}). Recall that the latter is a symmetric
element of $L^\infty(X^n)$ that satisfies the Ruelle estimate
(\ref{X13}). Having this in mind, we introduce Banach spaces
that contain $k_\mu$. As each $k: \Gamma_{\rm
fin} \to \mathds{R}$ is defined by its restrictions $k^{(n)}$ to the sets of
$\gamma =\{x_1 , \dots , x_n\}$, $n\in \mathds{N}_0$, cf.
(\ref{X11a}), we define
\begin{equation}
  \label{Y1}
\|k\|_\alpha = \sup_{n\in \mathds{N}_0} \|k^{(n)}\|_{L^\infty}
e^{-\alpha n}, \qquad \alpha \in \mathds{R},
\end{equation}
where
\[
\|k^{(n)}\|_{L^\infty} = {{\rm esssup}}_{(x_1, \dots , x_n)\in X^n}
|k^{(n)} x_1, \dots , x_n)|.
\]
Let $\mathcal{K}_\alpha$ be the real Banach space of $k:\Gamma_{\rm
fin}\to \mathds{R}$ for which $\|k\|_\alpha < \infty$. It is obvious
that
\begin{equation}
  \label{Y2}
\mathcal{K}_{\alpha'} \hookrightarrow \mathcal{K}_{\alpha} , \qquad
{\rm for} \ \ \alpha'< \alpha,
\end{equation}
where we mean continuous embedding. Let $k$ be in
$\mathcal{K}_\alpha$, $\alpha\in \mathds{R}$, and $G$ be in $B_{\rm
bs}$, see Definition \ref{X1df}. By (\ref{X10}) and (\ref{X2}) one
 easily obtains
\begin{equation*}
\langle\!\langle |k|,  |G| \rangle\!\rangle = \int_{\Gamma_{\rm
fin}} |k(\eta)| |G(\eta)| \lambda (d \eta) <
  \infty.
\end{equation*}
Note that, for each $\alpha\in \mathds{R}$, the measure $\chi = k \lambda$, $k\in \mathcal{K}_\alpha$ obviously satisfies \eqref{JUNNY}.  
Since $\mathcal{P}_{\rm exp}\subset \mathcal{P}_{\rm cor}(\Gamma)$, by Proposition \ref{JUNEpn} we have the following fact, cf. \cite[Theorems 6.1 and 6.2
and Remark 6.3]{Kuna}.
\begin{proposition}
  \label{Y1pn}
For each  $\alpha\in \mathds{R}$, the following is true. If $k\in
\mathcal{K}_\alpha$ is such that: (i) $k(\varnothing) = 1$; (ii)
$\langle\!\langle k,  G \rangle\!\rangle \geq 0$ for all $G\in
B^\star_{\rm bs}$, then $k$ is the correlation function for a unique
$\mu \in \mathcal{P}_{\rm exp}$ the type of which does not exceed
$e^\alpha$.
\end{proposition}
Let now $G:\Gamma_{\rm fin}\to \mathds{R}$ be such that each
$G^{(n)}$, $n\in \mathds{N}$, cf. (\ref{X10}), is a symmetric
element of $L^1(X^n)$. We denote its corresponding norm $\|G^{(n)}\|_{L^1}$ and
set
\begin{equation}
  \label{Y5}
 |G|_\alpha = |G(\varnothing)| + \sum_{n=1}^\infty \frac{1}{n!}
 e^{n\alpha} \|G^{(n)}\|_{L^1} = \int_{\Gamma_{\rm fin}} e^{\alpha
 |\eta|}|G(\eta)| \lambda ( d \eta),
\end{equation}
and also
\begin{equation}
  \label{Y6}
  \mathcal{G}_\alpha = \{G: |G|_\alpha < \infty\}, \qquad \alpha \in
  \mathds{R}.
\end{equation}
Thus, each  $\mathcal{G}_\alpha$ is a weighted $L^1$-type real
Banach space. Similarly as in (\ref{Y2}), we have
\begin{equation}
  \label{Y7}
  \mathcal{G}_{\alpha} \hookrightarrow  \mathcal{G}_{\alpha'},
  \qquad {\rm for} \ \ \alpha' < \alpha.
\end{equation}
However, here the embedding is also dense. Recall that the set of
functions $B_{\rm bs}$ is defined in Definition \ref{X1df}.
\begin{remark}
  \label{Y1rk}
Regarding the spaces $\mathcal{G}_\alpha$, $\alpha\in \mathds{R}$,
the following is true:
\begin{itemize}
  \item[(i)] For each $\alpha \in \mathds{R}$, $B_{\rm bs}$ is a
  dense subset of $\mathcal{G}_\alpha$.
  \item[(ii)] By (\ref{X6c}), (\ref{X14}) and (\ref{Y5}) one may write
\begin{equation}
  \label{Y8}
  |G|_\alpha = \pi_{e^\alpha} (K|G|),
\end{equation}
by which the $K$-map defined in (\ref{X4}) can be extended to
$\mathcal{G}_\alpha$ with an arbitrary $\alpha \in \mathds{R}$. In
this case, $K:\mathcal{G}_\alpha \to L^1(\Gamma, \pi_{e^\alpha})$.
\end{itemize}
\end{remark}

\section{The Results}

\label{Rsec}

\subsection{Solving the Fokker-Planck equation}
We begin by further defining our model.
The Kolmogorov operator
$L$ introduced in (\ref{L}) is subject to the following
\begin{assumption}
  \label{1ass}
The parameters of $L$  satisfy: (a) $a(0)>0$; $m$, $a$ and $b$ are nonnegative
and continuous; (b) the following quantities are finite
\begin{gather}
\label{6} 
\sup_{x\in X}\frac{a(x)}{\psi(x)} =:
  \|a\|_\psi, \quad 
  \sup_{x\in X}b(x) =:
  \|b\|, \quad \sup_{x\in X}m(x) =:
  \|m\|,
\end{gather}
where $\psi (x)$ is as in \eqref{psi}.
\end{assumption}
According to \eqref{6} $a$ is integrable, cf. \eqref{psi1}; hence, 
\begin{equation}
 \label{6a}
 \int_X a(x) d x =:\langle a \rangle \leq \|a\|_\psi \langle \psi \rangle < \infty.
\end{equation}
Moreover, by the triangle inequality and \eqref{psi} we have
\begin{gather}
 \label{6b}
 a(x-y) \leq \|a\|_\psi \psi(y) \left(1 + \sum_{l=0}^{d+1} { d+1 \choose l} |x|^l \right) =: \|a\|_\psi \psi(y) \ell_a (x).
\end{gather}
By this estimate, it  follows that for each $\theta \in C^{+}_{\rm cs} (X)$, the following holds, see \eqref{Gas} and \eqref{6b},
\begin{eqnarray}
 \label{6c}
\forall \gamma \in \Gamma_* \quad  \sum_{x\in  \gamma} \theta(x) \sum_{y\in \gamma\setminus x} a(x-y) \leq \gamma(\psi) \|a\|_\psi\sum_{x\in \gamma} \theta (x) \ell_a(x) < \infty , 
\end{eqnarray}
since the support of $\theta$ is compact. 

Below we use the following functions of $t\geq 0$ and $x\in X$
\begin{equation}
  \label{Y}
  q_t(x) =e^{-m(x) t}, \qquad \varrho_t(x) = \left\{
  \begin{array}{ll} \left( 1- e^{-m(x) t}\right) \frac{b(x)}{m(x)},
  \quad &{\rm if} \ \ m(x) >0; \\[.3cm]
b(x) t, \quad &{\rm if} \ \ m(x) =0.
  \end{array}\right.
\end{equation}
Let us turn to defining solutions of (\ref{FPE}), which we
precede by the following reminder; see e.g., \cite[pages 112,
113]{EK}. A subset $\mathcal{C}\subset C_{\rm b}(\Gamma)$ is said to
be separating if for each $\mu_1, \mu_2\in \mathcal{P}$, the
equality $\mu_1(F) = \mu_2(F)$, that holds for all $F\in \mathcal{C}$,
implies $\mu_1 = \mu_2$. If $\mathcal{C}$ is closed under
multiplication and separates points of $\Gamma$, it is separating.
The latter property means that, for each $\gamma_1 \neq \gamma_2$,
one finds $F\in \mathcal{C}$ such that $F(\gamma_1) \neq
F(\gamma_2)$. 
\begin{remark}
\label{OCTrk}
The
equality $\mu_1(F) = \mu_2(F)$, which holds for all $F\in \mathcal{C}$,
implies $\mu_1(H) = \mu_2(H)$ for all $H$ taken from the linear span of $\mathcal{C}$. 
That is why we do not require that $\mathcal{C}$ be an algebra, cf. 
\cite[Theorem 4.5, page 113]{EK}.
\end{remark}
\begin{definition}
  \label{1df}
Fix  separating $\mathcal{F}\subset C_{\rm b}(\Gamma)$ and $\mu_0\in \mathcal{P}(\Gamma)$. 
A map $\mathds{R}_{+} \ni t \mapsto \mu_t \in \mathcal{P}(\Gamma)$
is said to be a solution of the Fokker-Planck equation (\ref{FPE})
for $(L, \mathcal{F}, \mu_0)$ if, for each $F\in \mathcal{F}$, the following is true:
\begin{itemize}
  \item[(i)]  $LF$ is absolutely $\mu_t$-integrable for Lebesgue-almost all
  $t$, and the map $t \mapsto \mu_t(LF)$ is measurable and Lebesgue-integrable on each $[0,T]$, $T>0$.
  \item[(ii)] The equality in
  (\ref{FPE}) holds for all $t\geq 0$.
\end{itemize}
\end{definition}
Obviously, the domain $\mathcal{F}$ should be separating if one
strives for uniqueness of the solutions of (\ref{FPE}). Typically,
see, e.g., \cite[page 78]{Dawson}, generators $(L,\mathcal{F})$ are
chosen in such a way that $L:\mathcal{F}\to B_{\rm b}(\Gamma)$,
where the latter is the set of all bounded measurable functions
$F:\Gamma\to \mathds{R}$. In our case, however, it is barely
possible in view of the quadratic term in $L^{-}$. It might also be
clear from this definition that the proper choice of $\mathcal{F}$
is one of the main technical aspects of the current research.

For $\theta \in C_{\rm cs}^{+} (X)$, we set
\begin{equation}
  \label{X15}
  \varPhi^\theta (\gamma) = \sum_{x\in \gamma} \theta(x),
\end{equation}
and then
\begin{eqnarray}
  \label{X16}
  \varPhi^\theta_{\tau} (\gamma) = 
  \frac{\varPhi^\theta(\gamma)}{1+ \tau
  \varPhi^\theta(\gamma)} =  \int_0^{+\infty} \varPhi^\theta(\gamma)
  \exp\left(- \alpha\left[1+ \tau \varPhi^\theta (\gamma) \right] \right) d \alpha , \quad \tau \in(0,1/2].
\end{eqnarray}
Clearly  both $\varPhi^\theta$ and
$\varPhi^\theta_{\tau}$ are continuous, and also
\begin{equation*}
  0 \leq  \varPhi^\theta_{\tau} (\gamma)  \leq 1/\tau, \qquad \gamma\in \Gamma.
\end{equation*}
The obvious reason for introducing $\tau$ in \eqref{X16} is to pass from unbounded functions, as in \eqref{X15}, to bounded ones. Later, we prove in Lemma \ref{0lm} that we can take the limit $\tau \to 0$ and work directly with unbounded functions. The bound $\tau\leq 1/2$ -- needed to get an important property of the domain, see Lemma \ref{Y1lm} -- does not affect taking such limits. 

By $\vartheta$
we denote a finite multiset consisting of the elements of $C_{\rm cs}^{+}(X)$, see (\ref{X7}),
(\ref{X8}).  Define
\begin{equation}
  \label{M15}
 \varPsi^\vartheta_{\tau} (\gamma) = \prod_{\theta \in
 \vartheta} \varPhi^\theta_{\tau}(\gamma), \qquad \varPsi^\varnothing_{\tau} (\gamma) \equiv 1.
\end{equation}
Clearly, all $\varPsi^\vartheta_{\tau}$ are bounded and continuous. Thereafter, we set
\begin{equation}
  \label{X18a}
  \Theta =\{ \theta \in C_{\rm cs}^{+}(X):  \langle \theta \rangle \leq \langle \psi \rangle, \ \  \theta (x) \leq
  1 , \ x\in X\},
\end{equation}
see (\ref{X8}), \eqref{psi} and \eqref{psi1}, and
\begin{equation}
  \label{X18}
 \mathcal{F}_{\tau} =  \{\varPsi^\vartheta_{\tau}: {\rm all} \ {\rm possible} \ {\rm finite} \ \vartheta \subset
 \Theta\}, \qquad \mathcal{F}= \bigcup_{\tau\in  (0,1/2]} \mathcal{F}_{\tau}.
\end{equation}
Obviously, each $\mathcal{F}_{\tau}$ separates
points of $\Gamma$. For one can take $\theta\in \Theta$, who's support
contains some $x\in p(\gamma_1)$ and such that $\theta (y)
=0$ for all $y\in p(\gamma_2)$. Also, by the very definition
(\ref{M15}), each $\mathcal{F}_{\tau}$ is closed under
multiplication, and therefore separating; see Remark \ref{OCTrk} and \cite[Theorem 4.5,
page 113]{EK}. Thus, so is $\mathcal{F}$. The choice of the upper bound in \eqref{X18a} will be explained later. Here we just recall that $\psi(x)\leq 1$, and hence $\Theta$ is closed under multiplication.

Now we recall that the $K$-map and the spaces $\mathcal{G}_\alpha$ are defined in \eqref{X4} and \eqref{Y6}, respectively.
To proceed further, we introduce the following set of functions
\begin{equation}
 \label{NGa}
 \mathcal{F}_{\rm max} = \{F = KG: G\in \mathcal{G}_\alpha \ {\rm for} \ {\rm all} \  \alpha \in \mathds{R} \}, 
\end{equation}
see Remark \ref{Y1rk}. Clearly, $\Phi^\theta \in \mathcal{F}_{\rm max}$, see \eqref{X15}; hence, $\mathcal{F}_{\rm max}$  contains also unbounded functions. At the same time,
$\mathcal{F} \subset \mathcal{F}_{\rm max}$, see Lemma \ref{Y1lm} below. In view of these two facts, $\mathcal{F}_{\rm max}$ will serve as the maximal domain to which we extend $L$  from $\mathcal{F}$. 

For a compact $\Lambda \subset X$ and $\gamma\in \Gamma$, define
\begin{equation}
  \label{NG}
  N_\Lambda (\gamma) = \sum_{x\in \gamma} \mathds{1}_\Lambda (x)
  =\gamma(\Lambda),
\end{equation}
which is the total number of the elements of $\gamma$ contained
in $\Lambda$, cf. \eqref{Lambd}. 
Now we can
formulate our first statement. Recall that $F^\theta$ is defined in \eqref{X6b}.
\begin{theorem}
  \label{1tm}
  Let the parameters of the Kolmogorov operator $L$ introduced in
  (\ref{L}) satisfy Assumption \ref{1ass} and $\mathcal{F}$ be as
  in (\ref{X18}). Then, for each $\mu_0\in \mathcal{P}_{\rm exp}$, the Fokker-Planck
  equation for $(L, \mathcal{F} , \mu_0)$ has a unique solution
  such that $\mu_t \in \mathcal{P}_{\rm exp}$ for all $t\geq0$. This
  solution has the following properties:
 \begin{itemize}
  \item[(a)] For each $\theta \in C_{\rm cs}^{+}(X)$, 
  \begin{equation}
  \label{Y0}
\mu_t (F^\theta ) \leq \pi_{\varrho_t}  (F^\theta ) \mu_0^{q_t}
(F^\theta ), \qquad t >0,
\end{equation}
$\varrho_t$ and
$q_t$ are as in (\ref{Y}) and $\mu_0^{q_t}$ is the
$q_t$-thinning of $\mu_0$, see Definition \ref{Thdf}.
  \item[(b)] For each compact $\Lambda\subset X$ and $n\in \mathds{N}$, there
exists $C_{n,\Lambda}>0$ such that
\begin{equation}
  \label{NG1}
 \forall t > 0 \qquad  \mu_t (N^n_\Lambda) \leq C_{n,\Lambda},
\end{equation}
i.e., the moments of the observable (\ref{NG}) are globally bounded
in time.
\item[(c)] The sets defined in \eqref{X18} and \eqref{NGa} satisfy $\mathcal{F}\subset \mathcal{F}_{\rm max}$, and $\mu_t$ solves the Fokker-Planck equation \eqref{FPE} with each $F   \in  \mathcal{F}_{\rm max}$.
 \end{itemize}
 \end{theorem}
Let us make some preliminary comments on this statement. A priori, we look for solutions among all $\mu
\in \mathcal{P}(\Gamma)$ satisfying item (i) of Definition \ref{1df} -- restricting only the choice of the initial state $\mu_0$. The result is that the solution is unique and lies in $\mathcal{P}_{\rm  exp}$; i.e., 
 the evolution leaves invariant the set of sub-Poissonian measures.   
Note that item  (i) of Definition \ref{1df} means
\begin{equation}
 \label{0ja}
\forall F\in \mathcal{F} \quad \forall T>0 \qquad  \int_0^T \mu_t(|LF|) d t < \infty.
 \end{equation}
The measure on the right-hand side of \eqref{Y0} is the convolution, see \cite[page 15]{Li}, which corresponds to the mixture of two states: (a) the Poisson state $\pi_{\varrho_t}$ describing the distribution of the newcomers; (b) the thinned initial state. The inequality in \eqref{Y0} indicates that the states of the aforementioned mixture dominate those of the population considered. It shows that 
the sign of the quadratic term in $L^{-}$ was taken into account
properly.  
Finally, the boundedness as in \eqref{NG1}, which holds also if $m(x)\equiv 0$, is of the same nature as the boundedness of $N_t$ in the Verhulst model. As  already mentioned,     
the set $\mathcal{F}_{\rm max}$ contains also unbounded functions which are $\mu_t$-integrable for all $t>0$, see  \eqref{X12}. Below, we use the following important corollary of the theorem stated above.   
\begin{remark}
\label{OCTrk2}
Let $\widetilde{\mathcal{F}}\subset \mathcal{F}_{\rm max}\cap B_{\rm b}(\Gamma)$ be separating. Let also $t\mapsto \tilde{\mu}_t$ be a solution of the Fokker-Planck equation for $(L,\widetilde{\mathcal{F}},\mu_0)$, $\mu_0 \in \mathcal{P}_{\rm exp}$. Then $\tilde{\mu}_t$ coincides with the unique solution of \eqref{FPE} for $(L,{\mathcal{F}},\mu_0)$, the existence of which is stated in Theorem \ref{1tm}.   
\end{remark}
Indeed, by item (c) of Theorem \ref{1tm}, there exists only one $\mu_t$ which solves \eqref{FPE} for all $F \in \mathcal{F}_{\rm max}$. Hence, $\mu_t$ and $\tilde{\mu}_t$ agree on $\widetilde{\mathcal{F}}$ and therefore coincide as the latter set is separating.  

\subsection{The Markov process}

Theorem \ref{1tm} yields the evolution of states $\mu_0 \to \mu_t$ of the model corresponding to the Kolmogorov operator \eqref{L}. A more comprehensive description of the stochastic evolution of this model can be obtained by constructing a Markov process. In this work, we follow the approach elaborated in \cite{KR,KR1} in which the process in question is obtained by solving a restricted martingale problem. The central role here is played by the Fokker-Planck equation, especially stated in Theorem \ref{1tm} uniqueness, the facts that its solutions lie in the class of sub-Poissonian measures, and that the unique solution satisfies \eqref{FPE} with all $F\in \mathcal{F}_{\rm max}$.

As just mentioned, the evolution $\mu_0 \to \mu_t$ related to \eqref{FPE} leaves invariant the set of measures $\mathcal{P}_{\rm exp}$ the elements of which have the property $\mu(\Gamma_*)=1$, see Remark \ref{Gasrk}. Therefore, it might be natural to construct a process with values in $\Gamma_*$, cf. \eqref{Cad}, such that its 
one-dimensional marginals solve the Fokker-Planck equation. Such a process will be obtained as a solution of the restricted martingale problem involving $L$, the domain of which should be consistent with the domain used in the Fokker-Planck equation. Thus, we start by introducing the following functions
\begin{equation}
 \label{G1}
\widetilde{F}_v (\gamma) = \exp \left( - \sum_{x\in \gamma} v(x) \psi(x) \right), \qquad v\in C_{\rm b}^{+} (X). 
\end{equation}
Concerning them, the following is known; see \cite[Lemma 3.2.5 and Theorem 3.2.6, page 43]{Dawson}.
\begin{proposition}
 \label{G1pn}
There exists a countable set $\mathcal{V}\subset C_{\rm b}^{+}(X)$, which contains constants and is closed under addition, such that the set $\widetilde{\mathcal{F}} =\{ \widetilde{F}_v: v\in \mathcal{V} \}$ has the following properties:  
\begin{itemize}
 \item[(i)] The space of $\mathcal{B}(\Gamma_*)$-measurable functions is the bounded point-wise closure of $\widetilde{\mathcal{F}}$, see \cite[page 41]{Dawson}, and $\mathcal{B}(\Gamma_*)$ is generated by  $\widetilde{\mathcal{F}}$.
 \item[(ii)] $\widetilde{\mathcal{F}}$ is a strongly separating, and hence separating, subset of $C_{\rm b}(\Gamma_*)$. The former implies that this set is weak convergence determining. That is, for $\{\mu_n\}\subset \mathcal{P}(\Gamma_*)$ and $\mu \in\ \mathcal{P}(\Gamma_*)$, it follows that $\mu_n \Rightarrow \mu$ as $n\to \infty$ if and only if
 $\mu_n(\widetilde{F}_v) \to \mu(\widetilde{F}_v)$ for all $\widetilde{F}_v\in \widetilde{\mathcal{F}}$.  
\end{itemize}
\end{proposition}
It is clear that $C^{-1}_{v,\mu}:=\mu (\widetilde{F}_v)>0$ for each $v\in\mathcal{V}$ and $\mu \in\mathcal{P}(\Gamma_*)$.  
\begin{proposition}
 \label{G2pn}
 Along with the properties just mentioned, the set $\widetilde{\mathcal{F}}$ has the following ones: (a) for each $\mu\in \mathcal{P}_{\rm exp}$ and $v\in\mathcal{V}$, the measure $\mu_v = C_{v,\mu} \widetilde{F}_v \mu$ is in $\mathcal{P}_{\rm exp}$, and the types of these measures, see Definition \ref{X3df}, satisfy $\varkappa_{\mu_v} = \max\{C_{v,\mu} \varkappa_\mu; \varkappa_\mu\}$; (b) $\widetilde{\mathcal{F}}\subset \mathcal{F}_{\rm max}$, where the latter set is defined in \eqref{NGa}; (c) for each $\mu\in \mathcal{P}_{\rm exp}$ and $F\in \widetilde{\mathcal{F}}$, it follows that $\mu(|LF|)< \infty$.
\end{proposition}
\begin{proof}
 The proof of claim (a) readily follows by \eqref{X8}. To prove (b), we have to show that $\widetilde{F}_v = K G_v$ with $G_v \in \mathcal{G}_\alpha$ for all $\alpha \in \mathds{R}$. To this end, we use the function $\xi\mapsto e(\xi;h)$ introduced in \eqref{Y0a}, which is measurable and bounded for every $h\in B_{\rm b}(X)$.  
 Write, cf. \eqref{X4},
 \begin{gather}
  \label{G2}
  \widetilde{F}_v (\gamma) = \prod_{x\in \gamma} (1+ h_v(x)) = \sum_{\xi \Subset \gamma} e(\xi;h_v) = (K e(\cdot;h_v)) (\gamma), \\[.2cm] \nonumber h_v (x) := e^{-v(x) \psi(x)}-1.
 \end{gather}
Then, cf. \eqref{Y8}, 
\begin{gather}
 \label{G3}
 |e(\cdot;h_v)|_\alpha= \int_{\Gamma_{\rm fin}} e^{\alpha |\xi|} \prod_{x\in \xi} \left(1-e^{-v(x) \psi(x)}\right) \lambda (d\xi)\\[.2cm] \nonumber \leq \int_{\Gamma_{\rm fin}} e^{\alpha |\xi|} \left(\prod_{x\in \xi} v(x)\psi(x) \right)  \lambda (d\xi)\\[.2cm]  \leq \exp\left(e^\alpha \|v\|\langle \psi \rangle \right), \qquad \|v\|= \sup_{x\in X} v(x),\nonumber
\end{gather}
which completes the proof of item (b). Next, by \eqref{G1}  one gets
\begin{eqnarray}
 \label{Agata}
 \widetilde{F}_v (\gamma\cup x) -  \widetilde{F}_v (\gamma) & = & - \widetilde{F}_v (\gamma)\left[1- e^{-v(x) \psi(x) } \right], \\[.2cm] \nonumber
  |\widetilde{F}_v (\gamma\cup x) -  \widetilde{F}_v (\gamma)| & \leq & v(x) \psi(x),
\end{eqnarray}
cf. \eqref{G3}.
Then by \eqref{L} it follows that
\begin{equation}
 \label{Agata1}
 |L^{+} \widetilde{F}_v (\gamma)| \leq \int_X b(x) v(x) \psi(x) dx \leq \|b\| \|v\| \langle \psi \rangle, 
\end{equation}
see \eqref{psi1}. In dealing with $L^{-}\widetilde{F}_v$, we first take $\widetilde{F}_{v_n}$ with  $v_n = v \mathds{1}_{\Delta_n}$, $\Delta_n:=\{x\in X: |x|\leq n\}$. In this case, by \eqref{Agata} and \eqref{6c} we get
\begin{gather}
 \label{Agata2}
 |L^{-} \widetilde{F}_{v_n} (\gamma)| \leq \|m\| \|v\|  \sum_{x\in \gamma} \psi(x) + \sum_{x\in \gamma} v_n(x) \psi(x) \sum_{y\in \gamma \setminus x} a(x-y)\qquad \\[.2cm] \nonumber \leq \|m\| \|v\| \gamma(\psi) + \gamma(\psi) \|a\|_\psi \sum_{x\in \gamma} v_n (x) \ell_a (x) < \infty.
\end{gather}
At the same time, by \eqref{X12}, \eqref{X13}, the first line in \eqref{Agata2} and \eqref{6a} it follows that
\begin{eqnarray}
 \label{Agata3}
 \mu (|L^{-} \widetilde{F}_{v_n}|)& \leq & \|m\| \|v\| \int_X k^{(1)}_\mu (x) dx + \int_{X^2} k_\mu^{(2)} (x,y) v_n(x) \psi(x) a(x-y) d x dy \\[.2cm] \nonumber & \leq & \|v\|\varkappa_\mu \langle \psi \rangle (\|m\| + \varkappa_\mu \langle a \rangle).
\end{eqnarray}
By the monotone convergence theorem, one then gets that 
\[
 \mu (|L^{-} \widetilde{F}_{v}|) \leq {\rm RHS}\eqref{Agata3}, 
\]
which together with the estimate in \eqref{Agata1} yields the proof of claim (c).
\end{proof}
\begin{remark}
 \label{Agatark}
 By \eqref{Agata3}, the following extension of claim (c) of Proposition \ref{G2pn} can be obtained. For a subset, $\mathcal{P}\subset \mathcal{P}_{\rm exp}$, assume that $\sup_{\mu \in \mathcal{P}} \varkappa_\mu =:\varkappa < \infty$. Then 
 \[
  \sup_{\mu \in {\mathcal{P}}} \mu (|L\widetilde{F}_{v}|) \leq \|v\|\langle \psi \rangle (\|b\| +\|m\|{\varkappa} + \langle a \rangle {\varkappa}^2).
 \]
 \end{remark}
Fix now some $s\geq 0$, $t_1, t_2$, $t_2 > t_1 \geq s$, $\widetilde{F}_{v}$ as in \eqref{G1} and consider
\begin{equation}
 \label{Agata4}
 {\sf Q}^{\pm} (\bar{\gamma})=  \int_{t_1}^{t_2} (L^{\pm}\widetilde{F}_{v})(\varpi_{u} (\bar{\gamma}))   du, \qquad {\sf Q} = {\sf Q}^{+} + {\sf Q}^{-} .
\end{equation}
Clearly, ${\sf Q}$ is $\sigma(\mathfrak{F}_{[t_1,+\infty)})$-measurable. 
By the first line of \eqref{Agata} and \eqref{Agata1},  ${\sf Q}^{+}$ is bounded. However, $-{\sf Q}^{-}$ may increase ad infinitum.
\begin{proposition}
 \label{Agatapn}
 Let $P\in \mathcal{P}(\mathfrak{D}_{[0,+\infty)}(\Gamma_*))$ be such that  $P \circ \varpi^{-1}_u \in \mathcal{P}_{\rm exp}$ for all $u \geq 0$. Moreover, for each $t>0$, assume that $\sup_{u\in [0,t]} \varkappa_u =:\varkappa<\infty$, where $\varkappa_u$ is the type of $P \circ \varpi^{-1}_u$, see Definition \ref{X3df}.  Then $P (|{\sf Q}|)< \infty$. 
\end{proposition}
\begin{proof}
 As mentioned above,  ${\sf Q}^{+}$ is bounded; hence, it remains to prove that $P(|{\sf Q}^{-}|)< \infty$. First, as in \eqref{Agata2} we take 
 \begin{equation*}
  {\sf Q}^{-}_{n} (\bar{\gamma})=  \int_{t_1}^{t_2} (L^{-}\widetilde{F}_{v_n})( \varpi_{u} (\bar{\gamma}))   du ,\qquad n\in \mathds{N}.
 \end{equation*}
Note that $|{\sf Q}^{-}_{n} (\bar{\gamma})| = - {\sf Q}^{-}_{n} (\bar{\gamma})$, see \eqref{Agata}, and   ${\sf Q}^{-}_{n}: \mathfrak{D}_{[s,+\infty)}(\Gamma_*) \to \mathds{R}$, see \eqref{Agata2}. By the assumption of this statement, it follows that 
$P\circ \varpi^{-1}_u =:\mu_u \in \mathcal{P}_{\rm exp}$; hence,
\begin{gather}
 \label{Agata6}
 \int_{t_1}^{t_2} P (|(L^{-}\widetilde{F}_{v_n})\circ \varpi_{u}|) du = \int_{t_1}^{t_2} (P \circ \varpi^{-1}_u )( |(L^{-}\widetilde{F}_{v_n})|) du \\[.2cm] \nonumber = \int_{t_1}^{t_2} \mu_u (|(L^{-}\widetilde{F}_{v_n})|) du \leq (t_2-t_1)\|v\|\varkappa \langle \psi \rangle (\|m\| + \varkappa \langle a \rangle ),
\end{gather}
see \eqref{Agata3} and Remark \ref{Agatark}. By the Tonelli and Fubini theorems, see \cite[Theorems 4.4 and 4.5, page 91]{Brezis}, it then follows 
\[
 P (| {\sf Q}^{-}_{n}|) \leq {\rm RHS}\eqref{Agata6},
\]
which, by the monotone convergence theorem, yields the proof. 
\end{proof}
For the same $s$, $t_1$ and $t_2$ as in \eqref{Agata4}, let ${\sf J}: \mathfrak{D}_{[s,+\infty)}\to \mathds{R}$ be bounded and $\mathfrak{F}_{s,t_1}$-measurable. Then for $F\in \widetilde{\mathcal{F}}$, we define
\begin{equation}
 \label{Cad1}
 {\sf H} (\bar{\gamma}) = \left[F(\varpi_{t_2} (\bar{\gamma})) - F(\varpi_{t_1} (\bar{\gamma})) - \int_{t_1}^{t_2} (LF)(\varpi_{u} (\bar{\gamma}))   du \right]{\sf J}(\bar{\gamma}). 
\end{equation}
The next definition is an adaptation of the corresponding definition in \cite[Sect. 5.1, pages 78, 79]{Dawson}, see also \cite[Definition 3.3]{KR}. Here, for $s\geq 0$ and $\mu\in \mathcal{P}_{\rm exp}$, we deal with probability measures on $\mathfrak{D}_{[s,+\infty)}(\Gamma_*)$, cf. Proposition \ref{Agatapn}. 
\begin{definition}
 \label{Cad1df}
 A family of probability measures $\{P_{s,\mu}: s\geq 0, \mu\in \mathcal{P}_{\rm exp}\}$ is said to be a solution of the restricted martingale problem if, for all $s\geq 0$ and $\mu\in \mathcal{P}_{\rm exp}$, the following is true: (a) $P_{s,\mu} \circ \varpi^{-1}_s = \mu$; (b)  $P_{s,\mu} \circ \varpi^{-1}_u \in \mathcal{P}_{\rm exp}$ for all $u >s$; (c) for all $t>s$, the types $\varkappa_u$  of $P_{s,\mu} \circ \varpi^{-1}_u$ satisfy $\sup_{u\in [s, t]} \varkappa_u < \infty$; (d) $P_{s,\mu} ({\sf H}) =0$, which holds for $\sf H$ as in \eqref{Cad1} with each $F\in \widetilde{\mathcal{F}}$ and every bounded function  ${\sf J}: \mathfrak{D}_{[s,+\infty)}(\Gamma_*) \to \mathds{R}$, which is $\mathfrak{F}_{s,t_1}$-measurable, $t_1>s$, see \eqref{Cad}. The restricted martingale problem is said to be well-posed if it has a unique solution in the following sense:  if $\{P_{s,\mu}: s\geq 0, \mu\in \mathcal{P}_{\rm exp}\}$ and $\{P'_{s,\mu}: s\geq 0, \mu\in \mathcal{P}_{\rm exp}\}$ solve the problem, then all finite dimensional marginals of $P_{s,\mu}$ and $P'_{s,\mu}$ coincide for all $s$ and $\mu$, see \cite[page 182]{EK}.  
\end{definition}
\begin{remark}
 \label{Finrk}
Concerning the notions introduced in Definition \ref{Cad1df}, one should remark the following:
\begin{itemize}
\item[(a)] By Proposition \ref{Agatapn}, $\sf H$ as given in \eqref{Cad1} is absolutely $P_{s, \mu}$-integrable for each $s$ and $\mu\in \mathcal{P}_{\rm exp}$. 
 \item[(b)]  
  The map $[0, +\infty) \ni t \mapsto P_{0,\mu_0} \circ \varpi^{-1}_t$, $t>0$, solves the Fokker-Planck equation for $(L,\mathcal{F}, \mu_0)$. Indeed,  the map in question solves \eqref{FPE} with each $F\in \widetilde{\mathcal{F}}$, which can be obtained by taking ${\sf J} \equiv 1$ and interchanging integrations as in the proof of Proposition \ref{Agatapn}. Then it coincides with the solution described in Theorem \ref{1tm}, see Remark \ref{OCTrk2}.  
  \item[(c)] The particular form of the domain $\mathcal{F}$  (resp. $\widetilde{\mathcal{F}}$)  has been tailored to the needs of proving Theorem \ref{1tm} (resp. Theorem  \ref{2tm}). At the same time, the set $\mathcal{F}_{\rm max}$ is a common `roof' for the two domains, see Remark \ref{OCTrk}.
    \item[(d)] The adjective ``restricted'' points to condition (b) of Definition \ref{Cad1df} which forces the 
    one-dimensional marginals to be in $\mathcal{P}_{\rm exp}$.  
\item[(e)]   The function $\sf J$ in \eqref{Cad1} can be taken in the form
\begin{equation}
 \label{Cad3}
 {\sf J} (\bar{\gamma}) = J_1 ( \varpi_{s_1}(\bar{\gamma}) ) \cdots J_m ( \varpi_{s_m}(\bar{\gamma})),
\end{equation}
with all possible choices of $m\in \mathds{N}$ and $J_1, \dots , J_m \in \widetilde{\mathcal{F}}$, see eq. (3.4) on page 174 of \cite{EK} and the proof of Theorem 5.1.3 on page 81 of \cite{Dawson}.
\end{itemize}
\end{remark}
\begin{theorem}
 \label{2tm}
 Let $L$ and $\widetilde{\mathcal{F}}$ be as \eqref{L} and Proposition \ref{G1pn}, respectively. Then
 \begin{itemize}
  \item[(a)] the restricted martingale problem has a unique solution in the sense of Definition \ref{Cad1df}; 
  \item[(b)] the stochastic process related to the family
  \[
   \{\mathfrak{D}_{[s,+\infty)}(\Gamma_*), \mathfrak{F}_{s,+\infty}, \{\mathfrak{F}_{s,t}:t\geq s\}, \{P_{s,\mu}: s\geq 0, \mu\in \mathcal{P}_{\rm exp}\}: s\geq 0 \}
  \]
is Markov.
\end{itemize}
\end{theorem}

\subsection{The scheme of the proof of both theorems and comments}

\subsubsection{Concerning Theorem \ref{1tm}}

As mentioned above, the existence of a map $t\mapsto \mu_t \in \mathcal{P}_{\rm exp}$, which describes the evolution of states of the model with the generator \eqref{L}, was proved in \cite{KK}. It was done by constructing the evolution of the corresponding correlation functions; see Proposition \ref{X6pn} below. At the same time, it has remained unclear whether this map describes the evolution in question in a unique way -- obviously, due to the lack of a `canonical' way of constructing solutions, e.g., by means of a $C_0$-semigroup. Moreover, the domination expressed in \eqref{Y0} points to the possibility of constructing the evolution $t\mapsto \mu_t$ by standard methods: first, for finite systems by approximating $b(x)$ to secure $\int b(x) dx < \infty$; then, taking the corresponding limit. However, this way would not provide tools to prove uniqueness, nor to control the asymptotic properties of the distributions $\{p_{\Lambda,\mu_t}(n)\}$.

Our approach is based on the expectation that the Fokker-Planck equation \eqref{FPE} -- being a weaker version of the evolution equation -- would be more accessible by existing methods. In this case, however, the weakness just mentioned imposes the necessity of
\emph{specifying} its solutions, which we do by means of the triple $(L, \mathcal{F}, \mu_0)$, see  Definition  \ref{1df}. In view of this, the choice of the domain $\mathcal{F}$ becomes important. 
As a benefit, one can raise and ask the question of the uniqueness of solutions, understood as the coincidence of any two of them specified by the same triple. In Theorem \ref{1tm}, we state that the solution constructed in \cite{KK} is the unique solution of the Fokker-Planck equation \eqref{FPE} corresponding to the triple $(L, \mathcal{F}, \mu_0)$ with $\mu_0\in \mathcal{P}_{\rm exp}$. This is done according to the following scheme.           
\begin{itemize}
 \item To be able to control the sign of the quadratic term in $L^{-}$, we construct functions $F:\Gamma\to \mathds{R}$, see \eqref{X18}, in such a way that $\pm L^{\pm} F \geq 0$, which is done in subsect. 4.2, see \eqref{0d}. This allows one to control $\mu_t(F)$ for possible solutions $\mu_t$, see \eqref{0i}, \eqref{0j}, and thereby extend the domain of $L$ to a special class of unbounded functions, see Lemma \ref{0lm}. By means of this extension, we then prove, see Lemma \ref{01lm}, that each solution lies in $\mathcal{P}_{\rm exp}$, which is a key point of the whole story. 
 \item In Lemma \ref{Y1lm}, we prove inclusion $\mathcal{F}\subset \mathcal{F}_{\rm max}$, see claim (c) of Theorem \ref{1tm}. By \eqref{X12} this result allows us to pass on to the correlation functions and thus to prove (see Section 5) that the evolution $t \mapsto \mu_t\in \mathcal{P}_{\rm exp}$ constructed in \cite{KK} solves the Fokker-Planck equation for any $F\in \mathcal{F}_{\rm max}$ whenever $\mu_0 \in  \mathcal{P}_{\rm exp}$. The proof of uniqueness is mainly based on the fact proved in Lemma \ref{01lm}.   
\end{itemize}

\subsubsection{Concerning Theorem \ref{2tm}}

In constructing the Markov process that describes the stochastic evolution of the model corresponding to $L$ we mostly follow the way elaborated in our works \cite{KR,KR1}. It is based on solving the restricted martingale problem, 
which inherently is related to the Fokker-Planck equation. To see this, it suffices to compare \eqref{FPE} and \eqref{Cad1}. In particular, uniqueness in this case is a consequence of the same property as proved in Theorem \ref{1tm}. Note that, similarly to \cite{Ether}, the very construction of the process corresponding to \eqref{L} could be done more conventionally with the help of the dominance expressed in \eqref{Y0}. Our scheme, based on the Fokker-Planck equation, provides the means to obtain the uniqueness property that is lacking in \cite{Ether}. 

The proof of Theorem \ref{2tm} consists of the following steps:
\begin{itemize}
 \item We modify the model described by $L$ given in \eqref{L} by multiplying the model parameters by a certain function $\psi_\sigma$, see \eqref{z2}, \eqref{z3}. As a result, we obtain the family of generators $\{L_\sigma: \sigma \in (0,1]\}$ corresponding to what we call `auxiliary models'. For these models, all the results of \cite{KK} and the first part of this work are valid. At the same time, the mentioned modification allows one to construct Markov transition functions $p^\sigma_t$, see subsection 6.2, which is impossible for the initial model. By standard methods, one then gets, see  
\eqref{EKP}, the finite-dimensional laws of the Markov processes corresponding to the auxiliary models.  

\item As mentioned above, by means of the methods used in the first part, we construct the evolution $t \mapsto \mu^\sigma_t\in \mathcal{P}_{\rm exp}$. At the same time, for $\sigma \in (0,1]$, the transition functions $p^\sigma_t$ define the evolution $t\mapsto \hat{\mu}^\sigma_t$ by the formula $\hat{\mu}^\sigma_t (\cdot ) = \int  p^\sigma_t (\gamma, \cdot) \mu_0(d\gamma)$, cf. \eqref{CP1}. In Lemma \ref{CPlm}, we show that $\hat{\mu}^\sigma_t= {\mu}^\sigma_t$, which is one of the crucial points of our construction. 
Then in Lemma \ref{z2lm}, we prove that ${\mu}^\sigma_t \Rightarrow {\mu}_t$ as $\sigma \to 0$, where the latter is the evolution constructed in Theorem \ref{1tm}. This fact and the Chentsov-like estimates obtained in Lemma \ref{u1lm} yield the following: (a) the Markov processes corresponding to $L_\sigma$, $\sigma \in (0,1]$, have cadlag paths; (b) these processes weakly converge as $\sigma \to 0$ to a unique process with cadlag paths, which is the process in question.  
\end{itemize}

\subsubsection{Connection to the broader literature}

As mentioned above, an individual-based model of a large population with a logistic-type interaction term was first introduced in \cite{BP}, see also \cite{BP1,Law,MDL}. In \cite{Ether,Four}, populations modeled in this way were called \emph{locally regulated}. 
The principal results of these two works were based on the construction of Markov processes for locally regulated finite populations. In \cite{Ether}, such a process was also constructed for an infinite population if the initial state is a tempered configuration. However, studying the uniqueness issues was beyond the capabilities of the method used in that work.

As it was aptly noted in \cite{Ether}, the density-dependent (i.e., logistic) regulation term in the Bolker-Pacala model was introduced to overcome the clumping in the population. However, it was not specified what such clumping would consist of and how it could be detected mathematically. In \cite{KK2}, it was suggested to connect clustering (i.e., clumping) with the appearance of  `heavy tails' of the distributions $p_{\Lambda,\mu}(n) = \mu(\Gamma_\Lambda^{(n)})$, $\Lambda \subset X$, see \eqref{Lambd}, which, of course, assumes their study. By virtue of Definition \ref{X3df}, sub-Poissonian measures do not have heavy tails, and the fact that the evolution preserves the set $\mathcal{P}_{\rm exp}$ was proposed to be considered as the lack of clustering in the corresponding population. In \cite{KK1} (see also \cite{Dima} for a preliminary version), the (regulated) evolution $t \mapsto \mu_t \in \mathcal{P}_{\rm exp}$ was constructed for the Bolker-Pacala model. Similarly to \cite{KK}, the evolution $t\mapsto k_t$ was shown to take place in the scale of Banach spaces $\{\mathcal{K}_\alpha\}_{\alpha \in \mathds{R}}$; see \eqref{Y2}. Next, it was proved that each $k_t$ is the correlation function of a unique state $\mu_t \in \mathcal{P}_{\rm exp}$. The main steps of this construction for the model considered here are presented in Sect. 5 and the Appendix below. The key aspect of the method of \cite{KK1} was demonstrating a kind of suppression of the dispersal kernel $a^{+}$ by the competition kernel $a$. Note that such an effect is impossible in the contact model, cf. \cite{K}, where the death part of the generator contains only an intrinsic mortality term. We believe that a modification of the technique developed in this work, combined with that of \cite{KK1,K}, will allow us to prove an analog of the present Theorem \ref{2tm} also for the Bolker-Pacala model.

\section{Properties of Possible Solutions of the Fokker-Planck Equation}

In this section, we obtain a number of a priori properties of the solutions of \eqref{FPE}, which will allow us to locate these solutions and thus to elaborate tools for proving Theorem \ref{1tm}.

\subsection{A property of the domain}

As mentioned above, the main technical aspect of our approach consists in dealing with correlation functions of the states in question rather than with the states themselves. In view of this, we crucially use their relationship expressed in \eqref{X12}, which is based on the possibility to present the elements of $\mathcal{F}$ in the form $F=KG$; see \eqref{X4}. Recall that the $K$-map  can be extended to $\mathcal{G}_\alpha$, see Remark
\ref{Y1rk}. Below by $K$ we understand this extension.
\begin{lemma}
  \label{Y1lm}
For every $\vartheta \subset \Theta$, see \eqref{X18a}, and $\tau\in (0,1/2]$, there  exists a unique
$G^\vartheta_{\tau}$, which belongs  to $\mathcal{G}_\alpha$ with an arbitrary $\alpha\in \mathds{R}$, such that
$\varPsi^\vartheta_{\tau} = K
G^\vartheta_{\tau}$. In other words, $\mathcal{F}\subset \mathcal{F}_{\rm max}$.
\end{lemma}  
\begin{proof}
Fix some $\vartheta \subset \Theta$, and write $\vartheta=\{ \theta_1,
\dots , \theta_n\}$, $n\in \mathds{N}$. By (\ref{X16}) and
(\ref{M15}) we have
\begin{gather}
  \label{X25}
\varPsi^\vartheta_{\tau} (\gamma)  =  \int_0^{+\infty}
\cdots \int_0^{+\infty} e^{-(\beta_1 + \cdots +\beta_n)} \varPsi^\vartheta
(\gamma) \exp\left( - \tau\sum_{x\in \gamma} \sum_{j=1}^n \beta_j
\theta_j(x)\right) d \beta_1 \cdots d \beta_n,
\end{gather}
where
\begin{gather}
\label{X25a} \varPsi^\vartheta (\gamma) = \prod_{\theta \in \vartheta}
\varPhi^\theta (\gamma) = \left(\sum_{x\in \gamma} \theta_1(x) \right)
\cdots \left(\sum_{x\in \gamma} \theta_n(x) \right).
\end{gather}
For a positive integer $l\leq n$, let $d$ be a division of the set $\{1, \dots , n\}$ into $l$
unlabeled subsets $\delta \neq \varnothing$.   
For convenience, we label them in such a way that 
$|\delta_j| \geq |\delta_k|$ for $j<k$. Let also
$\mathfrak{d}_l$ be the collection of all such divisions. Its
cardinality is $S(n,l)$ -- Stirling's number of the second kind. For
$d\in \mathfrak{d}_l$, we set
\begin{equation}
  \label{X26}
  \widehat{\theta}_k (x) = \prod_{i\in \delta_k} \theta_i (x),
  \qquad k = 1 , \dots , l.
\end{equation}
Then by (\ref{X7}), (\ref{X6}) one gets the following relation, inverse to that in \eqref{X2},
\begin{eqnarray}
  \label{X27}
\varPsi^\vartheta (\gamma) & = & \sum_{l=1}^n  \sum_{d\in \mathfrak{d}_l}
\left( \sum_{x_1\in \gamma} \widehat{\theta}_1 (x_1) \sum_{x_2\in
\gamma\setminus x_1} \widehat{\theta}_2 (x_2) \cdots \sum_{x_l\in
\gamma\setminus \{x_1, \dots , x_{l-1}\}}
\widehat{\theta}_l (x_l) \right) \\[.2cm] \nonumber & = & \sum_{l=1}^n l!  \sum_{d\in
\mathfrak{d}_l} (K G^{\widehat{\vartheta}})(\gamma), \qquad
\widehat{\vartheta} := \{ \widehat{\theta}_1, \dots ,
\widehat{\theta}_l\},
\end{eqnarray}
where $G^{\widehat{\vartheta}}$ is as in (\ref{X7}). 
Note that each
$\widehat{\theta}_k$ is in $\Theta$ for $\psi(x)\leq 1$, see \eqref{psi}. Now we set,  cf. \eqref{G2},
\begin{gather}
  \label{X28}
H_\beta^\vartheta (\eta) = \prod_{x\in \eta} \left(e^{-\tau
\tilde{\theta}_\beta(x)} -1 \right), \qquad \tilde{\theta}_\beta (x)
= \sum_{j=1}^n \beta_j \theta_j (x), \qquad \eta \in \Gamma_{\rm
fin}.
\end{gather}
Then
\begin{eqnarray*}
 & & \exp\left( - \tau\sum_{x\in \gamma}
\sum_{j=1}^n \beta_j \theta_j(x)\right)  =  \exp\left( - \tau\sum_{x\in \gamma}
  \tilde{\theta}_\beta(x)\right) \\[.2cm] & & =  \prod_{x\in \gamma} \left(1 + [e^{-\tau \tilde{\theta}_\beta(x)} -1]\right) 
 = \sum_{\eta \Subset \gamma} H^\vartheta_\beta(\eta) =(K
H^\vartheta_\beta)(\gamma),
\end{eqnarray*}
and further, see (\ref{X25}),
\begin{gather}
  \label{X30}
  \varPsi^\vartheta_{\tau} (\gamma) = \sum_{l=1}^n
  \sum_{d\in \mathfrak{d}_l} \int_0^{+\infty} \cdots \int_0^{+\infty} e^{-(\beta_1 + \cdots + \beta_n)} l!
 (K G^{\widehat{\vartheta}})(\gamma) (K
H^\vartheta_\beta)(\gamma) d\beta_1 \cdots d \beta_n.
\end{gather}
It is known, see \cite[Definition 3.2]{Kuna}, that
\begin{gather}
  \label{X31}
(K G^{\widehat{\vartheta}})(\gamma) (K H^\vartheta_\beta)(\gamma) =
(K (G^{\widehat{\vartheta}} \star H^\vartheta_\beta))
(\gamma),\\[.2cm] \nonumber( G^{\widehat{\vartheta}} \star
H^\vartheta_\beta) (\eta) = \sum_{\xi_1 \subset \eta} \sum_{\xi_2
\subset \eta\setminus \xi_1}G^{\widehat{\vartheta} } (\xi_1 \cup
\xi_2) H^\vartheta_\beta (\eta\setminus \xi_2).
\end{gather}
Note that
the sums in the right-hand side of (\ref{X31}) are finite since $\eta$ is a finite multiset.  This allows one to interchange $K$ and the integration in \eqref{X30}.  
Then
 \[
\varPsi^\vartheta_{\tau} = K  G^\vartheta_{\tau},
 \]
with
\begin{gather}
  \label{X30a}
G^\vartheta_{\tau} (\eta)= \sum_{l=1}^n
  \sum_{d\in \mathfrak{d}_l}\int_0^{+\infty} \cdots \int_0^{+\infty} e^{-(\beta_1 + \cdots + \beta_n)} l!(G^{\widehat{\vartheta}} \star
H^\vartheta_\beta) (\eta) d \beta_1 \cdots d \beta_n.
\end{gather}
To prove the lemma, we have to estimate the norms of $G^\vartheta_{\tau}$, see \eqref{Y5} and \eqref{Y6}. To interchange the $\beta$-integration with that over $\Gamma_{\rm fin}$ we employ the Tonelli-Fubini theorems \cite[Theorems 4.4 and 4.5, page 91]{Brezis}. 
By the evident inequality $1-e^{-u} \leq \sqrt{2u}$, $u\geq 0$, we have, cf. \eqref{X28},
\begin{equation}
  \label{Xx}
|H^\vartheta_\beta(\eta)| \leq \prod_{x\in \eta} (1- e^{-\tau
\tilde{\theta}_\beta (x)}) \leq (2\tau)^{|\eta|/2} e(\eta;
{\theta}_\beta) \leq e(\eta;
{\theta}_\beta),  \quad   {\theta}_\beta (x) :=  \sqrt{ \tilde{\theta}_\beta (x)}, 
\end{equation}
see (\ref{Y0a}). According to (\ref{Y5}) and by means of (\ref{XL}),
(\ref{Xx}) we then get
\begin{gather}
  \label{X32}
|G^{\widehat{\vartheta}} \star H^\vartheta_\beta|_\alpha \leq
\int_{\Gamma_{\rm fin}} e^{\alpha |\eta|} \left( \sum_{\xi_1 \subset
\eta} \sum_{\xi_2 \subset \eta\setminus \xi_1}G^{\widehat{\vartheta} } (\xi_1 \cup \xi_2) e(\eta\setminus
\xi_2; {\theta}_\beta) \right) \lambda ( d \eta)\\[.2cm]
= \int_{\Gamma^2_{\rm fin}} e^{\alpha |\eta| + \alpha|\xi_1|} \left(
\sum_{\xi_2 \subset \eta} G^{\widehat{\vartheta} } (\xi_1 \cup \xi_2)e(\eta\setminus
\xi_2\cup \xi_1; {\theta}_\beta) \right)\lambda ( d\eta)
\lambda (d \xi_1) \nonumber \\[.2cm] \nonumber =
\int_{\Gamma^3_{\rm fin}} e^{\alpha |\eta| + \alpha|\xi_1|+
\alpha|\xi_2|}  G^{\widehat{\vartheta} }
(\xi_1 \cup \xi_2) e(\xi_1;{\theta}_\beta)
e(\eta;{\theta}_\beta) \lambda ( d \eta) \lambda ( d \xi_1)
\lambda ( d \xi_2) \\[.2cm] \nonumber = \exp\left(e^\alpha \int_X {\theta}_\beta (x) d x
\right) \int_{\Gamma^2_{\rm fin}}  e^{\alpha |\xi_1| +
\alpha|\xi_2|} G^{\widehat{\vartheta} } (\xi_1 \cup \xi_2)
e(\xi_1;{\theta}_\beta) \lambda ( d \xi_1) \lambda ( d \xi_2).
\end{gather}
Here we have taken into account that
\[
e(\xi_1 \cup \xi_2; {\theta}_\beta) = e(\xi_1;
{\theta}_\beta) e(\xi_2; {\theta}_\beta),
\]
see (\ref{Y0a}), and also (\ref{X14}). By (\ref{X7}) and the latter,
we have
\begin{eqnarray}
  \label{X33}
\Upsilon_l & := & l!\int_{\Gamma^2_{\rm fin}}  e^{\alpha |\xi_1| +
\alpha|\xi_2|} G^{\widehat{\vartheta} } (\xi_1 \cup \xi_2)
e(\xi_1;{\theta}_\beta) \lambda ( d \xi_1) \lambda ( d
\xi_2)\\[.2cm] \nonumber
& =  & l! \int_{\Gamma_{\rm fin}} e^{\alpha |\eta|}
G^{\widehat{\vartheta} } (\eta)  \left( \sum_{\xi\subset \eta}
 e(\xi; {\theta}_\beta)\right) \lambda ( d \eta).
\end{eqnarray}
Now we recall that all 
$\theta_i \in \vartheta$ are in $\Theta$, see (\ref{X18a}); hence, $\theta_i (x) \leq 1$.  By
(\ref{X28}) and (\ref{Xx}) we then have 
\begin{equation}
 \label{X33a}
 0\leq {\theta}_\beta (x)\leq \omega_\beta : = \sqrt{\beta_1+\cdots + \beta_n},  
\end{equation}
holding for all $x\in X$.
Hence, $e(\xi; {\theta}_\beta)\leq
\omega_\beta^{|\xi|}$, see (\ref{Y0a}). By means of this estimate and
(\ref{X7}), (\ref{X26}) we get in (\ref{X33}) the following one
\begin{eqnarray}
  \label{X34}
\Upsilon_l & \leq & l!  \int_{\Gamma_{\rm fin}}\left[ e^\alpha (1+ \omega_\beta)\right]^{|\eta|} G^{\widehat{\vartheta} }(\eta) \lambda ( d
\eta) \\[.2cm] \nonumber & = & \left[ e^\alpha (1+ \omega_\beta)\right]^l \prod_{k=1}^l
\langle \widehat{\theta}_k \rangle \leq \left[ e^\alpha \langle {\psi} \rangle (1+ \omega_\beta)\right]^l,
\end{eqnarray}
see (\ref{X8}) and \eqref{X18a}. Now we employ (\ref{X34}) and (\ref{X33a}) in (\ref{X32}) and finally
get
\begin{equation*}
l!|G^{\widehat{\vartheta}} \star H^\vartheta_\beta|_\alpha \leq
\exp\left( e^\alpha  \langle \psi \rangle \omega_\beta\right)
\left[e^\alpha \langle \psi \rangle  (1+ \omega_\beta )\right]^l,
\end{equation*}
see also  \eqref{psi1}.
This yields in \eqref{X30a} the following
\begin{eqnarray}
\label{X36}
 |G^\vartheta_\tau|_\alpha & \leq & \int_0^{+\infty}\cdots \int_0^{+\infty}  \exp\left( - \beta_1 \cdots - \beta_n +  e^\alpha  \langle \psi \rangle \omega_\beta\right)\\[.2cm]\nonumber & \times & \sum_{l=1}^n S(n,l) \left[e^\alpha \langle \psi \rangle  (1+ \omega_\beta )\right]^l d\beta_1 \cdots d\beta_n \\[.2cm] \nonumber & = & \int_0^{+\infty}\cdots \int_0^{+\infty}  \exp\left( - \beta_1 \cdots - \beta_n \right) W_\alpha (\beta_1, \dots , \beta_n)   d\beta_1 \cdots d\beta_n
\end{eqnarray}
where
\[
W_\alpha (\beta_1, \dots , \beta_n) :=  \exp\left( e^\alpha  \langle \psi \rangle \omega_\beta\right) T_n\left(e^\alpha \langle \psi \rangle  (1+ \omega_\beta )\right),
\]
Here $T_n(u) = \sum_{l=0}^n S(n,l) u^l$ is Touchard's polynomial, $\deg T_n = n$, and $S(n,l)$ is Stirling's number of the second kind.
Thus, the integral in the last line of \eqref{X36} 
is obviously convergent for all $\alpha \in \mathds{R}$, see \eqref{X33a}. This completes the proof.   
\end{proof}

\subsection{Useful estimates and their consequences}

Here, we derive some estimates by means of which we then obtain important properties of possible solutions of the Fokker-Planck equation with the domain $\mathcal{F}$ defined in (\ref{X18}), cf. Definition \ref{1df}.  As a result, we extend the domain of $L$ by adding some unbounded functions, which, in particular, will allow us to prove that each solution lies in $\mathcal{P}_{\rm exp}$. 

For $F\in \mathcal{F}$, write
\begin{equation}
 \label{X19}
 \nabla_x F(\gamma) = F(\gamma\cup x) - F(\gamma), \qquad x\in X.
\end{equation}
Now we take $F= \varPhi^\theta_\tau$, see \eqref{X16}, and get
\begin{equation*}
\nabla_x  \varPhi^\theta_\tau (\gamma) = \frac{\theta(x)}{[1+\tau \varPhi^\theta_\tau (\gamma\cup x) ][1+\tau \varPhi^\theta_\tau (\gamma)]},
\end{equation*}
which immediately yields
\begin{equation}
 \label{0b}
 0 \leq \nabla_x  \varPhi^\theta_\tau (\gamma)  \leq \theta (x), \qquad \frac{\partial}{\partial \tau} \nabla_x  \varPhi^\theta_\tau (\gamma) \leq 0,  
\end{equation}
Next, see \eqref{M15},
\begin{gather}
 \label{0c}
 0 \leq \nabla_x  \varPsi^\vartheta_\tau (\gamma) = \prod_{\theta\in \vartheta} \left[\nabla_x  \varPhi^\theta_\tau (\gamma) + \varPhi^\theta_\tau (\gamma)\right] - \prod_{\theta\in \vartheta}\varPhi^\theta_\tau (\gamma) \\[.2cm] \nonumber \leq \sum_{\varnothing \neq \vartheta'\subset \vartheta} \left(\prod_{\theta \in \vartheta'} \theta(x) \right)  \varPsi^{\vartheta\setminus \vartheta'}_\tau (\gamma).
\end{gather}
By \eqref{L} and \eqref{0b}, we then conclude that
\begin{equation}
 \label{0d}
 \pm L^{\pm} \varPsi^\vartheta_\tau (\gamma) \geq 0 \quad {\rm and} \quad \frac{\partial}{\partial \tau}\left(\pm L^{\pm} \varPsi^\vartheta_\tau (\gamma)  \right)\leq 0,
\end{equation}
and also
\begin{gather}
 \label{0da}
 L^{+} \varPsi^\vartheta_\tau (\gamma) \leq \sum_{\varnothing \neq \vartheta'\subset \vartheta} \left(\int_{X} b(x)\prod_{\theta \in \vartheta'} \theta(x) dx \right)  \varPsi^{\vartheta\setminus \vartheta'}_\tau (\gamma)\leq \|b\|\langle \psi \rangle \sum_{\varnothing \neq \vartheta'\subset \vartheta}\varPsi^{\vartheta\setminus \vartheta'}_\tau (\gamma),
\end{gather}
where $\langle \psi \rangle = \int \psi(x) d x$, see \eqref{psi1} and \eqref{X18a}.
By the first inequality in \eqref{0d} it follows that
\begin{equation}
 \label{0e}
 \mu_t (\varPsi^\vartheta_\tau) \leq \mu_0 (\varPsi^\vartheta_\tau) + \int_0^t \mu_s (L^{+}\varPsi^\vartheta_\tau) ds,
\end{equation}
which should be true for any solution $\mu_t$. By \eqref{X25}, we have
\begin{equation}
 \label{0ea}
 \varPsi^\vartheta_\tau (\gamma) \leq \varPsi^\vartheta (\gamma) , \qquad \lim_{\tau \to 0}\varPsi^\vartheta_\tau (\gamma) = \varPsi^\vartheta (\gamma).
\end{equation}
For $\mu_0 \in \mathcal{P}_{\rm exp}$, similarly as in \eqref{X34}, by means of \eqref{0ea}, \eqref{X27} and \eqref{X8} we obtain
\begin{gather}
 \label{0f}
 \mu_0(\varPsi^\vartheta_\tau) \leq \mu_0(\varPsi^\vartheta) = \sum_{l=1}^{|\vartheta|}\sum_{d\in \mathfrak{d}_l} \int_{X^l} k_{\mu_0}^{(l)} (x_1 , \dots , x_l) \widehat{\theta}_1(x_1)\cdots \widehat{\theta}_l(x_l) d x_1 \cdots d x_l \\[.2cm] \nonumber \leq \sum_{l=1}^{|\vartheta|}\sum_{d\in \mathfrak{d}_l} \varkappa_{\mu_0}^l \langle \widehat{\theta}_1\rangle\cdots \langle \widehat{\theta}_l \rangle  \leq \sum_{l=1}^{|\vartheta|} S(|\vartheta|, l) \left[ \varkappa_{\mu_0} \langle \psi \rangle \right]^l = T_{|\vartheta|} (\varkappa_{\mu_0}\langle \psi \rangle),
\end{gather}
where  $\varkappa_{\mu_{0}}$ is the type of $\mu_0$, see Definition \ref{X3df}. Recall that $T_n$ stands for Touchard's polynomial.   
By \eqref{0b}, \eqref{0e} and \eqref{0f} we then get
\begin{equation*}
 \mu_t (\varPhi^\theta_\tau) \leq T_{1} (\varkappa_{\mu_0}\langle\psi \rangle) + \|b\| \langle\psi \rangle t, \qquad \theta \in \Theta,
\end{equation*}
see \eqref{6} and \eqref{X18a}. Now by \eqref{0c} this can be generalized to the following recursion 
\begin{equation}
 \label{0h}
 \mu_t (\varPsi^\vartheta_\tau) \leq T_{|\vartheta|} (\varkappa_{\mu_0} \langle\psi \rangle) + 
 \int_0^t \left( \sum_{\varnothing \neq \vartheta' \subset \vartheta}\mu_s (\varPsi^{\vartheta\setminus \vartheta'}_\tau) \right) d s,
\end{equation}
by which, we obtain
\begin{equation}
 \label{0i}
 \mu_t (\varPsi^\vartheta_\tau) \leq Q_{|\vartheta|} (t) := \sum_{l=0}^{|\vartheta|} \frac{1}{l!} T_{|\vartheta|-l}(\varkappa_{\mu_0}\langle\psi \rangle) \left[\|b\| \langle\psi \rangle\right]^l t^l.
\end{equation}
By this estimate and \eqref{0da}, we also get
\begin{equation}
 \label{0j}
 \mu_t (L^{+}\varPsi^\vartheta_\tau) \leq \|b\| \langle \psi \rangle \sum_{k=1}^{|\vartheta|}{|\vartheta| \choose k} 
 Q_{k} (t)  =:  Q^{+}_{|\vartheta|} (t).
\end{equation}
In particular, this implies that  $\mu_t (L^{+}\varPsi^\vartheta_\tau)$ is a polynomial in $t$, which implies in turn the local integrability of the map $t \mapsto \mu_t (L^{+}\varPsi^\vartheta_\tau)$. Then the $\mu_t (d\gamma)dt$-integrability of  $|L^{-}\varPsi^\vartheta_\tau (\gamma)|$ follows by the triangle inequality from the assumed corresponding integrability of $|L\varPsi^\vartheta_\tau (\gamma)|$, see Definition \ref{1df}, cf. \eqref{0ja}, and \eqref{0j}. 

Now we can prove the following important statement. 
\begin{lemma}
 \label{0lm}
Let $\mu_t$ be a solution of the Fokker-Planck equation for $(L, \mathcal{F}, \mu_0)$ with some $\mu_0\in \mathcal{P}_{\rm exp}$. Then $\mu_t$ solves \eqref{FPE} also with  $F= \varPsi^\vartheta$,
see \eqref{X25a}.  
\end{lemma}
\begin{proof}
In view of \eqref{0ea}, to prove the lemma, one should show that, for all $t\geq 0$, the following holds 
\begin{gather}
 \label{0k}
\lim_{\tau\to 0} \mu_t (\varPsi^\vartheta_{\tau}) = \mu_t (\varPsi^\vartheta), \qquad  \lim_{\tau\to 0} \int_0^t \mu_s (L^{\pm}\varPsi^\vartheta_{\tau}) ds  = \int_0^t \mu_s (L^{\pm} \varPsi^\vartheta) ds.  
\end{gather}
By \eqref{X16}, \eqref{X15}, and then by \eqref{0b}, \eqref{0d}, the maps $\tau \mapsto  \mu_t (\varPsi^\vartheta_{\tau})$ and $\tau \mapsto \int_0^t \mu_s (L^{\pm}\varPsi^\vartheta_{\tau}) ds$ are monotone. Then the convergence as in \eqref{0k} follows by the monotone convergence    
theorem, the bounds in \eqref{0h}, \eqref{0j} and the following one, see also \eqref{0d},
\begin{gather*}
 - \int_0^t \mu_s (L^{-}\varPsi^\vartheta_{\tau}) ds = \int_0^t \mu_s (|L^{-}\varPsi^\vartheta_{\tau}|)ds = \mu_0 (\varPsi^\vartheta_{\tau}) - \mu_t (\varPsi^\vartheta_{\tau}) +  \int_0^t \mu_s (L^{+}\varPsi^\vartheta_{\tau}) ds \\[.2cm] \nonumber \leq \mu_0 (\varPsi^\vartheta_{\tau}) +  \int_0^t \mu_s (L^{+}\varPsi^\vartheta_{\tau}) ds \leq \mu_0 (\varPsi^\vartheta) + \int_0^t Q^{+}_{|\vartheta|} (s) ds.
\end{gather*}
This completes the proof. 
\end{proof}

\subsection{Locating the solutions}

This subsection aims to prove that any solution of the Fokker-Planck equation for $(L, \mathcal{F}, \mu_0)$ with $\mu_0\in \mathcal{P}_{\rm exp}$ lies in $\mathcal{P}_{\rm exp}$, which can be achieved with the help of Lemma \ref{0lm}. 

For $\vartheta \subset \Theta$, $|\vartheta|=n$, define
\begin{equation}
 \label{0n}
 F^\vartheta (\gamma) = \sum_{x_1\in \gamma} \theta_1 (x_1) \sum_{x_2\in \gamma
\setminus x_1} \theta_2 (x_2) \cdots \sum_{x_n\in \gamma \setminus \{x_1, \dots, x_{n-1}\}} \theta_n (x_n).  
\end{equation}
By \eqref{X7} and \eqref{X6} it follows that $F^\vartheta = n! K G^\vartheta$, which by \eqref{X8} yields
\begin{equation}
 \label{0na}
 \mu (F^\vartheta ) = \int_{X^n} k^{(n)}_\mu (x_1 , \dots , x_n) \theta_1 (x_1) \cdots \theta_n (x) dx_1 \cdots d x_n \leq \varkappa_\mu^n \langle \theta_1 \rangle \cdots \langle \theta_n \rangle,
\end{equation}
holding for all $\mu \in \mathcal{P}_{\rm exp}$. At the same time, by \eqref{X2} we get
\begin{equation}
 \label{0p}
F^\vartheta = \sum_{\mathbb{G}\subset \mathbb{K}_n} (-1)^{l_{\mathbb{G}}} \varPsi^{\widehat{\vartheta}},  \qquad \widehat{\vartheta}= \{\widehat{\theta}_1 , \dots \widehat{\theta}_{n_{\mathbb{G}}}\},  
\end{equation}
where $\widehat{\theta}_j (x)= \prod_{i\in \mathbb{V}_j}\theta_i (x)$,  $\mathbb{V}_j$ being the $j$-th connected component of $\mathbb{G}$. 
\begin{remark}
 \label{01rk}
Each  $\widehat{\theta}_j$ that appears in \eqref{0p} is in $\Theta$, see \eqref{X18a}. Then $F^\vartheta$ is a linear combination of $\varPsi^{\widehat{\vartheta}}$ with $\widehat{\theta} \in \Theta$. Hence, 
by Lemma \ref{0lm}, each solution $\mu_t$ solves \eqref{FPE} also with $F=F^\vartheta$ for any   $\vartheta \subset \Theta$. 
\end{remark}
Then similarly as in \eqref{0e} we get
\begin{equation}
 \label{0q}
 \mu_t(F^\vartheta) \leq \mu_0(F^\vartheta) +\sum_{\theta \in \vartheta} \|b\|\langle \theta \rangle \int_0^t \mu_s(L^{+}F^{\vartheta\setminus \theta}) ds.
\end{equation}
For a singleton $\vartheta = \{\theta\}$, by \eqref{0na} we have
\[
 \mu_t(F^\vartheta) \leq (\varkappa_{\mu_0} + \|b\|t)\langle \theta \rangle. 
\]
Then we iterate this estimate in \eqref{0q} and obtain
\begin{equation}
 \label{0qa}
 \mu_t(F^\vartheta) \leq \left(\varkappa_{\mu_0} + \|b\| t \right)^n \langle \theta_1 \rangle \cdots \langle \theta_n \rangle.
\end{equation}
Therefore, we have proved the following statement.
\begin{lemma}
 \label{01lm}
 For every $\mu_0 \in \mathcal{P}_{\rm exp}$, each solution $\mu_t$ of the Fokker-Planck equation \eqref{FPE} for $(L,\mathcal{F},\mu_0)$ lies in $\mathcal{P}_{\rm exp}$ for all $t>0$. Moreover, its type satisfies $\varkappa_{\mu_t} \leq \varkappa_{\mu_0} + \|b\| t$.  
\end{lemma}
For $\mu_t$ as in Lemma \ref{01lm}, by \eqref{Gas1} it follows that $\mu_t(\Gamma_*)=1$ for all $t\geq 0$, which by Remark \ref{Gasrk} yields that each $\mu_t$ can be redefined as an element of $\mathcal{P}(\Gamma_*)$. At the same time, for $\gamma \in \Gamma_*$ and $\vartheta \subset \Theta$, by \eqref{0c} and \eqref{0d} it follows that  
\begin{eqnarray*}
 |L^{-} \varPsi^\vartheta_\tau (\gamma)|&  \leq & H_\vartheta (\gamma) \sum_{\varnothing \neq \vartheta' \subset \vartheta} \varPsi^{\vartheta\setminus \vartheta'}_\tau (\gamma), \\[.2cm] \nonumber H_\vartheta (\gamma) & = & \sum_{x\in \gamma} \theta_*(x)\left( m(x)   +  \sum_{y\in \gamma\setminus x} a(x-y) \right) \\[.2cm] \nonumber & \leq & \sum_{x\in \gamma}\theta_* (x)\left(m(x) + \gamma(\psi) \|a\| \ell_a (x) \right)   < \infty,
\end{eqnarray*}
see also \eqref{6c}. Here $\theta_*\in C^{+}_{\rm cs} (X)$ is such that $\theta (x) \leq \theta_*(x)$ for all $\theta \in \vartheta$ and $x\in X$.

\section{Solving the Fokker-Planck equation}

In this section, we give the proof of  Theorem \ref{1tm}, which is essentially based on the results of \cite{KK}.
In view of Lemma \ref{01lm}, the evolution $\mu_0 \to \mu_t$ will be constructed as the evolution of the corresponding correlation functions in the following two steps: (a) constructing $k_{\mu_0} \to k_t$ in a scale of Banach spaces, see \eqref{Y2}; (b) proving that $k_t$ is the correlation function of a unique $\mu_t$.  

\subsection{Evolution of correlation functions }

The evolution of correlation functions is supposed to be obtained by employing the generator 
$L^\Delta$ deduced from the Kolmogorov operator \eqref{L} according to the following rule, cf. \eqref{X12}, 
\begin{equation}
  \label{X38a}
  \mu(L KG) = \langle \! \langle L^\Delta k_\mu, G \rangle \!
  \rangle,
\end{equation}
valid for $\mu\in \mathcal{P}_{\rm exp}$ and appropriate functions, e.g., for  $G\in B_{\rm bs}$, see Definition \ref{X1df}. With the help of \eqref{XL}, after some calculations, one obtains, see \cite[eqs. (2.9)--(2.11)]{KK}
\begin{equation}
  \label{X38}
  (L^\Delta k)(\eta) = \sum_{x\in \eta} b(x) k(\eta\setminus x) - E(\eta)
  k(\eta) - \int_X \left( \sum_{y\in \eta} a (x-y) \right) k(\eta
  \cup x) dx,
\end{equation}
where 
\begin{equation}
  \label{X37}
  E(\eta) = \sum_{x\in \eta} m(x) + \sum_{x\in \eta}\sum_{y\in \eta\setminus
  x}a(x-y), \qquad \eta\in \Gamma_{\rm fin}.
\end{equation}
The expression in \eqref{X38} will define 
linear operators acting in the Banach spaces
$\mathcal{K}_\alpha$, see (\ref{Y1}), (\ref{Y2}) and
Lemma \ref{Y3lm} below. To realize this, we set
\begin{equation}
  \label{X39}
  \mathcal{D}_\alpha = \{k \in \mathcal{K}_\alpha : L^\Delta k \in
  \mathcal{K}_\alpha\}, \qquad \alpha \in \mathds{R}.
\end{equation}
Keeping in mind the embedding as in (\ref{Y2}), for $\alpha' <
\alpha$ we now introduce a linear operator $L^\Delta_{\alpha
\alpha'}: \mathcal{K}_{\alpha'} \to \mathcal{K}_{\alpha}$ which acts
according to (\ref{X38}). It turns out that this operator is bounded
as its operator norm satisfies, see \cite[eq. (2.26)]{KK},
\begin{equation}
  \label{X40}
\|L^\Delta_{\alpha \alpha'}\| \leq \frac{4\|a\|_\psi}{e^2 (\alpha -
\alpha')^2} + \frac{\|b\|e^{-\alpha'} + \|m\| + \langle a \rangle
e^{\alpha}}{e (\alpha - \alpha')},
\end{equation}
where $\|a\|_\psi$, $\|m\|$, $\|b\|$, and $\langle a \rangle$ are as in \eqref{6} and \eqref{6a}.
Let $\mathcal{L}_{\alpha \alpha'}$ stand for the Banach space of
bounded linear operators $A :\mathcal{K}_{\alpha'} \to
\mathcal{K}_{\alpha}$. Then 
\begin{equation}
\label{Qqz}
L^\Delta_{\alpha \alpha'}\in
\mathcal{L}_{\alpha \alpha'} \quad {\rm and} \quad L^\Delta_{\alpha \alpha'}\vert_{\mathcal{K}_{\alpha''}} = L^\Delta_{\alpha \alpha''} \quad {\rm for} \quad \alpha'' < \alpha'. 
\end{equation}
In view of (\ref{Y2}), this fact
obviously implies that, see (\ref{X39}),
\begin{equation}
  \label{X41}
  \forall \alpha' < \alpha \qquad \mathcal{K}_{\alpha'} \subset
  \mathcal{D}_\alpha.
\end{equation}
Let us consider the following Cauchy problem in
$\mathcal{K}_{\alpha}$ for the operator $(L^\Delta,
\mathcal{D}_\alpha)$
\begin{equation}
  \label{X42}
  \frac{d}{dt} k_t = L^\Delta k_t, \qquad k_t|_{t=0} = k_0 \in
  \mathcal{D}_\alpha.
\end{equation}
In this general setting, it is barely possible to solve (\ref{X42})
by applying $C_0$-semigroup methods. Recall that we deal here
with $L^\infty$-type spaces, see (\ref{Y1}). However, if one takes
$k_0 \in \mathcal{K}_{\alpha'}$ for some $\alpha'< \alpha$, i.e.,
from a subset of the domain, see (\ref{X41}), a solution can be
obtained in the following way. Define
\begin{equation}
  \label{X43}
  (S (t) k)(\eta) = \exp\left( - t E(\eta) \right) k(\eta), \qquad
  t>0,
\end{equation}
where $E(\eta)$ is as in (\ref{X37}). Note that this is one of the
steps, where we properly take into account the sign of the quadratic
term in $L^{-}$. By means of (\ref{X43}) we then define $S_{\alpha
\alpha'}(t)\in \mathcal{L}_{\alpha \alpha'}$. One can show that the
map $t \mapsto S_{\alpha \alpha'}(t)\in \mathcal{L}_{\alpha
\alpha'}$ is continuous. Obviously, for each $t$, $S(t)$ as in (\ref{X43}) defines
a bounded operator acting from $\mathcal{K}_\alpha$ to
$\mathcal{K}_\alpha$. We use it as $S_{\alpha \alpha'}(t)$ (i.e.,
acting to a bigger space) to secure the continuity just mentioned. Let
$A_{\alpha \alpha'}\in \mathcal{L}_{\alpha \alpha'}$ be defined by
the expression
\begin{equation}
  \label{X44}
  (A k)(\eta) = - E(\eta) k(\eta),
\end{equation}
i.e., it is the multiplication operator by $-E(\eta)$. Clearly, the
map $t \mapsto S_{\alpha \alpha'}(t)$ is differentiable and the
following holds
\begin{equation*}
  \frac{d}{dt}  S_{\alpha \alpha'}(t) = A_{\alpha \alpha''}  S_{\alpha''
  \alpha'}(t) = S_{\alpha
  \alpha''}(t)  A_{\alpha'' \alpha'},
\end{equation*}
for each $\alpha'' \in (\alpha' , \alpha)$. Now we set $B= L^\Delta
- A$ and define the corresponding bounded linear operator $B_{\alpha
\alpha'}\in \mathcal{L}_{\alpha \alpha'}$. Let us prove that its norm 
satisfies, cf. (\ref{X40}),
\begin{equation}
\label{X46}
  \| B_{\alpha \alpha'}\| \leq \frac{\|b\| e^{-\alpha '} + \langle a \rangle e^\alpha}{e(\alpha -
  \alpha')}.
\end{equation}
By \eqref{X38}, we write $B= B' + B''$ with
\[
(B'k)(\eta) =  \sum_{x\in \eta} b(x) k(\eta\setminus x), \qquad (B''k)(\eta) = - \int_X \left( \sum_{y\in \eta}a(x-y)\right) k(\eta \cup x) d x.
\]
Next, by \eqref{Y1}, we have
\begin{equation}
\label{OC}
|k(\eta)| \leq \|k\|_{\alpha'} e^{\alpha'|\eta|}, \qquad \eta \in \Gamma_{\rm fin}.
\end{equation}
Then by the elementary inequality $te^{-\kappa t} \leq 1/e \kappa$, $t, \kappa >0$,  we get from \eqref{OC} the following 
\begin{eqnarray}
\label{OC1}
\|B'k\|_\alpha \leq \|b\| \|k\|_{\alpha'} e^{-\alpha'} \sup_{n\in \mathds{N}} n e^{- (\alpha-\alpha') n} \leq 
\frac{\|b\| e^{-\alpha'}}{e(\alpha - \alpha')}  \|k\|_{\alpha'}.
\end{eqnarray}
Similarly,
\begin{eqnarray*}
\|B''k\|_\alpha \leq \langle a \rangle \|k\|_{\alpha'} e^{\alpha'} \sup_{n\in \mathds{N}} n e^{- (\alpha-\alpha') n} \leq 
\frac{\langle a \rangle e^{\alpha}}{e(\alpha - \alpha')}  \|k\|_{\alpha'},
\end{eqnarray*}
which together with \eqref{OC1} yields \eqref{X46}.
It is crucial that $\alpha - \alpha'$ appears in \eqref{X46} in the first power.
Define
\begin{equation}
  \label{X46a}
  T(\alpha,\alpha') = \frac{\alpha -\alpha'}{\|b\| e^{-\alpha '} + \langle a \rangle
  e^\alpha}.
\end{equation}
Fix now some $\delta< \alpha - \alpha'$ and $l\in \mathds{N}$, and
then set
\begin{gather*}
  \alpha^{2s} = \alpha' + \frac{s}{l+1} \delta + s \epsilon, \qquad
  \epsilon = (\alpha - \alpha' -\delta) /l, \\[.2cm] \nonumber
  \alpha^{2s+1} = \alpha' + \frac{s+1}{l+1} \delta + s \epsilon,
  \qquad s= 0,1, \dots , l.
\end{gather*}
Note that $\alpha^0=\alpha'$ and $\alpha^{2l+1} = \alpha$. For
$t>0$, set
\begin{equation*}
  \mathcal{T}_l = \{ (t, t_1, \dots , t_l) : 0\leq t_l \leq t_{l-1}
  \leq \cdots \leq t_1\leq t\}\subset (\mathds{R}_{+})^{l+1},
\end{equation*}
and then
\begin{gather*}
  \Pi^l_{\alpha\alpha'} (t, t_1, \dots , t_l) = S_{\alpha
  \alpha^{2l}}(t-t_1) B_{\alpha^{2l}\alpha^{2l-1}} \cdots \times
  \\[.2cm] \nonumber \times
  S_{\alpha^{3}
  \alpha^{2}}(t-t_1) (t_{l-1}-t_l) B_{\alpha^2 \alpha^1} S_{\alpha^1
  \alpha'}(t_t), \quad (t, t_1, \dots , t_l) \in \mathcal{T}_l.
\end{gather*}
It is known, see \cite[Proposition 3.1]{KK}, that the map
\[
\mathcal{T}_l \ni (t, t_1, \dots , t_l) \mapsto
\Pi^l_{\alpha\alpha'} (t, t_1, \dots , t_l) \in
\mathcal{L}_{\alpha\alpha'}
\]
is continuous. Moreover, for each $\delta \in (0, \alpha-\alpha')$,
the operator norm satisfies
\begin{equation}
  \label{X50}
  \|\Pi^l_{\alpha\alpha'} (t, t_1, \dots , t_l)\| \leq
  \left(\frac{l}{eT(\alpha - \delta, \alpha')}\right)^l,
\end{equation}
see (\ref{X46a}). Define
\begin{gather}
  \label{X51}
Q_{\alpha \alpha'} (t) = S_{\alpha \alpha'}(t) + \sum_{l=1}^\infty
\int_0^t d t_1 \int_0^{t_1} d t_2 \cdots \int_0^{t_{l-1}} d t_l
\Pi^l_{\alpha\alpha'} (t, t_1, \dots , t_l).
\end{gather}
By means of (\ref{X50}) one readily gets that 
\begin{equation}
 \label{OC5}
 \|Q_{\alpha \alpha'} (t)\| \leq \frac{T(\alpha, \alpha')}{T(\alpha, \alpha') - t}, \qquad t \in [0, T(\alpha, \alpha')),
\end{equation}
which means that $Q_{\alpha \alpha'} (t) \in \mathcal{L}_{\alpha\alpha'}$ for such $t$.  
\begin{proposition}\cite[Proposition 3.2]{KK}
  \label{X4pn}
For each $\alpha, \alpha'\in \mathds{R}$, $\alpha' <\alpha$ and $t<
T(\alpha,\alpha')$, the series in (\ref{X51}) converges in the norm
of $\mathcal{L}_{\alpha\alpha'}$ in such a way that
\begin{itemize}
  \item[(i)] the map $[0,T(\alpha, \alpha')) \ni t \mapsto Q_{\alpha \alpha'}
  (t)\in \mathcal{L}_{\alpha\alpha'}$ is continuous and $Q_{\alpha \alpha'}
  (0)$ is  the embedding operator as in (\ref{Y2});
  \item[(ii)] for each $\alpha''\in (\alpha',\alpha)$ and $t< \min\{
  T(\alpha'',\alpha'); T(\alpha,\alpha'')\}$, the following holds
  \begin{equation}
    \label{X52}
    \frac{d}{dt}Q_{\alpha \alpha'}
  (t) = L^\Delta_{\alpha \alpha''} Q_{\alpha'' \alpha'}
  (t)= Q_{\alpha \alpha''} (t)L^\Delta_{\alpha'' \alpha'}.
  \end{equation}
\item[(iii)] the operators $Q_{\alpha\alpha'}(t)$ enjoy the
semigroup property
\begin{equation}
  \label{X53}
  Q_{\alpha \alpha'}
  (t+s) = Q_{\alpha \alpha''} (t)
  Q_{\alpha'' \alpha'}
  (s),
\end{equation}
that holds provided $t< T(\alpha, \alpha'')$, $s<
  T(\alpha'', \alpha')$ and $t+s < T(\alpha, \alpha')$.
\end{itemize}
\end{proposition}
This assertion allows us to solve the Cauchy problem in (\ref{X42})
in the following form.
\begin{proposition}\cite[Lemma 3.3]{KK}
  \label{X5pn}
For each $k_0 \in \mathcal{K}_{\alpha'}$, the problem in (\ref{X42})
has a unique classical solution $k_t \in \mathcal{K}_\alpha$, $t <
T(\alpha, \alpha')$, given by the formula
\begin{equation}
  \label{X54}
  k_t = Q_{\alpha \alpha'}(t) k_0.
\end{equation}
This solution has the properties: (a) $k_t(\varnothing) =k_0 (\varnothing)$, $t<
T(\alpha, \alpha')$; (b) its norm in $\mathcal{K}_\alpha$, see \eqref{Y1}, satisfies, cf. \eqref{OC5} 
\begin{equation}
 \label{X54a}
 \|k_t\|_\alpha \leq \frac{T(\alpha, \alpha')}{T(\alpha, \alpha') -t} \|k_0\|_{\alpha'}.
\end{equation}
\end{proposition}
At this point, we should stress that the aforementioned solution
$k_t$ need not be related to any state $\mu\in \mathcal{P}_{\rm
cor}(\Gamma)$. In particular, $k_t$ need not be positive, cf.
\eqref{X13}, which means that Proposition \ref{X5pn} says not too
much about the evolution of states of the model we consider.
This drawback is overcome by a method elaborated in \cite[subsect.
3.2]{KK}, based on Proposition \ref{Y1pn} and certain approximations
of the solution $k_t$. Roughly speaking, this method consists in approximating the initial model by a family of auxiliary models indexed by $\sigma \in (0,1]$, cf. Sect. 6 below, for which the analogs $k_t^\sigma$ of $k_t$ are correlation functions by construction itself, and hence have the properties mentioned in Proposition \ref{Y1pn}. 
It is then proven that these properties persist in the limit $\sigma \to 0$.    
One of the outcomes of this result is the proof of the positivity
of $k_t$, which allows us to continue the evolution $t \mapsto k_t$
to all $t>0$. The corresponding result can be formulated as follows.
\begin{proposition}\cite[Theorem 2.4]{KK}
  \label{X6pn}
Let $\mu_0\in \mathcal{P}_{\rm exp}$ be such that its correlation
function lies in $\mathcal{K}_{\alpha_0}$. Then the solution of the
problem in (\ref{X42}) with $\alpha >\alpha_0$ can be
continued to all $t>0$ uniquely, and in such a way that the following holds
\begin{equation}
  \label{X55}
  0 \leq k_t(\eta) \leq \sum_{\xi \subset \eta} e(\xi;\varrho_t)
  e(\eta\setminus \xi; q_t) k_{\mu_0} (\eta\setminus \xi), \qquad \eta \in \Gamma_{\rm fin},
\end{equation}
where $\varrho_t$ and $q_t$ are as in (\ref{Y}). Moreover, for each
$t>0$, this solution $k_t$ is the correlation function of a unique
$\mu_t \in \mathcal{P}_{\rm exp}$, which satisfies \eqref{NG1}. 
\end{proposition}
Note that the trajectory $t \mapsto k_t$ mentioned in  Proposition
\ref{X6pn} has the following property: if $k_s \in
\mathcal{K}_{\alpha'}$ for some $s\geq 0$ and $\alpha'\in
\mathds{R}$, which can be established by (\ref{X55}) and the type of
$\mu_0$, then by (\ref{X53}) and (\ref{X54}) it follows that
\begin{equation}
  \label{X55a}
   k_{t+s} = Q_{\alpha \alpha'} (t) k_s,
\end{equation}
holding for some $\alpha > \alpha'$ and $t< T(\alpha, \alpha')$.
Then the uniqueness of this trajectory follows by its local
(in $t$) uniqueness, established in Proposition \ref{X5pn}.

\subsection{The proof of Theorem \ref{1tm}}

In  this subsection, we prove that the map $t \mapsto \mu_t \in
\mathcal{P}_{\rm exp}$ obtained according to Proposition \ref{X6pn}
solves (\ref{FPE}), and that it is a unique solution in the sense of Definition \ref{1df}.

We begin by recalling Lemma \ref{Y1lm} and the definition of
$\mathcal{G}_\alpha$ in (\ref{Y6}). Now we set, cf. (\ref{X44}),
\begin{gather}
  \label{X56}
  \breve{L}= \breve{A} + \breve{B},\\[.2cm] \nonumber
 (\breve{A} G) (\eta) = - E(\eta) G(\eta), \\[.2cm] \nonumber
 (\breve{B} G)(\eta) = - \sum_{x\in \eta}\left( \sum_{y\in \eta\setminus x}
 a(x-y)\right) G(\eta \setminus x) + \int_X b(x) G(\eta\cup x) d x.
\end{gather}
For $G\in B_{\rm bs}$, $\breve{L} G$ can be calculated point-wise in $\eta \in \Gamma_{\rm fin}$, see Definition \ref{X1df}, since the sums in (\ref{X56}) are finite for such
$G$. In the sequel, by writing $\breve{L} G$ we mean this kind of action of $\breve{L}$.
At the same time, the main purpose of introducing (\ref{X56}) is to define bounded linear operators
$\breve{L}_{\alpha' \alpha}: \mathcal{G}_{\alpha} \to
\mathcal{G}_{\alpha'}$, $\alpha' < \alpha$, see (\ref{Y7}). One can
show, see \cite[eq. (3.19)]{KK}, that they  satisfy
\begin{equation}
  \label{X57}
  \|\breve{L}_{\alpha' \alpha} \| \leq {\rm RHS (\ref{X40})}.
\end{equation}
By direct inspection, one checks that, see (\ref{X12}),
\begin{equation}
  \label{X57a}
\langle \! \langle L^\Delta_{\alpha \alpha'} k , G \rangle\!\rangle
= \langle \!  \langle k , \breve{L} G
\rangle\!\rangle = \langle \!  \langle k , \breve{L}_{\alpha'\alpha}G
\rangle\!\rangle,
\end{equation}
holding for each $k\in \mathcal{K}_{\alpha'}$, $G\in B_{\rm bs}$ and
$\alpha > \alpha'$, see (\ref{X57}). By \eqref{X57}, this can be extended to $G\in
\mathcal{G}_\alpha$, see Remark \ref{Y1rk}. Moreover, 
it is possible to construct maps $\breve{Q}_{\alpha'\alpha}(t): \mathcal{G}_\alpha\to \mathcal{G}_{\alpha'}$, dual to those described in 
Propositions \ref{X4pn} and \ref{X5pn}, that satisfy
\begin{equation}
 \label{X58a}
 \langle \!\langle  Q_{\alpha \alpha'}(t) k, G \rangle \!\rangle =  \langle \!\langle   k, \breve{Q}_{\alpha' \alpha}(t) G \rangle \!\rangle, \qquad t < T (\alpha, \alpha'),
\end{equation}
whereas the norm \eqref{Y5} of $G_t := \breve{Q}_{\alpha' \alpha}(t) G$ in $\mathcal{G}_{\alpha'}$, cf. \eqref{X54a}, satisfies
\begin{equation}
 \label{X58b}
 |G_t|_{\alpha'} \leq \frac{T(\alpha, \alpha')}{T(\alpha, \alpha') -t} |G|_{\alpha},
\end{equation}
see \cite[eqs. (3.20), (3.21)]{KK} for more details. 
\begin{lemma}
  \label{Y3lm}
For a given $t>0$, let $\mu_t\in \mathcal{P}_{\rm exp}$ be as in
Proposition \ref{X6pn} and $e^{\alpha'}$, $\alpha'\in \mathds{R}$,
be its type, see Definition \ref{X3df} and Lemma \ref{01lm}. Then for all $\alpha >
\alpha'$ and $G\in \mathcal{G}_\alpha$, the following holds
\begin{equation}
  \label{X59}
  \mu_t (L KG) = \mu_t (K \breve{L}_{\alpha'\alpha}G) = \langle \!
  \langle k_t , \breve{L}_{\alpha'\alpha}G \rangle\!\rangle = \langle \!
  \langle L^\Delta_{\alpha \alpha'} k_t , G \rangle\!\rangle.
\end{equation}
\end{lemma}
\begin{proof}
It is possible to show, cf. \eqref{X38a} and \eqref{X57a}, that
\begin{equation}
  \label{X58}
  (L K G) (\gamma) = ( K \breve{L} G) (\gamma) , \qquad G\in
  {B}_{\rm bs}, \quad \gamma\in \Gamma.
\end{equation}
Then the first equality in (\ref{X59}) follows by the corresponding extension 
of (\ref{X58}) to $G\in \mathcal{G}_\alpha$. The second one follows by (\ref{X12}), whereas the
last equality is just (\ref{X57a}).
\end{proof}
\noindent {\it Proof of Theorem \ref{1tm}.} The inclusion $\mathcal{F}\subset \mathcal{F}_{\rm max}$ stated in claim (c) has been proved in Lemma \ref{Y1lm}. Let us prove that
the map $t\mapsto \mu_t$ as in Proposition \ref{X6pn} solves the
Fokker-Planck equation (\ref{FPE}) in the sense of Definition \ref{1df}
for each $F\in \mathcal{F}_{\rm max}$. To prove the validity of item (i), we fix some $T>0$ and then take $\alpha'= \ln (e^{\alpha_0} + \|b\| T)$; hence, $k_{\mu_t} \in \mathcal{K}_{\alpha'}$ for all $t\leq T$.  
Then we take $F=KG$ with $G\in \mathcal{G}_\alpha$ for some $\alpha > \alpha'$. 
Recall that this is possible for any $\alpha'$, see \eqref{NGa}. By (\ref{X59}) it
follows that, see (\ref{X13}),
\begin{gather}
  \label{X60}
  \mu_t (|LF|) = \mu_t (|LKG|) = \mu_t (|K\breve{L}_{\alpha'\alpha}G|) \leq \mu_t (K|\breve{L}_{\alpha'\alpha}G|) \\[.2cm] \nonumber  = \langle \!\langle k_t , |\breve{L}_{\alpha'\alpha}G
  |\rangle\!\rangle = \int_{\Gamma_{\rm fin}} k_t (\eta) |\breve{L}_{\alpha'\alpha}G
 (\eta) | \lambda (d \eta) \\[.2cm]  \nonumber \leq \int_{\Gamma_{\rm fin}} e^{\alpha'|\eta|}|\breve{L}_{\alpha'\alpha}G
 (\eta) | \lambda (d \eta) \leq \|\breve{L}_{\alpha'\alpha} \|
 |G|_\alpha.
\end{gather}
Thus, $|LF|$ is $\mu_t$-absolutely integrable and 
\[
\int_0^T |\mu_t (LF) |dt  \leq \int_0^T \mu_t (|LF|) dt \leq T  \|\breve{L}_{\alpha'\alpha} \|
 |G|_\alpha.
\]
To prove the validity of item (ii) of
Definition \ref{1df}, for $s\geq 0$, we assume the existence of a solution $[0,s]\ni u \mapsto \mu_u$ of \eqref{FPE} with every $F\in \mathcal{F}_{\rm max}$. This map is continuous in the sense that $\mu_{u} \Rightarrow \mu_t$ as $u\to t$, which holds for every $t\in [0,s]$. This weak continuity follows from the very form of \eqref{FPE} and by claims (ii) of Proposition \ref{G1pn} and (b) of Proposition \ref{G2pn}. 
According to Lemma \ref{01lm}, the assumed solution satisfies $\mu_u\in \mathcal{P}_{\rm exp}^{\alpha_s}$ with $\alpha_s \leq \ln (e^{\alpha_0} + \|b\|s)$. 
For $t>0$, by 
(\ref{X55a}) with $\alpha' = \alpha_s$ and some $\alpha > \alpha_s$, and also by (\ref{X52}), for an arbitrary $\alpha'' \in (\alpha_s, \alpha)$, we have
\begin{equation}
  \label{X61}
k_{s+t} = k_s + \int_0^t L^\Delta_{\alpha \alpha''} k_{s+u} du,
\end{equation}
where $k_s= k_{\mu_s}$. 
In view of \eqref{X55a}, $t$ in \eqref{X61} should satisfy 
$t<T( \alpha'', \alpha_s)$. 
By the continuity of $\alpha \mapsto T(\alpha, \alpha')$, see (\ref{X46a}), the latter can be secured by the corresponding choice of $\alpha''$ whenever 
\begin{equation}
 \label{OCTz}
 t < \frac{ \alpha - \alpha_s}{\|b\|e^{-\alpha_s} + \langle a \rangle e^{\alpha}}.
\end{equation}
For such $t$, by Lemma \ref{01lm}, $k_{s+t}$ is the correlation function of a unique $\mu_{t+s}\in \mathcal{P}_{\rm exp}$, the type of which verifies
\begin{equation}
\label{NO}
\varkappa_{\mu_{t+s}} \leq   e^{\alpha_0} + \|b\|(s+t)< e^\alpha.
\end{equation}
The latter inequality holds for $t$ satisfying \eqref{OCTz}, which an elementary calculation can show. 
Hence, $k_{s+u}\in
\mathcal{K}_{\alpha''}$ for all $u\in[0,t]$ and sufficiently small $\alpha - \alpha''$. Now we take $G\in \mathcal{G}_\alpha$, cf.
(\ref{X60}), apply (\ref{X59}) with $\alpha' = \alpha_s$ and obtain from (\ref{X61}) the
following
\begin{gather}
  \label{X62}
 \mu_{s+t} (KG) =  \langle \! \langle k_{s+t} , G\rangle \!
 \rangle = \langle \! \langle k_{s} , G\rangle \!
 \rangle + \langle \! \langle \int_0^t L^\Delta_{\alpha \alpha''} k_{s+u} du , G \rangle \!
 \rangle  \\[.2cm] \nonumber = \langle \! \langle k_{s} , G\rangle \!
 \rangle + \int_0^t \langle \! \langle  k_{s+u} ,
 \breve{L}_{\alpha''\alpha} G\rangle \!
 \rangle du = \mu_s (F)  +
 \int_0^t\mu_{s+u} (L F) du.
\end{gather}
Thus, for each $t$ satisfying \eqref{OCTz}, the map $[s,t]\ni u\mapsto \mu_u$ as in Proposition \ref{X6pn} solves \eqref{FPE} with any $F\in \mathcal{F}_{\rm max}$. Due to the aforementioned continuity of $u \mapsto \mu_u$,   
this implies that the map $[0,s+t]\ni u\mapsto \mu_u$ as in Proposition \ref{X6pn} solves \eqref{FPE} with any $F\in \mathcal{F}_{\rm max}$. Now we show that this map can be continued to all $u>0$. Fix some $\varepsilon \in (0,1)$,  set $s_0=0$ and 
\begin{equation}
 \label{No}
 s_1 = \frac{\varepsilon}{\|b\|e^{-\alpha_0}+ \langle a \rangle e^{\alpha_0 +1}},
\end{equation}
which verifies $\alpha_{s_1} < \alpha_0 +1$, see \eqref{NO}. 
As we have just shown, use \eqref{OCTz} with $s=s_0$ and $\alpha =1$, the map $[0,s_1]\ni u\mapsto \mu_u$ as in Proposition \ref{X6pn} solves \eqref{FPE} with any $F\in \mathcal{F}_{\rm max}$. Next, we set
\begin{equation}
 \label{No1}
 s_2 = \frac{\varepsilon}{\|b\|e^{-\alpha_0}+ \langle a \rangle e^{\alpha_{s_1} +1}} = \frac{\varepsilon}{\|b\|e^{-\alpha_0}+ \langle a \rangle e^{\alpha_0 +1} + \langle a \rangle  \|b\| s_1 }.
\end{equation}
By repeating the same arguments, we conclude that the map $[0,s_1+s_2]\ni u\mapsto \mu_u$ as in Proposition \ref{X6pn} solves \eqref{FPE} with any $F\in \mathcal{F}_{\rm max}$. After $n$ repetitions, we set
\begin{equation}
 \label{No1a}
 s_{n+1}=  \frac{\varepsilon}{\|b\|e^{-\alpha_0}+ \langle a \rangle e^{\alpha_0 +1} + \langle a \rangle  \|b\| (s_1 + \cdots + s_n)},
\end{equation}
and conclude that the map $[0,s_1+s_2 + \cdots + s_{n+1} ]\ni u\mapsto \mu_u$ as in Proposition \ref{X6pn} solves \eqref{FPE} with any $F\in \mathcal{F}_{\rm max}$. This procedure can be repeated ad infinitum. Assume now that $\sum_{n=1}^\infty=S < \infty$. Then ${\rm RHS} (\eqref{No1}) \geq A(s)$, which by \eqref{No1} contradicts the convergence of the series. Thus, the aforementioned solution can be continued to all $u$.

Now we turn to proving uniqueness with the help of the following arguments.
First, we write \eqref{FPE} for $\mu'_t$
\begin{equation}
 \label{FPEz}
 \mu'_{t} (F) = \mu_{0} (F) + \int_{0}^{t} \mu'_s (LF) ds, 
\end{equation}
which must also be satisfied by $\mu_t$ mentioned in Proposition \ref{X6pn}, with the same $\mu_0\in \mathcal{P}_{\rm exp}$. By Lemma \ref{01lm}, $\mu'_t$ lies in $\mathcal{P}_{\rm exp}$ and its correlation function satisfies, cf. \eqref{0qa} and \eqref{Y1}, \eqref{Y2}, 
\begin{gather}
 \label{0A}
 0\leq k_{\mu'_{t}} (\eta) \leq  \left(\varkappa_{\mu_0} + \|b\| t \right)^{|\eta|}, \qquad , \\[.2cm] \nonumber
k_{\mu'_t} \in \mathcal{K}_{\alpha_t}, \qquad {\rm for} \ \ \  \alpha_t := \ln (\varkappa_{\mu_0} + \|b\| t),
 \end{gather}
that must be true for all $t\geq 0$. By \eqref{FPEz} and Remark \ref{01rk} we then get that $k_{\mu'_{t}}$ satisfies, cf. \eqref{X59}, 
\begin{eqnarray}
 \label{0B}
 \langle \! \langle k_{\mu'_{t}} , G^\vartheta \rangle \!
 \rangle & = & \langle \! \langle k_{\mu_{0}} , G^\vartheta \rangle \!
 \rangle + \int_{0}^{t} \langle \! \langle k_{\mu'_{s}} , \breve{L}_{\alpha' \alpha} G^\vartheta \rangle \!
 \rangle ds \\[.2cm] \nonumber & = & \langle \! \langle k_{\mu_{0}} , G^\vartheta \rangle \!
 \rangle + \langle \! \langle L^\Delta_{\alpha \alpha'}  \int_{0}^{t}  k_{\mu'_{s}} ds ,  G^\vartheta \rangle \!
 \rangle ,
\end{eqnarray}
holding for all $\vartheta \subset \Theta$. Here $G^\vartheta$ is such that $F^\vartheta= K G^\vartheta$, i.e., it is given in \eqref{X7}. The integral $\int_{0}^{t}  k_{\mu'_{s}} ds$ is considered in the Banach space $\mathcal{K}_{\alpha'}$, $\alpha'\geq \alpha_{t}$, see \eqref{0A}, whereas $\alpha>\alpha'$ can be arbitrary since $G\in \mathcal{G}_{\alpha}$ for any $\alpha$. To interchange the integrations in \eqref{0B} we used the absolute integrability as in \eqref{X60}, cf. \eqref{X62}. 

The set $\{G^\vartheta: \vartheta \in \Theta\}$ is separating for the $\sigma$-finite positive measures on $\Gamma_{\rm fin}$, including $\lambda$, as it is closed under multiplication and separates points. Assume that an absolutely $\lambda$-integrable function $H:\Gamma_{\rm fin} \to \mathds{R}$ is such that $\int_{\Gamma_{\rm fin}} H G^\vartheta d \lambda =0$ for all $\vartheta \subset \Theta$. Then $H(\eta)=0$ for $\lambda$-almost all $\eta$. Indeed, write $H = H^{+} - H^{-}$, $H^{\pm} \geq 0$, and define $\lambda^{\pm} = H^{\pm }\lambda$. Then the assumed equality implies $\lambda^{+} = \lambda^{-}$ and hence $H^{+}(\eta) = H^{-}(\eta)$, which holds for $\lambda$-almost all $\eta$.
We use this argument in \eqref{0B} and thus get 
\begin{equation}
 \label{0C}
 k_{\mu'_{t}} = k_{\mu_{0} }+ L^\Delta_{\alpha \alpha'}  \int_{0}^{t}  k_{\mu'_{s}} ds,
\end{equation}
which is the equality of vectors in $\mathcal{K}_{\alpha}$.
Thus, $k_{\mu'_{t} }$ is a mild solution of the Cauchy problem in \eqref{X42}. To prove that this fact implies 
$k_{\mu'_{t} } = k_{\mu_{t} }$ (hence $\mu'_{t}  = \mu_{t}$), we adapt standard arguments by which uniqueness of mild solutions is proved when one deals with $C_0$-semigroups, see \cite[Proposition 6.4, page 146]{EN}.   
Fix $t>0$ and choose $\alpha_0$ according to $e^{\alpha_0} = \varkappa_{\mu_0}$, see Lemma \ref{01lm}. For $u< t$, define 
$q_u = k_{\mu_{u}} - k_{\mu'_{u}}$. Since $k_{\mu_t}$ also satisfies \eqref{0A}, then both $q_u$ and $\int_0^u q_s ds$ are in $\mathcal{K}_{\alpha_u}\subset \mathcal{K}_{\alpha_t}$. 
By \eqref{0C} and \eqref{Qqz} it follows that
\begin{equation}
 \label{0Ca}
 q_u = L^\Delta_{\alpha \alpha_t}  \int_{0}^{u} q_s ds,
\end{equation}
that holds in $\mathcal{K}_{\alpha}$ with every  $\alpha>\alpha_t$, see \eqref{Y2}.  
The latter means that we consider each $q_u \in \mathcal{K}_{\alpha_u}$ as an element of $\mathcal{K}_{\alpha_t}$ or $\mathcal{K}_{\alpha}$. In this sense, \eqref{0Ca} can be written as
\begin{equation}
\label{0CAa}
L^\Delta_{\alpha \alpha_t}  \int_{0}^{u} q_s ds = Q_{\alpha\alpha_u} (0) q_u= Q_{\alpha\alpha_t} (0) q_u,
\end{equation}
see \eqref{X53}. Now, for a fixed $\alpha>\alpha_t$,  we take $v\in (0, t]$ such that 
\begin{equation}
 \label{0Cz}
 v< T(\alpha, \alpha_t) = \frac{\alpha -  \alpha_t}{\|b\|e^{-\alpha_t} + \langle a \rangle e^{\alpha}},
\end{equation}
cf. \eqref{X46a}, which is possible for each $t>0$. Then for every $u\in [0,v]$, $Q_{\alpha \alpha_t} (v-u)$ makes sense and has the properties stated in Proposition \ref{X4pn}. Thus, we may consider, cf. \eqref{X52},
\begin{eqnarray}
 \label{0Cb}
& & \frac{d}{du} Q_{\alpha \alpha_t} (v-u)\int_0^u q_s ds  =  Q_{\alpha \alpha_t} (v-u) q_u - Q_{\alpha \alpha'} (v-u)L^{\Delta}_{\alpha' \alpha_t} \int_0^u q_s ds \qquad \\[.2cm]& & \quad =  Q_{\alpha \alpha'} (v-u)
\left[Q_{\alpha' \alpha_t} (0) q_u - L^{\Delta}_{\alpha' \alpha_t} \int_0^u q_s ds \right] =0 \nonumber
\end{eqnarray}
see \eqref{0CAa}. At the same time, integration over $u\in [0,v]$ of both sides of \eqref{0Cb} yields
\begin{equation}
\label{0Cf}
Q_{\alpha_u \alpha'} (0)\int_0^v q_s ds = 0, \quad {\rm hence} \quad \int_0^v q_s ds = 0,  
\end{equation}
which holds for all $v\in (0,t]$ satisfying \eqref{0Cz}. Then $q_v=0$ for all such $v$. Now we recall that $e^{\alpha_t} = e^{\alpha_0 }+\|b\|t$ and that  we have the freedom to choose $\alpha>\alpha_t$, which we use by setting $\alpha= \alpha_t + 1$, and also
\begin{equation}
\label{0Cd}
\tau (t) = (\|b\|e^{\alpha_0} + \langle a \rangle e^{\alpha_0+1} + \langle a \rangle \|b\| t)^{-1}.
\end{equation}
Note that $T(\alpha_t + 1, \alpha_t) < \tau(t)$, see \eqref{0Cz}. Let $t_1>0$ be the unique solution 
of $t=\tau(t)$. By  \eqref{0Ca} the equality in \eqref{0Cf} yields $q_v=0$, which implies $k_{\mu_{v}} = k_{\mu'_{v}}$ for all $v\leq t_1$. To further extend this equality, we rewrite \eqref{0C} in the form
\[
k_{\mu'_{t_1 + t}} = k_{\mu_{t_1}} + L^\Delta_{\alpha \alpha_{t_1+t}} \int_0^tk_{\mu'_{t_1 + s}} ds,
\]
and introduce $q_u = k_{\mu'_{t_1 + u}} - k_{\mu_{t_1 + u}}$, for which the following holds, cf. \eqref{0Ca},
\[
q_u = L^\Delta_{\alpha \alpha_{t_1+t}} \int_0^t q_s ds, \qquad u\leq t.
\]
Now we repeat the arguments as in \eqref{0Cb} with $\alpha = \alpha_{t_1 + t} +1$, subject to condition 
$t\leq t_2$, where $t_2$ is the unique solution of $t= \tau (t_1+t)$, see \eqref{0Cd}. The result is 
$k_{\mu_t} = k_{\mu'_t}$, which is true for all $t\leq t_1 +t_2$. 
Repeating this procedure the appropriate number of times, we obtain $k_{\mu_t} = k_{\mu'_t}$, for all $t\leq t_1 +t_2 + \cdots + t_n$, where $t_n$ verifies, cf.  \eqref{0Cd},
\begin{equation}
\label{0Cg}
t_n = \bigg{(}\|b\|e^{-\alpha_0} + \langle a \rangle  e^{\alpha_0+1} +  \langle a \rangle \|b\|(t_1 + \cdots + t_n)\bigg{)}^{-1}.
\end{equation}
If the series $\sum t_k$ converges, then the left-hand side of \eqref{0Cg} tends to zero as $n\to +\infty$, while the right-hand side remains separated away from zero. Hence, $\sum t_k=+\infty$, which results in $k_{\mu_t} = k_{\mu'_t}$, being valid for all $t>0$. This completes the proof of the theorem.

The proof of the bound in \eqref{Y0} readily follows by \eqref{X55}, while the bound in \eqref{NG1} was proved in \cite[Theorem 2.5]{KK}.

\hfill $\square$

\section{Constructing the Markov Process: Auxiliary Models}

This section aims to prepare the proof of Theorem \ref{2tm}.
As mentioned above, the proof will be done by approximating the initial model described by $L$ given in \eqref{L} by a family of models described by their Kolmogorov operators  $\{L_\sigma : \sigma \in [0,1]\}$ such that $L_0$ coincides with $L$ given in \eqref{L} whereas each $L_\sigma$, $\sigma \in (0,1]$ can be used to construct a Markov transition function, $p^\sigma_t$, by means of which one defines a Markov process, corresponding to $L_\sigma$. Then the process in question is obtained as the weak limit of the Markov processes obtained in this way.    

\subsection{The models}

We begin by recalling that each $\mu\in \mathcal{P}_{\rm exp}$ has the property $\mu(\Gamma_*)=1$, see \eqref{Gas}, \eqref{Gas1}. Keeping this in mind, we set, cf. \eqref{psi1},
\begin{equation}
 \label{z2}
 \psi_\sigma (x) = \frac{1}{1+\sigma |x|^{d+1}}, \qquad \ \sigma \in [0,1]. 
\end{equation}
and 
\begin{gather}
 \label{z3}
 b_\sigma (x) = b(x) \psi_\sigma (x) , \qquad m_\sigma (x) = m(x) \psi_\sigma (x) , \\[.2cm] \nonumber  a_\sigma (x,y) = a(x-y) \psi_\sigma (x) \psi_\sigma (y). 
\end{gather}
Clearly, 
\begin{equation}
 \label{z3a}
 \psi (x) \leq \psi_\sigma (x) \leq \sigma^{-1} \psi (x), \qquad \sigma \in (0,1],
\end{equation}
which yields, cf. \eqref{Gas},
\begin{equation}
 \label{z3b}
 \gamma(\psi) = \varPhi(\gamma) \leq   \varPhi_\sigma (\gamma) \leq \sigma^{-1} \varPhi(\gamma), \quad \Phi_\sigma (\gamma) := \gamma(\psi_\sigma). 
\end{equation}
Then we define
\begin{equation}
 \label{Ls}
L_\sigma F (\gamma) = \int_X b_\sigma(x) \nabla_x F (\gamma) dx - \sum_{x\in \gamma} \left( m_\sigma (x) + \sum_{y\in \gamma\setminus x} a_\sigma (x,y)  \right) \nabla_x F(\gamma \setminus x),    
\end{equation}
where $\nabla_x F (\gamma)$ is as in \eqref{X19}.
Clearly, $L_0$ coincides with the generator defined in \eqref{L}. By \eqref{6} and  \eqref{z3b} we have that
\begin{equation*}
 \sum_{x\in \gamma} m_\sigma (x) \leq \|m\| \varPhi_\sigma (\gamma) \leq \|m\| \sigma^{-1} \varPhi (\gamma),
\end{equation*}
and also, see \eqref{z3} and \eqref{6},
\begin{equation}
 \label{z5}
 \sum_{x\in \gamma} \sum_{y\in \gamma \setminus x}  a_\sigma (x,y)  \leq  \|a\|_\psi \sum_{x\in \gamma} \sum_{y\in \gamma \setminus x} \psi_\sigma (x)\psi_\sigma (y), \qquad \gamma \in \Gamma_*.
\end{equation}
Thereby we set, cf. \eqref{X37}, \eqref{X38}, 
\begin{eqnarray}
 \label{f}
 E_\sigma (\eta) & = & \sum_{x\in \eta} m_\sigma (x) + \sum_{x\in \eta}\sum_{y\in \eta\setminus x} a_\sigma (x,y), \qquad \eta \in \Gamma_{\rm fin}, \\[.2cm] \nonumber (L^\Delta_\sigma k)(\eta) & = &  \sum_{x\in \eta} b_\sigma (x) k(\eta\setminus x) -   E_\sigma (\eta) k(\eta) - \int_X \left(\sum_{y\in \eta} a_\sigma (x,y) \right) k(\eta \cup x) d x.
\end{eqnarray}
Similarly as in \eqref{X38a}, we then have
\begin{equation*}
 \mu(L_\sigma K G) = \langle\! \langle L^\Delta_\sigma k_\mu, G \rangle\! \rangle, \qquad \sigma \in (0,1], 
\end{equation*}
which holds for all $\mu\in \mathcal{P}_{\rm exp}$ and $G\in B_{\rm bs}$. In view of \eqref{z2} and
\eqref{z3}, see also \eqref{6}, \eqref{6a}, \eqref{6b}, the expression for $L^\Delta_\sigma$ in \eqref{Ls} can be used to define bounded linear operators $(L^\Delta_\sigma)_{\alpha\alpha'}: \mathcal{K}_{\alpha'} \to \mathcal{K}_\alpha$ the norms of which can be estimated according to \eqref{X40} and \eqref{X46}, which yields
\begin{equation*}
 \| (L^\Delta_\sigma)_{\alpha\alpha'}\| \leq {\rm RHS(\ref{X40})}.
\end{equation*}
Indeed, to get the letter, one employs the estimate $\psi_\sigma \leq 1$.  Hence, the Cauchy problem in $\mathcal{K}_\alpha$ for $(L^\Delta_\sigma, \mathcal{D}_\alpha)$, with the same domain as in \eqref{X42},
\begin{equation}
 \label{f3}
 \frac{d}{dt} k^\sigma_t = L^\Delta_\sigma k^\sigma_t, \qquad k^\sigma_t|_{t=0} = k_0 \in \mathcal{K}_\alpha, 
\end{equation}
has a unique solution, see Proposition \ref{X5pn}, given by the formula
\begin{equation}
 \label{f4}
 k^\sigma_t = Q^\sigma_{\alpha\alpha'}(t)k_0, \qquad t< T(\alpha,\alpha'),
\end{equation}
with $T(\alpha,\alpha')$ given in \eqref{X46a} and the family of operators $\{Q^\sigma_{\alpha\alpha'}(t): t\in [0,T(\alpha,\alpha')\}$ possessing all the properties established in Proposition \ref{X4pn}. 
\begin{remark}
 \label{g1rk}
Similarly as in Proposition \ref{X6pn} the evolution $t\mapsto k^\sigma_t$ described by \eqref{f4} determines the evolution of states  $t \mapsto \mu^\sigma_t\in \mathcal{P}_{\rm exp}$, $t>0$, the type of  which satisfies, cf. \eqref{0qa}, 
\begin{equation}
 \label{f4a}
\varkappa_{\mu^\sigma_t} \leq  e^{\alpha_t} :=\varkappa_{\mu_0} + \|b\|t. 
\end{equation}
These states $\mu^\sigma_t$ solve the Fokker-Planck equation for $(L_\sigma, \mathcal{F}, \mu_0)$ with the same domain given in \eqref{X18}.  Similarly as in Proposition \ref{X6pn} this solution is unique.
\end{remark}

\subsection{The Markov transition functions}

The above-mentioned transition functions $p^\sigma_t$ will be obtained in the form
\begin{equation}
 \label{f20}
 p^\sigma_t ( \gamma, \cdot) = S^\sigma (t) \delta_\gamma, \qquad t\geq 0, \quad \sigma \in (0,1],
\end{equation}
where $\delta_\gamma$ is the Dirac measure on $\Gamma_*$ centered at $\gamma\in  \Gamma_*$ and 
$S^\sigma (t)$ is a bounded positive operator acting in the Banach space of finite signed measures on  $\Gamma_*$, such that $S^\sigma =\{ S^\sigma (t)\}_{t\geq 0}$ is a stochastic semigroup related to $L_\sigma$ given in \eqref{Ls}. This semigroup will be constructed (see Lemma \ref{7lm} below) by means of the Thieme-Voigt technique developed in \cite{TV} which proved effective in problems like the one considered here. Its detailed presentation in the form adapted to the present context can be found in \cite[Sect. 7]{KR}. Here we just briefly outline the main aspects.   

\subsubsection{The Thieme-Voigt theory}

Let $\mathcal{X}$ be an ordered real Banach space with a generating cone $\mathcal{X}^{+}$ such that the norm of $\mathcal{X}$ is additive on the cone, i.e., $\|x+y\|_{\mathcal{X}} = \|x\|_{\mathcal{X}} + \|y\|_{\mathcal{X}}$, whenever $x,y \in \mathcal{X}^{+}$. By the latter fact there exists a linear positive functional, $\varphi_\mathcal{X}$, such that 
\begin{equation}
 \label{g21}
 \varphi_\mathcal{X} (x) = \|x\|_\mathcal{X}, \qquad {\rm for} \quad x \in \mathcal{X}^{+}. 
\end{equation}
A $C_0$-semigroup $S =\{ S (t)\}_{t\geq 0}$ of bounded linear operators on $\mathcal{X}$ is said to be \emph{stochastic} (resp. \emph{substochastic}) if the following holds $\| S(t) x\|_\mathcal{X} = \|x\|_\mathcal{X}$ (resp. $\| S(t) x\|_\mathcal{X} \leq \|x\|_\mathcal{X}$) for all $x\in \mathcal{X}^{+}$ and $t>0$.   
For a dense linear subset $\mathcal{D}\subset \mathcal{X}$, set $\mathcal{D}^{+} = \mathcal{D}\cap \mathcal{X}^{+}$ and assume that $(A, \mathcal{D})$ and $(B, \mathcal{D})$ are linear operators on $\mathcal{X}$. The Thieme-Voigt theory gives sufficient conditions on this pair of operators under which the closure of 
$(A+B, \mathcal{D})$ is the generator of a stochastic semigroup. Its key aspect is the use of a subspace $\widetilde{\mathcal{X}} \subset \mathcal{X}$ with a specific set of properties listed below. 
\begin{assumption}
 \label{gass1}
The linear subspace $\widetilde{\mathcal{X}} \subset \mathcal{X}$  has the following properties:
\begin{itemize}
 \item [(a)] $\widetilde{\mathcal{X}}$ is dense in  $\mathcal{X}$;
 \item [(b)] there exists a norm, $\|\cdot \|_{\widetilde{\mathcal{X}}}$, that makes $\widetilde{\mathcal{X}}$ a Banach space;
 \item [(c)] $\widetilde{\mathcal{X}}^{+} := \widetilde{\mathcal{X}}
 \cap \mathcal{X}^{+}$ is the generating cone in $\widetilde{\mathcal{X}}$, the norm $\|\cdot \|_{\widetilde{\mathcal{X}}}$ is additive on $\widetilde{\mathcal{X}}^{+}$; 
 \item [(d)] the cone  $\widetilde{\mathcal{X}}^{+}$ is dense in $\mathcal{X}^{+}$. 
\end{itemize}
\end{assumption}
\begin{definition}
 \label{JUNdf}
 The pair $(\mathcal{X}, \widetilde{\mathcal{X}})$ of Banach spaces with the aforementioned properties will be called a Thieme-Voigt pair.   
\end{definition}
By item (c) in Assumption \ref{gass1} there exists a linear positive functional, $\varphi_{\widetilde{\mathcal{X}}}$, cf \eqref{g21}, such that 
\begin{equation}
 \label{g21a}
 \varphi_{\widetilde{\mathcal{X}}} (x) = \|x\|_{\widetilde{\mathcal{X}}}, \qquad {\rm for} \quad x \in \widetilde{\mathcal{X}}^{+}. 
\end{equation}
For a dense linear subset $\mathcal{D}\subset \mathcal{X}$, let $(A,\mathcal{D})$ be a linear operator on $\mathcal{X}$. Define 
$\widetilde{\mathcal{D}} =\{x\in \mathcal{D} \cap \widetilde{\mathcal{X}}: Ax \in \widetilde{\mathcal{X}}  \}$
Then the operator $(A,\widetilde{\mathcal{D}})$ is said to be the \emph{trace} of $(A,\mathcal{D})$ in  
$\widetilde{\mathcal{X}}$. The next statement is an adaptation of \cite[Theorem 2.7]{TV}, see also \cite[Proposition 7.2]{KR}. 
\begin{proposition}
 \label{g1pn}
 Let $(A,\mathcal{D})$ and $(B,\mathcal{D})$ be linear operators on $\mathcal{X}$ which have the following properties
 \begin{itemize}
  \item [(i)] $-A: {\mathcal{D}}^{+} \to {\mathcal{X}}^{+}$ and $B: {\mathcal{D}}^{+} \to {\mathcal{X}}^{+}$; 
  \item [(ii)] $(A,\mathcal{D})$ is the generator of a substochastic semigroup, $S_0=\{S_0(t)\}_{t\geq 0}$, on $\mathcal{X}$ such that $S_0(t) : \widetilde{\mathcal{X}} \to \widetilde{\mathcal{X}}$, holding for all $t\geq 0$, and the restrictions  $S_0(t)|_{\widetilde{\mathcal{X}}}$, $t\geq 0$, constitute a $C_0$-semigroup on $\widetilde{\mathcal{X}}$ generated by $(A,\widetilde{\mathcal{D}})$;
  \item [(iii)] $B: \widetilde{\mathcal{D}} \to \widetilde{\mathcal{X}}$ and 
 \begin{equation*}
  \varphi_{{\mathcal{X}}} ((A+B)x) = 0, \qquad {\rm  for} \ \ {\rm all} 
 \  \ x \in \mathcal{D}^{+};
 \end{equation*}
 \item[(iv)] there exist positive $c$ and $\epsilon$ such that
 \begin{equation*}
  \varphi_{\widetilde{\mathcal{X}}} ((A+B)x) \leq c \varphi_{\widetilde{\mathcal{X}}} (x) - \epsilon \|Ax\|_{\mathcal{X}}, \qquad {\rm  for} \ \ {\rm all} 
 \  \ x \in \widetilde{\mathcal{D}} \cap \mathcal{X}^{+}.
 \end{equation*}
 \end{itemize}
Then the closure of $(A+B,\mathcal{D})$ in $\mathcal{X}$ is the generator of a stochastic semigroup, $S=\{S(t)\}_{t\geq 0}$, on $\mathcal{X}$ which leaves $\widetilde{\mathcal{X}}$ invariant.   
\end{proposition}

\subsubsection{The Banach spaces of measures}

Now we turn to constructing the semigroups $S^\sigma$ that appear in \eqref{f20}. Let $\mathcal{M}$ stand for the set of all finite signed measures on $\Gamma_*$, see \cite[Chapt. 12]{Cohn}. Set $\mathcal{M}^{+} = \{\mu \in \mathcal{M}: \mu (\mathbb{A}) \geq 0, \ \mathbb{A}\in \mathcal{B} (\Gamma_*)\}$. Each $\mu\in \mathcal{M}$ can be uniquely decomposed $\mu= \mu^{+} - \mu^{-}$ with $\mu^{\pm}\in \mathcal{M}^{+}$, which means that the latter is the generating cone in $\mathcal{M}$. Set $|\mu| =  \mu^{+} + \mu^{-}$. Then 
\begin{equation*}
 \|\mu \| := |\mu| (\Gamma_*)
\end{equation*}
is a norm, that is additive on  $\mathcal{M}^{+}$. With this norm $\mathcal{M}$ is a Banach space, see \cite[Proposition 4.1.8, page 119]{Cohn}. For $n\in \mathds{N}$, let now $F_n$ stand for $F^\vartheta$, $\vartheta = \{\psi , \dots , \psi\}$, $|\vartheta|=n$, see \eqref{0n}. Also set $F_0\equiv 1$. Since $\psi(x) \leq 1$, see \eqref{z2}, these functions satisfy
\begin{gather}
 \label{z33a}
 F_n (\gamma\cup x) = F_n (\gamma) + n \psi(x) F_{n-1} (\gamma), \\[.2cm] \nonumber
 F_1 (\gamma) F_n (\gamma) \leq F_{n+1}(\gamma) + n F_n (\gamma), \\[.2cm] \nonumber F_2 (\gamma) F_n (\gamma) \leq F_{n+2}(\gamma) + 2n F_{n+1} (\gamma) + n(n-1) F_n (\gamma).
\end{gather}
As the first two lines of \eqref{z33a} are immediate, we prove the validity of the last. First, we observe that
\[
F_{n+2}(\gamma) = \sum_{x\in \gamma} \psi(x) \sum_{y\in \gamma\setminus x} \psi(y) F_n (\gamma\setminus \{x,y\}),
\]
and
\begin{gather*}
F_n(\gamma)  = F_n (\gamma\setminus x) + n \psi(x) F_{n-1} (\gamma\setminus x)  = 
F_n (\gamma\setminus \{x,y\}) \\ + n \psi(y) F_{n-1} (\gamma\setminus\{ x,y\}) + n \psi(x) F_{n-1} (\gamma\setminus\{ x,y\}) \\+ n(n-1) F_{n-2}(\gamma\setminus \{x,y\}),
\end{gather*}
see the first line in \eqref{z33a}.
Then
\begin{eqnarray*}
F_2(\gamma) F_n (\gamma) & = & F_{n+2} (\gamma) + n \sum_{x\in \gamma} \psi(x) \sum_{y\in \gamma\setminus x} [\psi(y) ]^2F_n (\gamma\setminus \{x,y\})\\ & + & n \sum_{x\in \gamma} [\psi(x) ]^2\sum_{y\in \gamma\setminus x} \psi(y) F_n (\gamma\setminus \{x,y\}) \\ & + & n (n-1) \sum_{x\in \gamma} [\psi(x) ]^2\sum_{y\in \gamma\setminus x} [\psi(y) ]^2F_n (\gamma\setminus \{x,y\}) \\ & \leq & F_{n+2} (\gamma) + 2 n F_{n+1} (\gamma) + n(n-1) F_n(\gamma). 
\end{eqnarray*}
Define
\begin{equation}
 \label{z35}
 \|\mu \|_n = \sum_{k=0}^n \frac{1}{k!}|\mu|(F_k), \qquad \mathcal{M}_n = \{\mu \in \mathcal{M}: \|\mu \|_n< \infty\}, \quad n\in \mathds{N}.
\end{equation}
Clearly, each $\mathcal{M}_n$ is a Banach space. Let $\mathcal{M}^{+}_n$ denote the corresponding cones of positive measures. Then 
\begin{equation*}
 \mathcal{M}_{n+1} \subset  \mathcal{M}_{n} \subset \mathcal{M}, \qquad 
 \mathcal{M}_{n+1}^{+} \subset  \mathcal{M}^{+}_{n} \subset \mathcal{M}^{+}, \ \ n\in \mathds{N},
\end{equation*}
where all inclusions are dense in the corresponding topologies, cf. \cite[Lemma 7.4]{KR}. Let $\mathcal{M}^{1,+}$ consist of all $\mu\in \mathcal{M}^{+}$ for which $\|\mu\|=\mu(\Gamma_*)=1$, i.e, of all probability measures. Then, for each $n\in \mathds{N}$, it follows that
\begin{equation}
 \label{g25}
 \mathcal{P}_{\rm exp} \subset \mathcal{M}_{n} \cap \mathcal{M}^{1,+}.
\end{equation}
Indeed, by \eqref{0na} for $\mu\in \mathcal{P}_{\rm exp}$ one gets
\[
 \mu(F_k) \leq [\varkappa_\mu \langle \psi \rangle]^k, \qquad k \in \mathds{N}.
\]
We conclude this part by noting that, for each $n\in \mathds{N}$, the Banach spaces $\mathcal{M}$ and $\mathcal{M}_n$ satisfy all the conditions of Assumption \ref{gass1}, whereas the linear functionals as in \eqref{g21} and \eqref{g21a} are 
\begin{equation}
 \label{g26}
 \varphi(\mu) := \varphi_{\mathcal{M}} (\mu) = \mu(\Gamma_*), \qquad \varphi_n(\mu) := \varphi_{\mathcal{M}_n} (\mu) = \sum_{k=1}^n \frac{1}{k!}\mu(F_k), 
\end{equation}
respectively. 

\subsubsection{The semigroup}

For $L_\sigma$ introduced in \eqref{Ls}, we define $L^\dagger_\sigma$ by the formula
\begin{equation}
 \label{z29}
(L^\dagger_\sigma \mu)(F) = \mu(L_\sigma F),   
\end{equation}
where $F$ is a suitable function, e.g., $F = KG$ with $G\in B_{\rm bs}$. To this end, we first define the following measure kernel
\begin{equation}
 \label{z30}
 \Omega^\gamma_\sigma (\mathbb{A}) = \int_X b_{\sigma} (x) \mathds{1}_{\mathbb{A}} (\gamma \cup x) dx + \sum_{x\in \gamma} \left(m_\sigma (x) + \sum_{y\in \gamma\setminus x} a_{\sigma} (x,y)\right) \mathds{1}_{\mathbb{A}} (\gamma \setminus x),
 \end{equation}
and the function
\begin{equation}
 \label{z31}
 R_\sigma (\gamma) = \Omega^\gamma_\sigma (\Gamma_*) = \langle b_\sigma \rangle + E_\sigma (\gamma) =: \int_X b_{\sigma} (x) dx + \sum_{x\in \gamma} \left(m_\sigma (x) + \sum_{y\in \gamma\setminus x} a_{\sigma} (x,y)\right).
\end{equation}
By \eqref{z3b} and \eqref{z5} one obtains that $R_\sigma (\gamma)< \infty$ for all $\gamma\in \Gamma_*$.
Then $L^\dagger_\sigma$ can be presented in the form
\begin{gather}
 \label{z32}
 L^\dagger_\sigma = A + B, \\[.2cm] \nonumber (A\mu) (d \gamma) = - R_\sigma (\gamma) \mu( d \gamma), \qquad (B\mu) (d \gamma) = \int_{\Gamma_*}  \Omega^{\gamma'}_\sigma (d \gamma) \mu ( d\gamma').
\end{gather}
Let us show that both $-A$ and $B$ introduced in  \eqref{z32} define positive (unbounded) operators in each of $\mathcal{M}_{n}$.  Set
\begin{equation}
 \label{z36a}
 {\rm Dom}_n (A) := \{\mu \in \mathcal{M}:  R_\sigma|\mu|  \in \mathcal{M}_n\}, \ \  n\in \mathds{N}_0,
\end{equation}
where $\mathcal{M}_0$ is just $\mathcal{M}$.
By \eqref{z3a} and \eqref{z31} it follows that
\begin{equation}
 \label{z36}
 R_\sigma (\gamma) \leq \langle b_\sigma \rangle + \frac{\|m\|}{\sigma} F_1(\gamma) + \frac{\|a\|_\psi}{\sigma^2}F_2 (\gamma). 
\end{equation}
Then by \eqref{z33a} we get
\begin{gather}
 \label{z36q}
 R_\sigma (\gamma) F_k (\gamma) \leq  \langle b_\sigma \rangle F_k(\gamma) + \frac{\|m\|}{\sigma}k F_k(\gamma) + \frac{\|m\|}{\sigma} F_{k+1}(\gamma)\\[.2cm] \nonumber + \frac{\|a\|_\psi}{\sigma^2}k (k-1)F_k (\gamma) + \frac{2\|a\|_\psi}{\sigma^2}k F_{k+1} (\gamma)+ \frac{\|a\|_\psi}{\sigma^2}F_{k+2} (\gamma). 
\end{gather}
We apply this estimate in \eqref{z35} and get
\begin{gather*}
\|A\mu\|_n \leq  \langle b_\sigma \rangle \|\mu\|_n + \frac{2 (n+1) \|m\|}{\sigma} \|\mu\|_{n+1} + \frac{4 (n+1)(n+2)\|a\|_\psi}{\sigma^2}\|\mu\|_{n+2}.
\end{gather*} 
By \eqref{z36a} this yields 
\begin{equation}
 \label{z37}
   \mathcal{M}_{n+2} \subset {\rm Dom}_n (A), \ \ n \in \mathds{N}_0. 
\end{equation}
\begin{remark}
 \label{z1rk}
For each $n\geq 1$, the operator $(A, {\rm Dom}_n(A))$ is the trace of $(A, {\rm Dom}_0 (A))$ in the Banach space $\mathcal{M}_n$.
\end{remark}
For an appropriate $\mu\in \mathcal{M}$, by \eqref{z30} one gets
\begin{eqnarray}
 \label{z38}
 \int_{\Gamma_*} F_n (\gamma) (B\mu) (d \gamma) & = & \int_{\Gamma_*^2} F_n (\gamma') \Omega^\gamma_\sigma (d \gamma') \mu(d \gamma) 
 \\[.2cm] \nonumber & = & \int_{\Gamma_*} \left( \int_X b_\sigma (x) F_n (\gamma \cup x) dx \right)\mu( d \gamma) \\[.2cm] \nonumber & + & \int_{\Gamma_*} \left( \sum_{x \in \gamma} \left[m_\sigma (x)    +  \sum_{y\in \gamma\setminus x} a_\sigma (x,y)\right] F_n (\gamma \setminus x) \right)\mu( d \gamma).
\end{eqnarray}
For $n=0$ and $\mu\in \mathcal{M}^{+}$, this implies 
\begin{equation}
 \label{z39}
 \|B\mu \| = \mu(R_\sigma), \quad {\rm hence} \quad  (L^\dagger_\sigma \mu )(\Gamma_*) = \mu (R_\sigma) - \mu (R_\sigma) = 0,
\end{equation}
see \eqref{z32}.
For $\mu\in \mathcal{M}^{+}$ and $n\geq 1$, by \eqref{z33a} and the evident estimate $F_n (\gamma\setminus x) \leq F_n (\gamma)$ we get from \eqref{z38} the following
\begin{eqnarray*}
 \int_{\Gamma_*} F_n (\gamma) (B\mu) (d \gamma) & \leq & n \|b\| \langle \psi \rangle  \mu (F_{n-1}) +  \int_{\Gamma_*} F_n (\gamma) R_\sigma (\gamma) \mu (d \gamma),
\end{eqnarray*}
which readily yields that
\begin{gather}
 \label{z41}
B : {\rm Dom}_n (A) \to \mathcal{M}_n, \qquad n\geq 2.
 \end{gather}
Now we set, see \eqref{z37}, 
\begin{gather}
 \label{Dom}
 {\rm Dom}(L^\dagger_\sigma) = {\rm Dom}_0(A) = \{ \mu \in \mathcal{M}:  |\mu|(R_\sigma) < \infty\}.
\end{gather}
\begin{lemma}
 \label{7lm}
 For each $\sigma\in (0,1]$, the closure of the operator $(L^\dagger_\sigma, {\rm Dom}(L^\dagger_\sigma))$ in $\mathcal{M}$ is the generator of a stochastic semigroup, $S^\sigma= \{S^\sigma(t)\}_{t\geq 0}$, such that $S^\sigma(t): \mathcal{M}_n\to \mathcal{M}_n$ for each $n\geq 2$ and $t\geq 0$.    
\end{lemma}
\begin{proof}
Our aim is to show that the operator defined in \eqref{z32} and \eqref{Dom} satisfies the conditions of Proposition \ref{g1pn}. By the very definition of $A$ and $B$ in \eqref{z32}, and then by \eqref{z36a}, \eqref{z41} and \eqref{Dom}, it follows that condition (i) is met.
Define
\begin{equation}
 \label{u1}
 S_0(t) \mu(d \gamma) = \exp \left( - t R_\sigma (\gamma)\right) \mu(d \gamma), \qquad \mu \in \mathcal{M}, \ \ t\geq 0. 
\end{equation}
Clearly, for each $n\in \mathds{N}_0$, $S_0(t):\mathcal{M}_n \to \mathcal{M}_n$, acting as a multiplication operator, and $S_0 = \{S_0(t)\}_{t\geq 0}$ is a positive semigroup. Certainly,
\begin{equation}
 \label{u2}
\| S_0(t) \mu\| \leq \|\mu \|, \qquad \mu \in \mathcal{M}^{+}.
\end{equation}
To show the strong continuity of $S_0$ in $\mathcal{M}$, for fixed $\mu\in \mathcal{M}$ and $\varepsilon >0$, we have to find $\delta>0$ such that $\|S_0(t) \mu - \mu\| < \varepsilon$ whenever $t < \delta$. Since $\mathcal{M}_2$ is dense in $\mathcal{M}$, see \eqref{z36a} and \eqref{z37}, we can find $\mu'\in \mathcal{M}_2$ such that $\|(\mu -\mu')^{\pm}\|< \varepsilon/6$. Then by \eqref{u1} and \eqref{u2} we have
\begin{gather*}
 \|S_0(t) \mu - \mu\| \leq \|\mu - \mu'\| + \|S_0(t) (\mu - \mu')\| + \|S_0(t) \mu' - \mu'\|\\[.2cm] \nonumber \leq t \|A \mu'\|+ 2\varepsilon /3 = t |\mu'|( R_\sigma)+  2\varepsilon /3 \leq c_\sigma t \|\mu'\|_2 + 2\varepsilon /3,
\end{gather*}
where $c_\sigma = \max\{\langle b_\sigma \rangle ; \|m\|/\sigma; 2 \|a\|/\sigma^2 \}$, see \eqref{z36}. This estimate yields the strong continuity in question. The strong continuity of $S_0$ in $\mathcal{M}_n$ can be shown in a similar way by employing \eqref{z36q} and \eqref{z37}. Now we take into account Remark \ref{z1rk} and thus conclude that condition (ii) of Proposition \ref{g1pn} is also met. The validity of (iii) follows by \eqref{z41} (about $B$) and \eqref{g26}, \eqref{z39} (about $\varphi$). 
To complete the whole proof, we have to show that, for each $n\geq 2$, there exist positive $c$ and $\epsilon$ such that, for each $\mu \in {\rm Dom}_n (A) \cap \mathcal{M}_n^{+}$, the following holds 
\begin{equation*}
 \varphi_n (L^\dagger_\sigma \mu ) \leq c \varphi_n (\mu ) - \epsilon \mu(R_\sigma),
\end{equation*}
where $\varphi_n$ is as in \eqref{g26}. In view of \eqref{z36}, to this end it is enough to show that  
\begin{equation}
 \label{u3}
\sum_{k=0}^n \frac{1}{k!} \int_{\Gamma_*} F_k (\gamma) (L^\dagger_\sigma \mu) (d \gamma) \leq C  \varphi_n (\mu ), \qquad \mu \in {\rm Dom}_n (A) \cap \mathcal{M}_n^{+},
\end{equation}
holding for some $C>0$. By \eqref{z29} we have
\begin{equation*}
 \int_{\Gamma_*} F_k (\gamma) (L^\dagger_\sigma \mu) (d \gamma) = \mu(L_\sigma F_k).
\end{equation*}
At the same time, by \eqref{z33a} and similarly as in \eqref{X19}, see also \eqref{0d},  we conclude that
\[
 \mu(L_\sigma F_k) \leq \mu(L^{+}_\sigma F_k) \leq \|b\| \langle \psi \rangle k \mu( F_{k-1}),
\]
which by \eqref{g26} yields the validity
of \eqref{u3} with $C = n \|b\| \langle \psi \rangle$. Now the proof follows by Proposition \ref{g1pn}.
\end{proof}

Lemma \ref{7lm} establishes the existence of the transition function \eqref{f20}. It is straightforward that   $p^\sigma_t$ defined in this way satisfies the standard conditions and thus determines finite-dimensional
distributions of a Markov process; see \cite[pages 156, 157]{EK}. The next step, see Sect. 7, is to prove that such processes have cadlag versions.

\subsubsection{The Cauchy problem}

Now we use the semigroup constructed in Lemma \ref{7lm} to solve the following Cauchy problem in $\mathcal{M}$
\begin{equation}
 \label{CP}
 \frac{d}{dt} \mu_t = L^\dagger_\sigma \mu_t, \qquad \mu_t|_{t=0} = \mu_0, \quad   \sigma \in (0,1]. 
\end{equation}
By \cite[Proposition 6.2, page 145]{EN} this problem has a unique solution given by the formula
\begin{equation}
 \label{CP1}
 \hat{\mu}^\sigma_t = S^\sigma(t) \mu_0 = \int_{\Gamma_*} p^\sigma_t (\gamma, \cdot) \mu_0(d \gamma),
\end{equation}
whenever $\mu_0 \in {\rm Dom} (L^\dagger_\sigma)$, see \eqref{Dom}. Here we use notations $\hat{\mu}^\sigma_t$ in order not to mix this solution with the measures ${\mu}^\sigma_t$ mentioned in Remark \ref{g1rk}. By \eqref{g25} and \eqref{z37} $\mu_0\in \mathcal{P}_{\rm exp}$ lies in ${\rm Dom} (L^\dagger_\sigma)$. Hence, $\hat{\mu}^\sigma_t$ lies in each $\mathcal{M}_n$, but a priori not in $\mathcal{P}_{\rm exp}$.  
\begin{lemma}
 \label{CPlm}
For a given $\mu_0\in \mathcal{P}_{\rm exp}$ and $\sigma \in (0,1]$, let the map $t \mapsto \mu^\sigma_t$ be the unique solution of the Fokker-Planck equation for $(L_\sigma, \mathcal{F}, \mu_0)$ mentioned in Remark \ref{g1rk}. Then for each $t>0$, it follows that $\mu^\sigma_t = \hat{\mu}^\sigma_t$, where the latter is obtained as in \eqref{CP1} with the same $\mu_0$.
\end{lemma}
\begin{proof}
Since $\hat{\mu}^\sigma_t$ solves \eqref{CP}, for each $F\in \mathcal{F}$ one has 
\[
 \hat{\mu}^\sigma_t (F) = \hat{\mu}^\sigma_0 (F) + \int_0^t (L_\sigma^\dagger \hat{\mu}^\sigma_s) (F) ds,
\]
which by \eqref{z29} yields that $\hat{\mu}^\sigma_t$ solves the  Fokker-Planck equation for $(L_\sigma, \mathcal{F}, \mu_0)$; hence,  $\hat{\mu}^\sigma_t = {\mu}^\sigma_t$ since the solution is unique.
\end{proof}

\section{The Markov Process}

\subsection{The cadlag paths}
In this subsection, we use Chentsov's theorem \cite{Chentsov} in the version formulated below as Proposition \ref{u1pn}. 
Its presentation is preceded by recalling that the metric $\rho$ introduced in \eqref{rho} makes $\Gamma_*$ a Polish space, see Proposition \ref{Gas1pn}. By means of this metric and the transition function as in \eqref{f20} we define
\begin{eqnarray}
 \label{u4}
 w^\sigma_u (\gamma) & = & \int_{\Gamma_*} \rho (\gamma, \gamma') p^\sigma_u (\gamma, d \gamma'), \\[.2cm] \nonumber W^\sigma_{u,v} (\gamma) & = & \int_{\Gamma_*} \rho (\gamma, \gamma') w^\sigma_u (\gamma') p^\sigma_v (\gamma, d \gamma').
\end{eqnarray}
Now for a certain $\mu \in \mathcal{P}_{\rm exp}$ and a triple $(t_1, t_2 , t_3)$, $0\leq t_1 < t_2 < t_3$, we set
\begin{equation}
 \label{u5}
 \widehat{W}^\sigma(t_1, t_2, t_3) = \int_{\Gamma_*} W^\sigma_{t_3-t_2,t_2 - t_1} (\gamma) \mu^\sigma_{t_1} (d \gamma), \qquad \mu^\sigma_{t} = S^\sigma (t) \mu. 
\end{equation}
see Lemma \ref{CPlm}.
By means of the main result of \cite{Chentsov} and \cite[Theorems 7.2 and 8.6 -- 8.8, pages 128 and 137--139]{EK} one can state the following, cf. \cite[Proposition 7.8]{KR}.
\begin{proposition}
 \label{u1pn}
 Assume that: (a) for each $t>0$, the family $\{\mu^\sigma_t: \sigma \in (0,1]\}\subset \mathcal{P}(\Gamma_*)$ is weakly relatively compact; (b) for each $T>0$, there exists $C(T)>0$ such that for each triple $(t_1, t_2 , t_3)$ satisfying $t_3\leq T$ the following holds
 \begin{equation}
  \label{u6}
  \widehat{W}^\sigma(t_1, t_2, t_3) \leq C(T)(t_3 - t_1)^2. 
 \end{equation} 
Then:
\begin{itemize}
 \item[(i)] For each $\sigma\in (0,1]$, the transition function $p^\sigma$ and $\mu\in \mathcal{P}_{\rm exp}$, see \eqref{u5}, determine a probability measure $P^\sigma_\mu$ on $\mathfrak{D}_{\mathds{R}_{+}} (\Gamma_*)$.
 \item[(ii)] The family $\{P^\sigma_\mu: \sigma\in (0,1]\}$ of the path measures just mentioned is tight, hence possesses accumulation points in the weak topology of $\mathcal{P}(\mathfrak{D}_{\mathds{R}_{+}}(\Gamma_*))$. 
 \end{itemize}
 \end{proposition}
\begin{remark}
 \label{u1rk}
For each $s>0$, one can consider $\widehat{W}^\sigma(t_1, t_2, t_3)$ with $t_1 \geq s$ and  $\mu^\sigma_{t} = S^\sigma (t-s) \mu$. Then by Proposition \ref{u1pn} $p^\sigma$ and $\mu\in \mathcal{P}_{\rm exp}$ determine a probability measure $P^\sigma_{s,\mu}$ on $\mathfrak{D}_{[s,+\infty)} (\Gamma_*)$, and the family $\{P^\sigma_{s,\mu}: \sigma\in (0,1]\}$ is tight in the weak topology of $\mathcal{P}(\mathfrak{D}_{[s,+\infty)}(\Gamma_*))$.
\end{remark}
The measure $P^\sigma_{s,\mu}$ is defined by its finite-dimensional marginals, which in turn are defined by 
$p^\sigma$ and $\mu$ as follows, see eq. (1.10), page 157 in \cite{EK}. For a given $m\in \mathds{N}$, the $m$-dimensional marginal is the following measure 
\begin{eqnarray}
 \label{EKP}
 & & P^\sigma_{s,\mu} ((\mathds{1}_{\mathbb{A}_1}\circ \varpi_{t_1}) \cdots (\mathds{1}_{\mathbb{A}_m}\circ \varpi_{t_m})  )= \int_{\Gamma_*^{m+1}}\mathds{1}_{\mathbb{A}_m} (\gamma_m) p^\sigma_{t_m-t_{m-1}} (\gamma_{m-1}, d \gamma_m)\\[.2cm] \nonumber & & \quad \times \mathds{1}_{\mathbb{A}_{m-1}} (\gamma_{m-1}) p^\sigma_{t_{m-1}-t_{m-2}} (\gamma_{m-2}, d \gamma_{m-1}) \cdots \mathds{1}_{\mathbb{A}_1} (\gamma_1) p^\sigma_{t_1-s} (\gamma_{0}, d \gamma_1) \mu(d\gamma_0),
\end{eqnarray}
where  $\mathbb{A}_1, \dots , \mathbb{A}_m$ are  in $\mathcal{B}(\Gamma_*)$. 
\begin{lemma}
 \label{u1lm}
 For each $\mu\in \mathcal{P}_{\rm exp}$ and $T>0$, the estimate in \eqref{u6} holds for all $\sigma\in (0,1]$ with a $\sigma$-independent $C(T)>0$.
\end{lemma}
\begin{proof}
For each $\gamma \in \Gamma_*$ and $\sigma \in (0,1]$, we have $\delta_\gamma\in {\rm Dom}(L^\dagger_\sigma)$, see \eqref{Dom} and \eqref{z36}. By standard semigroup formulas, e.g. \cite[page 9]{EK}, it then follows 
\begin{equation}
 \label{u7}
 p^\sigma_u (\gamma, \cdot) = \delta_\gamma +\int_0^u L^\dagger_\sigma p_v^\sigma (\gamma, \cdot) d v.
\end{equation}
We apply this in \eqref{u4} and then get
\begin{eqnarray}
 \label{u8}
 w^\sigma_u (\gamma) & = & w^\sigma_0 (\gamma) + \int_0^u \left( \int_{\Gamma_*} \rho(\gamma, \gamma')(L^\dagger_\sigma p_v^\sigma (\gamma, d\gamma'))\right) d v \\[.2cm] \nonumber & = & \int_0^u \left( \int_{\Gamma_*}(L_\sigma \rho(\gamma, \gamma')) p_v^\sigma (\gamma, d\gamma')\right) d v,
\end{eqnarray}
since $w^\sigma_0 (\gamma) = \rho(\gamma, \gamma) = 0$. With the help of the triangle inequality 
we obtain, cf. \eqref{Ls} and \eqref{rho},
\begin{gather}
 \label{u9}
 |\nabla_x \rho(\gamma, \gamma')| = |\rho(\gamma, \gamma'\cup x) - \rho(\gamma, \gamma')| \leq \rho(\gamma'\cup x, \gamma') \leq \psi (x),
\end{gather}
which then yields
\begin{gather}
 \label{u10}
 |L_\sigma \rho(\gamma, \gamma')| \leq \|b\| \langle \psi \rangle + \|m\| F_1(\gamma') + \sum_{x\in \gamma'} \psi(x) \sum_{y\in \gamma'\setminus x} a_\sigma (x,y) =: V_\sigma (\gamma').
\end{gather}
Now, similarly as in \eqref{u8}, by \eqref{u7} and the latter one gets
\begin{eqnarray}
 \label{u11}
 h^\sigma_v (\gamma) & := & \int_{\Gamma_*}(L_\sigma \rho(\gamma, \gamma')) p_v^\sigma (\gamma, d\gamma') \leq \int_{\Gamma_*}V_\sigma (\gamma')p_v^\sigma (\gamma, d\gamma')  \\[.2cm] \nonumber & = & V_\sigma (\gamma) + \int_0^v (L_\sigma V_\sigma (\gamma')) p_s^\sigma (\gamma, d\gamma') ds.
\end{eqnarray}
By \eqref{u10} it follows that
\begin{gather}
 \label{u12}
 \nabla_x V_\sigma (\gamma') = \|m\| \psi(x) + \omega_\sigma (x, \gamma'),  
\end{gather}
where
\begin{equation}
\label{u19c}
\omega_\sigma (x, \gamma') = \sum_{z\in \gamma'} \left(\psi(x) a_\sigma (x,z) + \psi(z) a_\sigma (z,x) \right).
\end{equation}
Since $\nabla_x V_\sigma (\gamma')\geq 0$, cf. \eqref{0b}, \eqref{0d}, then 
\begin{gather}
 \label{u13}
 L_\sigma V_\sigma (\gamma') \leq L^{+}_\sigma V_\sigma (\gamma') = \|m\| \int_X b_\sigma (x) \psi(x) d x \\[.2cm] \nonumber + \int_X b_\sigma (x)\left(\sum_{z\in \gamma'} \psi(z) a_\sigma (z,x) +  \psi(x) \sum_{y\in \gamma'}  a_\sigma (x,y) \right) d x \\[.2cm] \nonumber \leq \|m\| \|b\| \langle \psi \rangle + \|b\| \langle a \rangle F_1(\gamma') +  \|b\| \sum_{x\in \gamma'} f_\sigma (x) =: \widetilde{V}_\sigma (\gamma'),  
\end{gather}
where 
\begin{equation}
 \label{u14}
 f_\sigma (x) = \int_X a_\sigma(z,x) \psi(z) dz.
\end{equation}
We apply \eqref{u13} and \eqref{u7} in \eqref{u11} and get
\begin{eqnarray}
 \label{u15}
  h^\sigma_v (\gamma) &  \leq & V_\sigma (\gamma) + \int_0^v \widetilde{V}_\sigma (\gamma') p^\sigma_s (\gamma ,d \gamma') \\[.2cm] \nonumber & = & 
   V_\sigma (\gamma) +
  v \widetilde{V}_\sigma (\gamma) + \int_0^v \int_0^s \int_{\Gamma_*} \left(L_\sigma \widetilde{V}_\sigma (\gamma') \right) p^\sigma_t (\gamma, d \gamma') dt ds.
\end{eqnarray}
Furthermore, by \eqref{u13} and \eqref{u14}, we get
\begin{gather*}
 L_\sigma \widetilde{V}_\sigma (\gamma') \leq L^{+}_\sigma \widetilde{V}_\sigma (\gamma') \leq \|b\|^2 \left(\langle a \rangle \langle \psi \rangle + \int_X f_\sigma (x) dx  \right) \leq 2 \|b\|^2 \langle \psi \rangle \langle a \rangle  =:C_V,
\end{gather*}
which we use in \eqref{u15} and then in \eqref{u8}, and thereafter obtain
\begin{gather}
 \label{u17}
  w^\sigma_u (\gamma)  \leq \tilde{w}^\sigma_u (\gamma):= u V_\sigma (\gamma) + \frac{u^2}{2} \widetilde{V}_\sigma (\gamma) +\frac{u^3}{6}C_V.
\end{gather}
Now we turn to estimating $W^\sigma_{u,v}$. First we write, see the second line in \eqref{u4} and \eqref{u17}, 
\begin{eqnarray}
 \label{u18}
W^\sigma_{u,v} (\gamma) & \leq & u \Upsilon_{1,v}^\sigma (\gamma) + \frac{u^2}{2} \Upsilon_{2,v}^\sigma (\gamma) +\frac{u^3}{6} \Upsilon_{3,v}^\sigma (\gamma),  \\[.2cm]  \nonumber \Upsilon_{1,v}^\sigma (\gamma)& = & \int_{\Gamma_*} \rho(\gamma, \gamma') V_\sigma (\gamma') p^\sigma_v (\gamma, d \gamma') = \int_0^v \int_{\Gamma_*}L_\sigma \rho(\gamma, \gamma') V_\sigma (\gamma') p^\sigma_s (\gamma, d \gamma') ds,  \\[.2cm]  \nonumber \Upsilon_{2,v}^\sigma (\gamma) & = & \int_{\Gamma_*} \rho(\gamma, \gamma') \widetilde{V}_\sigma (\gamma') p^\sigma_v (\gamma, d \gamma'),   \\[.2cm]  \nonumber \Upsilon_{3,v}^\sigma (\gamma) & = & C_V \int_{\Gamma_*} \rho(\gamma, \gamma') p^\sigma_v (\gamma, d \gamma') . 
\end{eqnarray}
In the same way as in \eqref{u9} by (\ref{u12}) we get
\begin{gather*}
|\nabla_x \rho(\gamma, \gamma') V_\sigma (\gamma')| \leq \rho(\gamma, \gamma'\cup x)\nabla_x V_\sigma (\gamma') + |\nabla_x \rho(\gamma, \gamma')| V_\sigma (\gamma') \\[.2cm] \nonumber \leq \psi(x) V_\sigma (\gamma') + \|m\|\psi(x) + \omega_\sigma (x,\gamma').
\end{gather*}
Thereafter, similarly as in \eqref{u9} and \eqref{u10}, it follows that
\begin{gather}
 \label{u19a}
 |L_\sigma \rho(\gamma, \gamma') V_\sigma (\gamma')| \leq\left( V_\sigma (\gamma')+ \|m\|  \right) \int_X b_\sigma (x) \psi(x)d x \\[.2cm] \nonumber       + \left( V_\sigma (\gamma')+ \|m\|  \right) \sum_{x\in \gamma'}\psi(x) \left[m_\sigma (x) + \sum_{y\in \gamma' \setminus x} a_\sigma(x,y) \right]\\[.2cm] \nonumber + \int_X b_\sigma (x) \omega_\sigma (x, \gamma') dx
 \\[.2cm] \nonumber + \sum_{x\in \gamma'} \left(m_\sigma (x) + \sum_{y\in \gamma'\setminus x} a_\sigma (x,y) \right)\omega_\sigma (x, \gamma'\setminus x),
 \\[.2cm] \nonumber
  \leq V_\sigma (\gamma') [\|m\|+ V_\sigma (\gamma')] + \Delta_\sigma (\gamma')=: V_{1,\sigma}(\gamma'),
\end{gather}
where $\omega_\sigma (x, \gamma')$ is given in \eqref{u19c} and
\begin{gather}
 \label{u19b}
 \Delta_\sigma (\gamma') = \int_X b_\sigma (x) \omega_\sigma (x, \gamma') dx + \sum_{x\in \gamma'} \left(m_\sigma (x) + \sum_{y\in \gamma'\setminus x} a_\sigma (x,y) \right)\omega_\sigma (x, \gamma'\setminus x).
\end{gather}
Then, see \eqref{u18},
\begin{gather}
 \label{u20}
\Upsilon_{1,v}^\sigma (\gamma) \leq \int_0^v \int_{\Gamma_*} V_{1,\sigma} (\gamma') p^\sigma_s (\gamma', d \gamma) d s. 
\end{gather}
Now we recall that $\mu_t^\sigma$ lies in $\mathcal{P}_{\rm exp}$ and thus, see Lemma \ref{01lm}, its correlation functions satisfy
\begin{equation}
 \label{u23}
 0\leq k^{(n)}_{\mu^\sigma_t}(x_1, \dots , x_n) \leq \varkappa(t) := \varkappa_{\mu_0} + \|b\|t, 
\end{equation}
where the right-hand side is independent of $\sigma$. By \eqref{u5} and then by \eqref{u18} one can come to the conclusion that proving \eqref{u6} now amounts to estimating the integrals
\begin{equation}
 \label{u24}
 I^\sigma_{j} (t_2-t_1)= \int_{\Gamma_*}\Upsilon^\sigma_{j,t_2-t_1} (\gamma) \mu^\sigma_{t_1} ( d \gamma), \qquad j=1,2,3.
\end{equation}
Indeed, in this setting we have
\begin{eqnarray}
 \label{u24a}
 \widehat{W}^\sigma (t_1, t_2, t_3) &\leq & (t_3 - t_2) \bigg{(} I^\sigma_1 (t_2 - t_1) \\[.2cm] \nonumber & + & \frac{(t_3 - t_2)}{2} I^\sigma_2 (t_2 - t_1) + \frac{(t_3 - t_2)^2}{6} I^\sigma_3 (t_2 - t_1) \bigg{)}\\[.2cm] \nonumber & \leq & (t_3 - t_2)I^\sigma_1 (t_2 - t_1) + \frac{(t_3 - t_1)^2}{2} \bigg{(} I^\sigma_2 (t_2 - t_1) + \frac{t_3 - t_2}{3} I^\sigma_3 (t_2 - t_1)\bigg{)} .
\end{eqnarray}
Note that the second summand in the last line of \eqref{u24a} already has the desired power of $t_3 - t_1$. Thus, it is enough for us just to get $\sigma$-independent bounds for $I^\sigma_2 (t_2 - t_1)$ and $I^\sigma_3 (t_2 - t_1)$. At the same time, $I^\sigma_1 (t_2 - t_1)$ requires a more accurate estimating. 

By the semigroup property it follows that
\begin{equation}
 \label{u25}
 \int_{\Gamma_*} p^\sigma_{s} (\gamma, d \gamma') \mu_{t_1}^\sigma (d \gamma) = \mu_{t_1+s}^\sigma (d \gamma').  
\end{equation}
We take this into account and by \eqref{u24}, \eqref{u19a}, \eqref{u19b} and \eqref{u20} obtain
\begin{eqnarray}
 \label{u25a}
  I^\sigma_{1} (t_2-t_1) & \leq &  \int_0^{t_2-t_1} \int_{\Gamma_*} V_{1,\sigma} (\gamma) \mu_{t_1+s}^\sigma (d \gamma)ds\\[.2cm] \nonumber & = & I^\sigma_{1a} (t_2-t_1) + I^\sigma_{1b} (t_2-t_1) + I^\sigma_{1c} (t_2-t_1),
\end{eqnarray}
where
\begin{eqnarray}
 \label{u26}
 I^\sigma_{1a} (t_2-t_1) & = &
 \|m\|\int_0^{t_2-t_1} \int_{\Gamma_*} V_{\sigma} (\gamma) \mu_{t_1+s}^\sigma (d \gamma)ds \\[.2cm] \nonumber I^\sigma_{1b} (t_2-t_1) & = & 
 \int_0^{t_2-t_1} \int_{\Gamma_*} [V_{\sigma} (\gamma)]^2 \mu_{t_1+s}^\sigma (d \gamma)ds \\[.2cm] \nonumber I^\sigma_{1c} (t_2-t_1) & = & \int_0^{t_2-t_1} \int_{\Gamma_*} \Delta_{\sigma} (\gamma) \mu_{t_1+s}^\sigma (d \gamma)ds. 
\end{eqnarray}
To estimate the integrals in \eqref{u26} we employ the fact that each of the integrands can be presented as $KG$ with an appropriate $G$, and then use \eqref{X12} and \eqref{u23}.
By \eqref{u10} it follows that
\begin{gather}
 \label{u26a}
V_\sigma = K G_\sigma, \quad G^{(0)}_\sigma = \|b\|\langle \psi \rangle, \quad G^{(1)}_\sigma (x)= \|m\|\psi(x), \\[.2cm] \nonumber G^{(2)}_\sigma (x,y) = \psi(x) a_{\sigma}(x,y) + \psi(y) a_{\sigma}(y,x),
\end{gather}
and $G^{(k)}_\sigma \equiv 0$ for $k>2$. Then
\begin{gather}
 \label{u27}
 I^\sigma_{1a} (t_2-t_1) = \|m\| \int_0^{t_2-t_1}\left( \|b\|\langle \psi \rangle + \|m\| \int_X \psi (x) k_{\mu_{t_1+s}^\sigma}^{(1)} (x) d x  \right. \\[.2cm] \nonumber \left. + \int_{X^2} \psi (x) a_\sigma (x,y) k_{\mu_{t_1+s}^\sigma}^{(2)} (x,y) d x d y \right)ds \\[.2cm] \nonumber \leq (t_2-t_1) \|m\| \langle \psi \rangle \bigg{(}   \|b\| + \|m\|  \varkappa(T) +  \langle a \rangle [\varkappa(T)]^2 \bigg{)} := (t_2-t_1) C_{1a}(T). 
\end{gather}
To proceed further, we write 
\[
V_\sigma = K G_\sigma \star G_\sigma =: K H_\sigma, \quad H_\sigma (\eta) = \sum_{\xi_1\subset \eta}  \sum_{\xi_2\subset \eta\setminus \xi_1} G_\sigma (\xi_1\cup \xi_2)G_\sigma (\eta\setminus \xi_2),
\]
see \eqref{X31}, which by \eqref{u26a} yields
\begin{gather*}
 H^{(0)}_\sigma =\left[\|b\|\langle \psi \rangle \right]^2, \quad H^{(1)}_\sigma (x) = 2\|b\|\|m\|  \langle \psi \rangle \psi(x) +[\|m\|\psi(x) ]^2, \\[.2cm] \nonumber   H^{(2)}_\sigma (x_1, x_2) = 2  \|b\|\langle \psi \rangle  G^{(2)}_\sigma (x_1, x_2) + 2  \|m\|^2 \psi(x_1) \psi(x_2)\\[.2cm] \nonumber +2  \|m\| (\psi(x_1) + \psi(x_2)) G^{(2)}_\sigma (x_1, x_2) + [G^{(2)}_\sigma (x_1, x_2)]^2, \\[.2cm] \nonumber H^{(3)}_\sigma (x_1, x_2, x_3) = 2 \|m\| \psi(x_1) G^{(2)}_\sigma (x_2, x_3) + 2 \|m\| \psi(x_2) G^{(2)}_\sigma (x_1, x_3)\\[.2cm] \nonumber + 2 \|m\| \psi(x_3) G^{(2)}_\sigma (x_1, x_2) + 2 G^{(2)}_\sigma (x_1, x_2)G^{(2)}_\sigma (x_1, x_3)   \\[.2cm] \nonumber + 2 G^{(2)}_\sigma (x_1, x_2)G^{(2)}_\sigma (x_2, x_3) + 2 G^{(2)}_\sigma (x_1, x_3)G^{(2)}_\sigma (x_2, x_3)  ,
\end{gather*}
and
\begin{gather*}
 H^{(4)}_\sigma (x_1, x_2, x_3, x_4) =  G^{(2)}_\sigma (x_1, x_2)G^{(2)}_\sigma (x_3, x_4) +  G^{(2)}_\sigma (x_1, x_3)G^{(2)}_\sigma (x_2, x_4)\\[.2cm] \nonumber +  G^{(2)}_\sigma (x_1, x_4)G^{(2)}_\sigma (x_2, x_3).
 \end{gather*}
Then similarly as in \eqref{u27} by means of the estimates $G^{(2)}_\sigma (x_1, x_2) \leq 2 \|a\|$ and
\[
 \int_{X^2} G^{(2)}_\sigma (x_1, x_2) d x_1 d x_2 \leq 2  \int_{X^2} \psi(x_1) a(x_1- x_2) d x_1 d x_2 = 2\langle \psi\rangle  \langle a\rangle ,
\]
\[
 \int_{X^3} G^{(2)}_\sigma (x, y) G^{(2)}_\sigma (y ,z) d x d y dz \leq 4  \langle a \rangle^2  \langle \psi \rangle, 
\]
we get
\begin{eqnarray}
 \label{u29}
 I^\sigma_{1b} (t_2 - t_1) & = & \int_0^{t_2-t_1} \int_{\Gamma_{\rm fin}} H_\sigma (\eta)  k_{\mu^\sigma_{t_1+s}}(\eta) \lambda ( d \eta) \\[.2cm] \nonumber
 & \leq & (t_2 - t_1) \bigg{(} \left(\|b\|\langle \psi\rangle \right)^2 + \|m\| \langle \psi\rangle  \left(2 \|b\| \langle \psi\rangle +\|m\| \right) \varkappa(T)\\[.2cm] \nonumber & + &  \langle \psi \rangle \left( 2 \|b\| \langle a \rangle\langle \psi \rangle + \|m\|^2 \langle \psi \rangle + 4 \|m\| \langle a \rangle + 2 \|a\| \langle a \rangle\right)[\varkappa(T)]^2 \\[.2cm] \nonumber & + & 2 \langle \psi \rangle \langle a \rangle \left( \|m\| \langle \psi \rangle + 2 \langle a \rangle  \right) [\varkappa(T)]^3 + \frac{1}{2} \left(\langle \psi \rangle \langle a \rangle \right)^2 [\varkappa(T)]^4 \bigg{)} \\[.2cm] \nonumber & =: &(t_2 - t_1) C_{1b}(T).
\end{eqnarray}
By \eqref{u26a}, \eqref{u19c} and \eqref{u19b} one can write
\begin{gather*}
 \Delta_\sigma(\gamma) = \int_X b_\sigma (x) \sum_{y\in \gamma} G^{(2)}_\sigma (x,y) dx + \sum_{x\in \gamma} m_\sigma (x) \sum_{y\in \gamma\setminus x} G^{(2)}_\sigma (x,y)\\[.2cm] \nonumber + \sum_{x\in \gamma}\sum_{y\in \gamma\setminus x} a_\sigma (x,y) G^{(2)}_\sigma (x,y) + \sum_{x\in \gamma}\sum_{y\in \gamma\setminus x} \sum_{z\in \gamma\setminus \{x,y\}} a_\sigma (x,y) G^{(2)}_\sigma (x,z).
\end{gather*}
Then similarly as in \eqref{u29} one obtains
\begin{eqnarray}
 \label{u31}
 I^\sigma_{1c} (t_2 - t_1) & \leq & (t_2 - t_1) 2 \langle \psi \rangle \langle a \rangle \varkappa(t_2) \bigg{(}\|b\| + [\|m\|+ \|a\|]\varkappa(T) + \langle a \rangle [\varkappa(T)]^2 \bigg{)}\qquad \\[.2cm] \nonumber & =: & (t_2 - t_1)C_{1c}(T).
\end{eqnarray}
By \eqref{u25a}, \eqref{u27}, \eqref{u29} and \eqref{u31} we then get
\begin{gather}
 \label{u32}
 I^\sigma_{1} (t_2 - t_1)  \leq (t_2 - t_1) (C_{1a}(T)+ C_{1b} (T) + C_{1c}(T)) =: (t_2 - t_1) C_{1}(T).
\end{gather}
Now similarly as in \eqref{u24} and \eqref{u18}, \eqref{u13} we get
\begin{eqnarray}
 \label{u33}
 I^\sigma_2 (t_2-t_1)  & = & \int_{\Gamma_*^2} \rho (\gamma, \gamma') \widetilde{V}_\sigma (\gamma') p^\sigma_{t_2-t_1}(\gamma, d \gamma') \mu_{t_1}^\sigma (d\gamma) \\[.2cm] \nonumber & \leq &  \int_{\Gamma_*} \widetilde{V}_\sigma (\gamma) \mu_{t_2}^\sigma (d\gamma)\leq \|b\|\langle \psi \rangle \left(\|m\| + 2\langle a \rangle \varkappa(T) \right) =: C_2(T), 
\end{eqnarray}
where we used \eqref{u25} and the fact that $\rho(\gamma, \gamma') \leq 1$, see \eqref{rho}, and the estimate
\[
 \int_X f_\sigma (x) d x\leq \langle \psi \rangle \langle a \rangle, 
\]
see \eqref{u14}. 
Similarly,
\begin{gather}
 \label{u34}
I^\sigma_3 (t_2-t_1) = C_V \int_{\Gamma_*^2} \rho (\gamma, \gamma')  p^\sigma_{t_2-t_1}(\gamma, d \gamma') \mu_{t_1}^\sigma (d\gamma) \leq C_V.
 \end{gather}
Now we employ \eqref{u34}, \eqref{u33} and \eqref{u32} in \eqref{u24a} and thereby come to the conclusion that the desired inequality in \eqref{u6} holds true with
\[
 C(T) = \frac{1}{4}C_1(T) + \frac{1}{2}C_2(T) + \frac{T}{6}C_V, 
 \]
which completes the proof.
\end{proof}

\subsection{The weak convergence}

This subsection aims to prove that the aforementioned states $\hat{\mu}^\sigma_t=\mu^\sigma_t$, see Lemma \ref{CPlm}, weakly converge to the corresponding states $\mu_t$ which solve the Fokker-Planck equation for  $(L,\mathcal{F}, \mu_0)$. By this, we also demonstrate that condition (a) of Proposition \ref{u1pn} is met.
We begin by proving the following lemma in which $\alpha_t$ is as in \eqref{f4a}. Recall also that the Banach spaces $\mathcal{G}_{\alpha}$ are defined in \eqref{Y6}.
\begin{lemma}
 \label{z1lm}
For some $\alpha_0$ and  all $\sigma \in (0,1]$, let $\mu_0$ and $\mu^\sigma_0$ be in $\mathcal{P}^{\alpha_0}_{\rm exp}$. Assume also that   
\begin{equation}
 \label{z6x}
  \forall G \in \mathcal{G}_{{\alpha}_0} \qquad \langle\! \langle k_{\mu^\sigma_0}, G \rangle \! \rangle \to \langle\! \langle k_{\mu_0}, G \rangle \! \rangle  \qquad {\rm as} \ \ \sigma \to 0.
\end{equation}
Next, for these $\mu^\sigma_0$ and $\mu_0$, let $\mu^\sigma_t$ and $\mu_t$ be the unique solutions of the Fokker-Planck equation for $(L_\sigma, \mathcal{F}, \mu_0^\sigma)$ and $(L, \mathcal{F}, \mu_0)$, respectively. Then for each $t>0$, one finds $\tilde{\alpha}_t \geq \alpha_t$ such that the  following holds
\begin{equation}
 \label{z6}
 \forall G \in \mathcal{G}_{\tilde{\alpha}_t} \qquad \langle\! \langle k_{\mu^\sigma_t}, G \rangle \! \rangle \to \langle\! \langle k_{\mu_t}, G \rangle \! \rangle  \qquad {\rm as} \ \ \sigma \to 0. 
\end{equation}
\end{lemma}
\begin{proof}
Recall that $\alpha_t = \ln(e^{\alpha_0} + \|b\|t)$, see \eqref{f4a}. Also, by Proposition \ref{X6pn} 
it follows that $k_{\mu_t} = k_t$, where the latter solves the Cauchy problem in \eqref{X42}, and hence 
\begin{equation}
 \label{z6a}
 k_{\mu_{t+s}} = k_{t+s} = Q_{\alpha \alpha_t} (s) k_t, \qquad s < T(\alpha, \alpha_t), \quad \alpha > \alpha_t.
\end{equation}
A similar representation also holds for $k_{\mu^\sigma_{t+s}} = k^\sigma_{t+s}$, see \eqref{f3}, \eqref{f4}. 

Assume that the convergence stated in \eqref{z6} holds for a given $t\geq 0$, see \eqref{z6x}. Let us then prove that there exists a possibly $t$-dependent  $s_0>0$ such that this convergence holds also for $t+s$ with $s\leq s_0$. 
Write, cf. \eqref{z6a},
\begin{equation*}
 k_{t+s} - k^\sigma_{t+s} = Q_{\bar{\alpha}_t \alpha_t} (s)k_t - Q^\sigma_{\bar{\alpha}_t \alpha_t} (s) k^\sigma_t, 
\end{equation*}
where $\bar{\alpha}_t = {\alpha}_t + 1$ and the equality is written in $\mathcal{K}_{\bar{\alpha}_t}$. In view of \eqref{X52}, we then obtain, cf. \eqref{0Cb},
\begin{eqnarray}
 \label{z8}
 k_{t+s} - k^\sigma_{t+s}& = & Q_{\bar{\alpha}_t \alpha_t} (s)(k_t - k^\sigma_t) - \left(\int_0^s \frac{d}{du} [Q_{\bar{\alpha}_t \alpha_1}(s-u)Q^\sigma_{{\alpha}_1 \alpha_t}(u)] d u \right)k^\sigma_t \\[.2cm] \nonumber & = & Q_{\bar{\alpha}_t \alpha_t}(s) (k_t - k^\sigma_t) + \int_0^s Q_{\bar{\alpha}_t \alpha_2} (s-u)L^\Delta_{\alpha_2\alpha_1}Q^\sigma_{{\alpha}_1 \alpha_t}(u) k^\sigma_t du \\[.2cm] \nonumber & - & \int_0^s Q_{\bar{\alpha}_t \alpha_2} (s-u)(L^{\Delta}_\sigma)_{\alpha_2\alpha_1}Q^\sigma_{{\alpha}_1 \alpha_t}(u) k^\sigma_t du \\[.2cm] \nonumber & = & Q_{\bar{\alpha}_t \alpha_t}(s) (k_t - k^\sigma_t) +\int_0^s Q_{\bar{\alpha}_t \alpha_2} (s-u)\widetilde{L}^{\Delta,\sigma}_{\alpha_2\alpha_1}Q^\sigma_{{\alpha}_1 \alpha_t}(u)  k^\sigma_t du
\end{eqnarray}
where, see \eqref{X38}, \eqref{z3} and \eqref{f},
\begin{gather}
 \label{z9}
 \widetilde{L}^{\Delta,\sigma}_{\alpha'\alpha} k(\eta) := [L^{\Delta}_{\alpha'\alpha} - (L^{\Delta}_{\sigma})_{\alpha'\alpha}  ]k(\eta)  = \sum_{x\in \eta} \tilde{b}_\sigma (x) k(\eta \setminus x) \\[.2cm] \nonumber - \widetilde{E}_\sigma (\eta) k(\eta) - \int_X \sum_{y\in \eta} \tilde{a}_\sigma(x,y) k(\eta \cup x) dx,  
\end{gather}
and
\begin{eqnarray}
 \label{z10}
 & & \tilde{b}_\sigma (x) = b(x) - b_\sigma (x) = b(x)\tilde{\psi}_\sigma(x), \qquad \tilde{\psi}_\sigma(x) = \frac{\sigma |x|^{d+1}}{1+\sigma |x|^{d+1}}\\[.2cm] \nonumber & & \tilde{a}_\sigma (x,y)  = a(x-y)[1 - \psi_\sigma (x)\psi_\sigma (y)] ,  \\[.2cm] \nonumber & & \widetilde{E}_\sigma (\eta) = \sum_{x\in \eta} m(x) \tilde{\psi}_\sigma (x) + \sum_{x\in \eta} \sum_{y\in \eta \setminus x} \tilde{a}_\sigma(x,y). 
 \end{eqnarray}
In the last line in \eqref{z8}, $\alpha_1\in (\alpha_t, \bar{\alpha}_t)$ and  $\alpha_2\in (\alpha_1, \bar{\alpha}_t)$ should be such that 
\begin{equation*}
s < \min\left\{ T(\bar{\alpha}_t, \alpha_2); T({\alpha}_1, \alpha_t) \right\}.
\end{equation*}
Set $\alpha_1 = \alpha_t + \varepsilon$ and find $\varepsilon \in (0,1)$ from the condition
\begin{equation}
 \label{z12}
 T({\alpha}_t + \varepsilon, \alpha_t)  =  T({\alpha}_t+1 , \alpha_t + \varepsilon).
\end{equation}
Recall that $\alpha_t = \ln (e^{\alpha_0} + \|b\|t)$. The equation in \eqref{z12}
has a unique solution which defines a continuous decreasing function $\varepsilon: \mathds{R}_{+} \to (0,1)$ such that $\lim_{t \to +\infty}\varepsilon (t) = \varepsilon_* \in (0,1)$, where the latter is a unique solution of the equation $\varepsilon = (1- \varepsilon ) e^{-(1- \varepsilon )}$. Then we define 
\begin{equation}
 \label{z13}
 \upsilon (t) = T({\alpha}_t + \varepsilon (t), \alpha_t), 
\end{equation}
take some $\epsilon \in (0,1)$, and set
\begin{equation}
 \label{z14}
 s_0 = \epsilon \upsilon (t). 
\end{equation}
It is possible to show that 
\begin{equation}
 \label{z14a}
 \alpha_{t+\upsilon (t)} < \alpha_t + \varepsilon (t), \quad {\rm hence} \quad k_{t+s} \in \mathcal{K}_{\alpha_{t+\upsilon (t)}} \subset \mathcal{K}_{\alpha_1} , \ \ {\rm for}  \ \ s\leq s_0.
\end{equation}
As the map $\alpha \mapsto T(\alpha', \alpha)$ is continuous, see \eqref{X46a}, one finds $\alpha_2 \in (\alpha_t + \varepsilon (t), \alpha_t + 1)$ such that $s_0 < T (\alpha_t + 1, \alpha_2)$. Thus, for the chosen in this way $\alpha_1$ and $\alpha_2$, the operators $Q_{\bar{\alpha}_t \alpha_2} (s-u)$ and $Q^\sigma_{{\alpha}_1 \alpha_t} (u)$  in the last line of \eqref{z8} make sense for all $s\leq s_0$ and $u\leq s$. Now we take $G \in \mathcal{G}_{\bar{\alpha}_t}$ and set $G_s = \breve{Q}_{\alpha_2  \bar{\alpha}_t}(s) G$, $s\leq s_0$, see \eqref{X58a}. This $G_s$ lies in $\mathcal{G}_{\alpha_2} \subset \mathcal{G}_{\alpha_t}$, cf. \eqref{Y7}, and then get from \eqref{z8} the following formula
\begin{gather}
 \label{z15}
 \langle \! \langle k_{t+s} - k^\sigma_{t+s}, G \rangle \! \rangle = \langle \! \langle k_{t} - k^\sigma_{t}, G_s \rangle \! \rangle  + \Upsilon_\sigma (s), \\[.2cm] \nonumber
 \Upsilon_\sigma (s) := \int_0^s \langle \! \langle \widetilde{L}^{\Delta,\sigma}_{\alpha_2\alpha_1} k^\sigma_{t+u}, G_{s-u} \rangle \! \rangle du. 
\end{gather}
Our aim now is to prove that $\Upsilon_\sigma (s) \to 0$ as $\sigma \to 0$. In view of \eqref{z9} and \eqref{z10}, we then write
\begin{eqnarray}
 \label{z16}
 \Upsilon_\sigma (s) & = & \Upsilon^{(1)}_\sigma (s) + \Upsilon^{(2)}_\sigma (s) + \Upsilon^{(3)}_\sigma (s)+ \Upsilon^{(4)}_\sigma (s), \\[.2cm] \nonumber
\Upsilon^{(1)}_\sigma (s) & = &  \int_0^s \left(\int_{\Gamma_{\rm fin}}\sum_{x\in \xi} \tilde{b}_\sigma (x) k^\sigma_{t+u} (\xi \setminus x) G_{s-u} (\xi)  \lambda ( d \xi) \right) d u, \\[.2cm] \nonumber
\Upsilon^{(2)}_\sigma (s) & = & - \int_0^s \left(\int_{\Gamma_{\rm fin}} \sum_{x\in \xi} m(x) \tilde{\psi}_\sigma (x) k^\sigma_{t+u} (\xi)G_{s-u} (\xi)  \lambda ( d \xi) \right) d u, 
\end{eqnarray}
and
\begin{eqnarray}
 \label{z17}
\Upsilon^{(3)}_\sigma (s) & = & - \int_0^s \left(\int_{\Gamma_{\rm fin}} \sum_{x\in \xi} \sum_{y\in \xi \setminus x} \tilde{a}_\sigma (x,y) k^\sigma_{t+u} (\xi)G_{s-u} (\xi)  \lambda ( d \xi) \right) d u
\\[.2cm] \nonumber  \Upsilon^{(4)}_\sigma (s) & = & - \int_0^s \left(\int_{\Gamma_{\rm fin}} \int_X \sum_{y\in \xi}  \tilde{a}_\sigma (x,y) k^\sigma_{t+u} (\xi\cup x)G_{s-u} (\xi) dx \lambda ( d \xi) \right) d u.
\end{eqnarray}
By Lemma \ref{01lm} it follows that $k^\sigma_{t+u}(\xi) \leq e^{\alpha_{t+u}|\xi|} \leq e^{\alpha_{1}|\xi|}$, see \eqref{z14a}. We take this into account in \eqref{z16}, \eqref{z17} and then get
\begin{gather}
 \label{z18}
 |\Upsilon^{(i)}_\sigma (s)| \leq  \int_0^s \int_{\Gamma_{\rm fin}} \widetilde{H}^{(i)}_\sigma (\xi)|G_{s-u}(\xi)| \lambda (d\xi) du, \qquad i=1,2,3,4.
\end{gather}
Here, see \eqref{z10},
\begin{eqnarray}
 \label{z18a}
 \widetilde{H}^{(1)}_\sigma (\xi)  & = & e^{\alpha_1 |\xi| - \alpha_1}  \sum_{x\in \xi} \tilde{\psi}_\sigma(x) b(x) \leq e^{- \alpha_1}\|b\| |\xi|e^{\alpha_1|\xi|}=: C^{(1)} |\xi|e^{\alpha_1 |\xi|}, \\[.2cm] \nonumber \widetilde{H}^{(2)}_\sigma (\xi)  & = & e^{\alpha_1 |\xi|}  \sum_{x\in \xi} \tilde{\psi}_\sigma(x) m(x) \leq \|m\| |\xi|e^{\alpha_1|\xi|}=: C^{(2)} |\xi|e^{\alpha_1 |\xi|}, \\[.2cm] \nonumber \widetilde{H}^{(3)}_\sigma (\xi)  & = & e^{\alpha_1 |\xi|}  \sum_{x\in \xi} \sum_{y\in \xi\setminus x} \tilde{a}_\sigma (x,y) \leq \|a\| |\xi|^2 e^{\alpha_1|\xi|}=: C^{(3)} |\xi|^2e^{\alpha_1 |\xi|}, \\[.2cm] \nonumber \widetilde{H}^{(4)}_\sigma (\xi)  & = & e^{\alpha_1 |\xi| + \alpha_1} \int_X  \sum_{y\in \xi} \tilde{a}_\sigma (x,y) dx \leq e^{\alpha_1} \langle a \rangle |\xi| e^{\alpha_1|\xi|} =: C^{(4)} |\xi|e^{\alpha_1 |\xi|}.
\end{eqnarray}
By the elementary inequality $se^{-\beta s} \leq 1/\beta e$, $s, \beta >0$, we  obtain
\begin{eqnarray}
 \label{z20}
 {\rm RHS \eqref{z18}} & \leq & \frac{C^{(i)}}{e(\alpha_2 - \alpha_1)} \int_0^s \int_{\Gamma_{\rm fin}} e^{\alpha_2 |\xi|}|G_{s-u}(\xi)| \lambda (d\xi) du\\[.2cm] & = & \frac{C^{(i)}}{e(\alpha_2 - \alpha_1)} \int_0^s|G_{s-u}|_{\alpha_2} du  \nonumber \\[.2cm] \nonumber & \leq & \frac{s C^{(i)}}{e(\alpha_2 - \alpha_1)} \frac{T(\bar{\alpha}_t, \alpha_2) |G|_{\bar{\alpha}_t}}{T(\bar{\alpha}_t, \alpha_2) - s_0}, \qquad i=1,2,4, \nonumber
\end{eqnarray}
where we have used also \eqref{X58b}. For $i=3$, by the same calculations we have
\begin{gather}
 \label{z23}
   {\rm RHS \eqref{z18}} \leq  \frac{4s \|a\|}{[e(\alpha_2 - \alpha_1)]^2} \frac{T(\bar{\alpha}_t, \alpha_2) |G|_{\bar{\alpha}_t}}{T(\bar{\alpha}_t, \alpha_2) - s_0}.
\end{gather}
In view of \eqref{z10}, each of $\widetilde{H}^{(i)}_\sigma (\xi)$, $i=1, \dots , 4$ decreases to zero as $\sigma \to 0$. By \eqref{z18}, \eqref{z20}, \eqref{z23} and the monotone convergence theorem this yields  
\begin{equation}
 \label{z21}
 \Upsilon^{(i)}_\sigma (s) \to 0,  \qquad {\rm as} \quad  \sigma \to 0, 
\end{equation}
holding for all $i=1, \dots , 4$.
By \eqref{z21} and \eqref{z15},   we prolong  
the convergence as in \eqref{z6} from $t$ to $t+s$, $s\leq s_0$, see \eqref{z14}. Our aim now is to prolong it ad infinitum. Define
\begin{equation}
 \label{z25}
 t_l = t_{l-1} + s_{l}, \quad t_0 = 0, \quad s_{l} = \epsilon \upsilon (t_{l-1}), \quad l \in \mathds{N}.  
\end{equation}
Since $k_0 = k^\sigma_0 = k_{\mu_0}$, the arguments developed above yield the stated convergence for $t\leq \sup_{l} t_l = \lim_{l\to +\infty} t_l$.
Then, to complete the proof, we have to show that $t_l \to + \infty$.  Assume that   $\sup_{l} t_l = t_* < +\infty$.  By \eqref{z25} it follows that $t_l = s_1 +\cdots + s_l$; hence, $s_l \to 0$ in this case. The function $\upsilon (t)$ defined in \eqref{z13} is continuous, for both $\varepsilon (t)$ and $\alpha_t$ being continuous. By \eqref{z25} we would then  have 
\[
 \upsilon (t_*) = \frac{\varepsilon (t_*)}{\|b\|e^{-\alpha_{t_*}} + \langle a \rangle e^{\alpha_{t_*}+ \varepsilon (t_*)}}= 0, 
\]
which is obviously impossible. This completes the proof.
 \end{proof}
 The result proved above enables us to achieve the primary objective of this subsection.
\begin{lemma}
 \label{z2lm} 
Let $\mu_0$, $\mu^\sigma_0$, $\mu_t$ and $\mu^\sigma_t$ be as in Lemma \ref{z1lm}. Then for each $t\geq 0$, it follows that $\mu^\sigma_t \Rightarrow \mu_t$ as $\sigma \to 0$, where we mean the weak convergence of measures on $\Gamma_*$. 
\end{lemma}
\begin{proof}
By Proposition \ref{G2pn} it follows that $\widetilde{F}_v = K G_v$ with $G_v\in \mathcal{G}_\alpha$ for an arbitrary $\alpha \in \mathds{R}$. Then we apply \eqref{X12} and obtain by Lemma \ref{z1lm} that
\begin{equation*}
  \mu_t^\sigma (\widetilde{F}_v) = \langle \!\langle k_{\mu_t^\sigma} , G_v\rangle \!\rangle \to \langle \!\langle k_{\mu_t} , G_v\rangle \!\rangle   = \mu_t (\widetilde{F}_v), \quad \sigma \to 0,
\end{equation*}
holding for all $t\geq 0$ and $v\in \mathcal{V}$. Then the proof follows by statement (ii) of Proposition \ref{G1pn}.
\end{proof}

\subsection{The proof of Theorem \ref{2tm}}

Let the marginals of $P^\sigma_{s,\mu}$ be defined as in \eqref{EKP}. 
By Lemmas \ref{z1lm} and \ref{u1lm} both conditions (a) and (b) of Proposition \ref{u1pn} are met. Therefore,
for each $\sigma \in (0,1]$, by statement (i) of the latter, see also Remark \ref{u1rk}, a given measure $\mu \in \mathcal{P}_{\rm exp}$  determines probability measures $P^\sigma_{s,\mu}$, $s\geq 0$, through their marginals given in \eqref{EKP}. In particular, this yields
\begin{equation}
 \label{g44}
 P^\sigma_{s,\mu}\circ \varpi^{-1}_t = S^\sigma (t-s) \mu, \qquad t\geq s.
\end{equation}
For each $\mu\in\mathcal{P}_{\rm exp}$, by Lemma \ref{CPlm} it follows that $S^\sigma (t-s) \mu=\mu_t$ also lies in $\mathcal{P}_{\rm exp}$, and $\varkappa_{\mu_t} \leq \varkappa_\mu + \|b\|(t-s)$.
This yields that the family $\{P^\sigma_{s,\mu}: s\geq 0, \mu \in \mathcal{P}_{\rm exp}\}$ satisfies conditions (a), (b) and (c) of Definition \ref{Cad1df}.  To prove that condition (d) is also met, 
we take $\sf J$ as in \eqref{Cad3} with given fixed $m$, $J_1, \dots J_m$ in $\widetilde{\mathcal{F}}$, see \eqref{G1} and item (d) of Remark \ref{Finrk}, and $s \leq s_1 < s_2 < \cdots < s_m \leq t_1$. Now we take $F\in \widetilde{\mathcal{F}}$ and then write 
\begin{equation}
 \label{g45}
 {\sf F}_u = F\circ \varpi_u, \quad {\sf L}_u = LF\circ \varpi_u, \quad {\sf L}^\sigma_u = L_\sigma F\circ \varpi_u, \quad  u\in [s_m, t_2].
\end{equation}
Define
\begin{eqnarray}
 \label{u35}
\tilde{\varsigma}^\sigma_{1,s_1} = c_{1,\sigma} J_1 \mu^\sigma_{s_1} = c_{1,\sigma} J_1 S^\sigma(s_1-s) \mu, \qquad c_{1,\sigma}^{-1} =  \mu^\sigma_{s_1} (J_1) >0. 
 \end{eqnarray}
The latter inequality follows by the fact that all $J_l$ in $\sf J$ are strictly positive. 
Since $\mu_{s_1}^\sigma\in \mathcal{P}_{\rm exp}^{\alpha_{s_1}}$, $e^{\alpha_{s_1}} =  \varkappa_\mu + \|b\|(s_1-s)$, by item (c) of Remark \ref{X1rk} it follows that $\tilde{\varsigma}^\sigma_{1,s_1}$ is in  $\mathcal{P}_{\rm exp}^{\tilde{\alpha}_{s_1}}$ with $\tilde{\alpha}_{s_1} = \alpha_{s_1} + \alpha_1$, where $\alpha_1\geq 0$ depends only on the choice of $J_1$. Next we set
\begin{gather}
 \label{Ag}
 \varsigma_{2,s_2}^\sigma = S^\sigma(s_2-s_1) \tilde{\varsigma}^\sigma_{1,s_1}, \qquad c_{2,\sigma}^{-1} =  \varsigma_{2,s_2}^\sigma (J_2)>0, \\[.2cm] \nonumber \tilde{\varsigma}^\sigma_{2,s_2} = c_{2,\sigma} J_2 \varsigma_{2,s_2}^\sigma.
\end{gather}
Similarly as above, $\tilde{\varsigma}^\sigma_{2,s_2}\in \mathcal{P}_{\rm exp}^{\tilde{\alpha}_{s_2}}$ with $\tilde{\alpha}_{s_2} = 
\tilde{\alpha}_{s_1} + $

By \eqref{u35} it follows that

and Lemma \ref{CPlm} $\varsigma^\sigma_{s_1}$ solves the Fokker-Planck equation for $(L_\sigma, \mathcal{F}, \varsigma_{s})$, hence also for $(L_\sigma, \mathcal{F}_{\rm max}, \varsigma_{s})$, see statement (c) of Theorem \ref{1tm}, where $\varsigma_{s} = c_1 J_1 \mu$. By Lemma \ref{z2lm} it then follows that $\varsigma^\sigma_{s_1}
\Rightarrow \varsigma_{s_1}\in \mathcal{P}_{\rm exp}^{\tilde{\alpha}_{s_1}}$ as $\sigma \to 0$, where $\varsigma_{s_1}$ is the solution of the Fokker-Planck equation for $(L, \mathcal{F}_{\rm max}, \varsigma_{s})$. 
Now we define recursively
\begin{eqnarray*}
\varsigma^\sigma_{s_l} = c_l J_l  S^\sigma(s_l-s_{l-1}) \varsigma^\sigma_{s_{l-1}} , \qquad l=2, \dots, m, 
 \end{eqnarray*}
where $c_l$ is such that $\varsigma^\sigma_{s_l} $ is a probability measure. Then  $\varsigma^\sigma_{s_l}\in \mathcal{P}_{\rm exp}^{\tilde{\alpha}_{s_l}}$ where $\tilde{\alpha}_{s_l}$ is independent of $\sigma$. 
Also 
\begin{equation}
 \label{u35b}
 \varsigma^\sigma_{s_l} \Rightarrow \varsigma_{s_l}, \qquad \sigma \to 0,\qquad l=2, \dots, m,
\end{equation}
where $\varsigma_{s_l}$ is the solution of the Fokker-Planck equation for $(L, \mathcal{F}_{\rm max}, \varsigma_{s_{l-1}})$.
By \eqref{EKP} and the latter we then conclude that for $u \geq s_m$ the following holds
\begin{gather}
 \label{u36}
 P^\sigma_{s,\mu} ({\sf F}_u {\sf J}) = C_\sigma P^\sigma_{s_m,\varsigma^\sigma_{s_m}} ({\sf F}_u ) = C_\sigma P^\sigma_{s_m,\varsigma^\sigma_{s_m}} (F \circ \varpi_u)= C_\sigma\varsigma^\sigma_u (F),\\[.2cm] \nonumber \varsigma^\sigma_u = S^\sigma (u-s_m) \varsigma^\sigma_{s_m},  \qquad C_\sigma= P^\sigma_{s,\mu} ( {\sf J})>0. 
\end{gather}
By the second line of \eqref{u36} and \eqref{u35b} the map $u \mapsto \varsigma^\sigma_u$ solves the Fokker-Planck for $(L_\sigma, \mathcal{F}_{\rm max}, \varsigma^\sigma_{s_m})$ on the time interval $[s_m, +\infty)$. And also 
\begin{equation}
 \label{u35c}
 \varsigma^\sigma_{u} \Rightarrow \varsigma_{u}, \qquad \sigma \to 0,
\end{equation}
where $\varsigma_{u}$ solves the Fokker-Planck for $(L, \mathcal{F}_{\rm max}, \varsigma_{s_m})$ on the same time interval. Moreover, 
\begin{equation}
 \label{u37c}
 \varsigma^\sigma_u, \varsigma_u \in \mathcal{P}_{\rm exp}^{\tilde{\alpha}_u}, \qquad \tilde{\alpha}_u = \ln[ e^{\tilde{\alpha}_{s_m}} + \|b\|(u- s_m)].
\end{equation}
Thereby, see \eqref{g45}, 
\begin{gather}
 \label{u37}
 P^\sigma_{s,\mu} ({\sf H}) = P^\sigma_{s,\mu} ({\sf F}_{t_2} {\sf J}) - P^\sigma_{s,\mu} ({\sf F}_{t_1} {\sf J}) - \int_{t_1}^{t_2} P^\sigma_{s,\mu} ({\sf L}^\sigma_{u} {\sf J}) d u \\[.2cm] \nonumber = C_\sigma\bigg{(} \varsigma^\sigma_{t_2} (F) - \varsigma^\sigma_{t_1} (F) - \int_{t_1}^{t_2} \varsigma^\sigma_{u} (L_\sigma F)du\bigg{)} = 0.
 \end{gather}
By Lemmas \ref{u1lm} and \ref{z1lm}, and hence by Proposition \ref{u1pn}, the set $\{P^\sigma_{s,\mu} : \sigma \in (0,1] \}$ has accumulations points. Let  
$P_{s,\mu}$ be one of them and $\{\sigma_n\}_{n\in \mathds{N}}$ be such that $\sigma_n \to 0$ and $P^{\sigma_n}_{s,\mu} \Rightarrow P_{s,\mu}$ as $n\to +\infty$. Define
\begin{equation}
 \label{u37a}
 \hat{\varsigma}_u (\mathbb{A}) = C^{-1} P_{s,\mu} ((\mathds{1}_{\mathbb{A}}\circ \varpi_u){\sf J}), \quad C= P_{s,\mu} ({\sf J}), \quad u\in [s_m,t_2], \quad \mathbb{A}\in \mathcal{B}(\Gamma_*). 
\end{equation}
Then $C_{\sigma_n}$ and $\varsigma^{\sigma_n}_u$ defined in \eqref{u36} satisfy, see Lemma \ref{z1lm},
\begin{equation}
 \label{u37b}
 C_{\sigma_n} \to C , \qquad \varsigma^{\sigma_n}_u \Rightarrow \hat{\varsigma}_u = {\varsigma}_u, \qquad n \to +\infty.
\end{equation}
The equality $\hat{\varsigma}_u = {\varsigma}_u$, which follows by \eqref{u35c}, yields that the measure defined in \eqref{u37a} solves the Fokker-Planck for $(L, \mathcal{F}_{\rm max}, \varsigma_{s_m})$ on $[s_m, +\infty)$. In view of \eqref{u37}, one can write 
\begin{gather}
 \label{u38}
 P_{s,\mu} ({\sf H})  =  P_{s,\mu} ({\sf H}) - P^{\sigma_n}_{s,\mu} ({\sf H}) = a_n (t_2) - a_n (t_2) - \int_{t_1}^{t_2} b_n (u) d u - \int_{t_1}^{t_2} c_n (u) d u  , \\[.2cm] \nonumber a_n (u)  =   P_{s,\mu} ({\sf F}_{u} {\sf J}) - P^{\sigma_n}_{s,\mu} ({\sf F}_{u} {\sf J}), \quad b_n (u)  =   P_{s,\mu} ({\sf L}_{u} {\sf J}) - P^{\sigma_n}_{s,\mu} ({\sf L}_{u} {\sf J}), \\[.2cm] \nonumber c_n (u)  = P^{\sigma_n}_{s,\mu} (({\sf L}_{u} - {\sf L}^{\sigma_n}_u)  {\sf J}), \quad u\in [s_m, t_2].
\end{gather}
For fixed $s_i$, $i=1,\dots, m$, and $u\in [s_m, t_2]$, let $\Pi$ and $\Pi_n$ be the $(m+1)$-dimensional marginals of $P_{s,\mu}$ and $P^{\sigma_n}_{s,\mu}$, respectively. Then, see \eqref{Cad3},
\begin{eqnarray*}
P_{s,\mu} ({\sf F}_u {\sf J}) &  = & \int_{\Gamma_*^{m+1}} F(\gamma_{m+1}) J_m (\gamma_{m}) \cdots J_1 (\gamma_{1}) \Pi (d \gamma_{1},\dots , d \gamma_{m}, d \gamma_{m+1}), \\[.2cm] \nonumber P^{\sigma_n}_{s,\mu} ({\sf F}_u {\sf J}) & = & \int_{\Gamma_*^{m+1}} F(\gamma_{m+1}) J_m (\gamma_{m}) \cdots J_1 (\gamma_{1}) \Pi_n (d \gamma_{1},\dots , d \gamma_{m}, d \gamma_{m+1}).        
\end{eqnarray*}
By \cite[Theorem 7.8, page 131]{EK}, $P^{\sigma_n}_{s,\mu} \Rightarrow P_{s,\mu}$ implies $\Pi_n \Rightarrow \Pi$ as $n\to +\infty$. Since all $J_i$ and $F$ are taken from $\widetilde{\mathcal{F}}$, see item (ii) of Proposition \eqref{G1pn}, the latter implies $a_n (u) \to 0$ for each  $u\in [s_m, t_2]$.   

To proceed further, by \eqref{u38},  \eqref{u37a} and  \eqref{u36} we write 
\begin{equation}
 \label{u39}
 b_n(u) = C (\varsigma_u (L F) - \varsigma^{\sigma_n}_u (LF)) + (C-C_{\sigma_n})  \varsigma^{\sigma_n}_u (LF).
\end{equation}
By \eqref{NGa} it follows that $F = KG$ with $G \in \mathcal{G}_\alpha$, holding for each $\alpha \in \mathds{R}$. Then, see \eqref{X38a}, \eqref{X57a} and \eqref{u37c}, we have
\begin{gather}
 \label{u40}
 |\varsigma^{\sigma_n}_u (LF)| = |\langle \! \langle k_{\varsigma^{\sigma_n}_u}, \breve{L}_{\tilde{\alpha}_u \alpha} G \rangle \! \rangle | \leq \int_{\Gamma_{\rm fin}} e^{\tilde{\alpha}_u|\eta|}|\breve{L}_{\tilde{\alpha}_u \alpha} G (\eta)| \lambda ( d \eta) \leq \|\breve{L}_{\tilde{\alpha}_u \alpha}\||G|_\alpha,
\end{gather}
where the operator norm $\|\breve{L}_{\tilde{\alpha}_u \alpha}\|$ satisfies the same estimate as in \eqref{X40}. By \eqref{u37b} the second summand on the right-hand side of \eqref{u39} vanishes as $n \to + \infty$ since the bound in \eqref{u40} is independent of $n$. At the same time, similarly as in \eqref{u40} we have
\begin{gather}
 \label{u41}
 \varsigma_u (LF) -\varsigma^{\sigma_n}_u (LF) = \langle\! \langle k_{\varsigma_u}, \breve{L}_{\tilde{\alpha}_u \alpha} G  \rangle\! \rangle - \langle\! \langle k_{\varsigma^{\sigma_n}_u}, \breve{L}_{\tilde{\alpha}_u \alpha} G  \rangle\! \rangle \to 0, \quad n\to +\infty, 
\end{gather}
which follows by Lemma \ref{z1lm}. Then $b_n(u) \to 0$ as $n\to +\infty$, holding for each $u\in [s_m, t_2]$. Finally, similarly as in \eqref{u36}, see also \eqref{z9}, we have
\begin{gather*}
 c_n (u) = C_{\sigma_n} \varsigma^{\sigma_n}_u ((L-L_\sigma) F)= C_{\sigma_n} \langle\! \langle \widetilde{L}^{\Delta,\sigma_n}_{\alpha\tilde{\alpha}_u } k_{\varsigma^{\sigma_n}_u}, G \rangle\! \rangle,
\end{gather*}
where $G$ is the same as in \eqref{u40}, \eqref{u41}, i.e., such that $F=KG$. Recall that $G\in \mathcal{G}_\alpha$ for each $\alpha \in \mathds{R}$, see \eqref{NGa}. Then similarly as in \eqref{z15}, \eqref{z16} we may write
\begin{eqnarray}
 \label{u43}
 & &\langle\! \langle \widetilde{L}^{\Delta,\sigma_n}_{\alpha\tilde{\alpha}_u } k_{\varsigma^{\sigma_n}_u}, G \rangle\! \rangle = \Upsilon^{(1)}_n + \Upsilon^{(2)}_n +\Upsilon^{(3)}_n + \Upsilon^{(4)}_n,\\[.2cm] \nonumber & & \Upsilon^{(1)}_n = \int_{\Gamma_{\rm fin} } \sum_{x\in \xi}\tilde{b}_{\sigma_n} (x) k_{\varsigma^{\sigma_n}_u} (\xi\setminus x) |G(\xi)| \lambda (  d\xi),\\[.2cm] \nonumber & & \Upsilon^{(2)}_n = \int_{\Gamma_{\rm fin} } \sum_{x\in \xi} m(x) \tilde{\psi}_{\sigma_n}(x) k_{\varsigma^{\sigma_n}_u} (\xi) |G(\xi)| \lambda (  d\xi),\\[.2cm] \nonumber & & \Upsilon^{(3)}_n = \int_{\Gamma_{\rm fin} } \sum_{x\in \xi} \sum_{y\in \xi \setminus x}\tilde{a}_{\sigma_n}(x,y)k_{\varsigma^{\sigma_n}_u} (\xi) |G(\xi)| \lambda (  d\xi),\\[.2cm] \nonumber & & \Upsilon^{(4)}_n = \int_{\Gamma_{\rm fin} } \int_X \sum_{y\in \xi} \tilde{a}_{\sigma_n}(x,y) k_{\varsigma^{\sigma_n}_u} (\xi\cup x) |G(\xi)| \lambda (  d\xi) dx.
\end{eqnarray}
Now, similarly as in \eqref{z18}, \eqref{z18a}, we estimate
\[
 \Upsilon^{(i)}_n \leq  \int_{\Gamma_{\rm fin} } h^{(i)}(\xi)|G(\xi)| \lambda (  d\xi),\qquad i=1,2,3,4,
\]
where 
\begin{gather*}
 h^{(1)}(\xi) = \|b\|e^{-\tilde{\alpha}_u}|\xi| e^{\tilde{\alpha}_u|\xi|} =: c_1 |\xi| e^{\tilde{\alpha}_u|\xi|}, \quad h^{(2)}(\xi) = \|m\||\xi| e^{\tilde{\alpha}_u|\xi|} =: c_2 |\xi| e^{\tilde{\alpha}_u|\xi|}, \\[.2cm] \nonumber h^{(3)}(\xi) = \|a\||\xi|^2 e^{\tilde{\alpha}_u|\xi|} =: c_3 |\xi|^2 e^{\tilde{\alpha}_u|\xi|}, \quad h^{(4)}(\xi) = \langle a\rangle e^{\tilde{\alpha}_u}|\xi| e^{\tilde{\alpha}_u|\xi|} =: c_4 |\xi| e^{\tilde{\alpha}_u|\xi|}.
\end{gather*}
By these estimates, we obtain
\begin{eqnarray}
 \label{u45}
 \Upsilon^{(i)}_n \leq \frac{c_i}{e(\alpha - \tilde{\alpha}_u)}|G|_\alpha, \quad i=1,2,4, \qquad \Upsilon^{(3)}_n \leq \frac{4c_3}{e(\alpha - \tilde{\alpha}_u)^2}|G|_\alpha, 
\end{eqnarray}
holding with an arbitrary $\alpha > \tilde{\alpha}_u$ since 
$G\in \mathcal{G}_\alpha$ for each $\alpha \in \mathds{R}$. By \eqref{u45} each $\Upsilon^{(i)}_n$ is bounded. At the same time, $\tilde{b}_{\sigma_n}(x)$, $\tilde{\psi}_{\sigma_n}(x)$, $\tilde{a}_{\sigma_n}(x,y)$ are monitone functions decreasing to zero. By the monotone convergence theorem one then concludes that each of the summands in the first line of \eqref{u43} converges to zero as $n\to +\infty$, which by \eqref{u38} concludes the proof that $P_{s, \mu}({\sf H})=0$; hence, the accumulation point under discussion solves the restricted martingale problem.     

Assume now that there exists another accumulation point, say $\{P'_{s, \mu}: s\geq 0, \mu \in \mathcal{P}_{\rm exp}\}$, which also solves the restricted martingale problem as we just have shown. Then the one-dimensional marginals of  
$P_{s, \mu}$ and $P'_{s, \mu}$ should coincide, see Remark \ref{Finrk}, due to the uniqueness stated in Theorem \ref{1tm}. Thus, to complete the whole proof, we have to show that all finite-dimensional marginals of these measures coincide; see Definition \ref{Cad1df}. Take $F_1 \in \widetilde{\mathcal{F}}$ and $t_1 > s$ and define the following measures on $\mathfrak{D}_{[t_1, +\infty)}(\Gamma_*)$ by setting
\begin{equation}
 \label{u46}
 Q_{t_1} (\mathbb{A}) = \frac{P_{s,\mu}(\mathds{1}_{\mathbb{A}}F_1\circ \varpi_{t_1})}{P_{s,\mu}(F_1\circ \varpi_{t_1})} , \quad Q'_{t_1} (\mathbb{A}) = \frac{P'_{s,\mu}(\mathds{1}_{\mathbb{A}}F_1\circ \varpi_{t_1})}{P'_{s,\mu}(F_1\circ \varpi_{t_1})}.
\end{equation}
Recall that $F_1(\gamma)>0$, see \eqref{G1} and Proposition \ref{G2pn}. By the inductive assumption, it follows that
\begin{equation*}
  Q_{t_1} \circ \pi^{-1}_{t_1} = Q'_{t_1} \circ \pi^{-1}_{t_1} =: \varsigma_{t_1}. 
\end{equation*}
By \eqref{g44} and \eqref{u35} it follows that $\varsigma^\sigma_{t_1} \Rightarrow \varsigma_{t_1}$ as $\sigma \to 0$, where  is as in \eqref{u35} with $J_1 = F_1$. For $t_2 > t_1$ and $F_2 \in \mathcal{F}_{\rm max}$, define, cf. \eqref{Cad1},
\begin{gather*}
 {\sf H}_1 (\bar{\gamma}) = F_2 (\varpi_{t_2} (\bar{\gamma})) - F_2 (\varpi_{t_2} (\bar{\gamma})) - \int_{t_1}^{t_2} (LF_2) (\varpi_{u} (\bar{\gamma})) du \\[.2cm] \nonumber
 {\sf H}_2 (\bar{\gamma}) = {\sf H}_1 (\bar{\gamma}) F_1 (\varpi_{t_1} (\bar{\gamma}))
\end{gather*}
By \eqref{u46} and the assumed properties of $P_{s,\mu}$ and $P'_{s,\mu}$ it follows that
\begin{equation*}
 Q_{t_1} ({\sf H}_1) = \frac{P_{s,\mu} ({\sf H}_2)}{P_{s,\mu}(F_1\circ \varpi_{t_1})} = 0, \qquad Q'_{t_1} ({\sf H}_1) = \frac{P'_{s,\mu} ({\sf H}_2)}{P_{s,\mu}(F_1\circ \varpi_{t_1})} = 0 ,
\end{equation*}
by which both $Q_{t_1}\circ \varpi_{t}^{-1}$ and $Q'_{t_1}\circ \varpi_{t}^{-1}$, $t>t_1$, solve the  
Fokker-Planck equation for $(L, \mathcal{F}_{\rm max}, \varsigma_{t_1})$ on the time interval $[t_1, +\infty)$. By Theorem \ref{1tm} we then conclude that $Q_{t_1}\circ \varpi_{t}^{-1}=Q'_{t_1}\circ \varpi_{t}^{-1} \in \mathcal{P}_{\rm exp}$, which implies in turn that the two dimensional marginals of $P_{s,\mu}$ and $P'_{s,\mu}$ also coincide. Then the extension of this statement to all finite-dimensional marginals goes by induction. This completes the proof.

\section*{Appendix}
\subsection*{Outline of the proof of Proposition \ref{JUN1pn}}
Although there may be several ways to prove this statement, we take advantage of the scheme developed in Section 6 of this work.

Below, by $\Lambda$ we mean the box $(-\ell , \ell]^d \subset X$ with some $\ell >0$. Then we introduce the functions $\widetilde{F}_0(\gamma) \equiv 1$, $\widetilde{F}_n (\gamma) = Q^{(n)}_\gamma (\Lambda)$, $n\in \mathds{N}$, which obviously satisfy  
\begin{eqnarray}
 \label{A1}
 & & \widetilde{F}_n (\gamma)  =  \widetilde{F}_n (\gamma\setminus x) + n \mathds{1}_\Lambda (x) \widetilde{F}_{n-1} (\gamma\setminus x), \\ \nonumber & & \widetilde{F}_1 (\gamma)\widetilde{F}_n (\gamma)  \leq  \widetilde{F}_{n+1} (\gamma) + n \widetilde{F}_n (\gamma),
\end{eqnarray}
see \eqref{z33a}. Let $\widetilde{\mathcal{M}}$ denote the set of all finite signed measures on $\Gamma$, and
\begin{eqnarray}
\label{A2}
\|\mu\|_n = \sum_{k=1}^n \frac{1}{k!}|\mu|(\widetilde{F}_k), \quad \widetilde{\mathcal{M}}_n = \{ \mu \in \widetilde{\mathcal{M}}: \|\mu\|_n < \infty\}, \quad n \in \mathds{N}, \quad \widetilde{\mathcal{M}}_0 =\widetilde{\mathcal{M}}.
\end{eqnarray}
Next, we set, cf. \eqref{g26}, 
\begin{equation*}
\widetilde{\varphi}(\mu)= \widetilde{\varphi}_0(\mu)= \mu(\Gamma), \quad  \widetilde{\varphi}_n(\mu)= \sum_{k=0}^n \frac{1}{k!} \mu (\widetilde{F}_k),  \quad n \in \mathds{N}.
\end{equation*}
For each $n$, $\widetilde{\mathcal{M}}_n$ is a Banach space such that $\|\mu\|_n= \widetilde{\varphi}_n(\mu)$ for $\mu \in \widetilde{\mathcal{M}}_n^{+}$, and $(\widetilde{\mathcal{M}}, \widetilde{\mathcal{M}}_n)$ is a Thieme-Voigt pair; see Definition \ref{JUNdf}. 

For a measurable $m: X \to [0, m_*]$, which will be chosen in the sequel, consider
\begin{equation}
 \label{A4}
 M_\Lambda (\gamma) = \sum_{x\in \gamma} m(x) \mathds{1}_\Lambda (x).
\end{equation}
By \eqref{A1} we then have
$ M_\Lambda (\gamma) \widetilde{F}_n (\gamma) \leq m_* \widetilde{F}_{n+1} (\gamma) + m_* n \widetilde{F}_n (\gamma)$.
Now we define, cf. \eqref{z32},
\begin{gather}
\label{A6}
L^\dagger_\Lambda = A_\Lambda+B_\Lambda, \quad {\rm Dom}(L^\dagger_\Lambda) = \{\mu\in \widetilde{M}: |\mu|(M_\Lambda) < \infty\},  \\ \nonumber (A_\Lambda\mu)(d\gamma) = - M_\Lambda (\gamma) \mu (d\gamma), \quad (B_\Lambda\mu)(d\gamma) = \int_\Gamma \Omega^{\gamma'}_\Lambda (d\gamma)  \mu (d\gamma'), \\ \nonumber
\Omega^\gamma_\Lambda (\mathds{A}) := \sum_{x\in \gamma} m(x) \mathds{1}_\Lambda (\gamma) \mathds{1}_{\mathds{A}} (\gamma\setminus x). 
\end{gather}
Repetition of the arguments used in the proof of Lemma \ref{7lm} leads to the following
\begin{itemize}
    \item[(a)] The closure of the operator $(L^\dagger_\Lambda , {\rm Dom}(L^\dagger_\Lambda))$, defined in $\widetilde{\mathcal{M}}$ by \eqref{A6}, is the generator of a stochastic semigroup $S_\Lambda = \{S_\Lambda (t)\}_{t\geq 0}$ such that $S_\Lambda (t): \widetilde{\mathcal{M}}_n \to \widetilde{\mathcal{M}}_n $ for all $n\geq 1$ and $t>0$. 
    \item[(b)] For each $\mu \in \mathcal{P}_{\rm cor}(\Gamma)$, the map $t \mapsto \mu_t = S_\Lambda (t) \mu\in \mathcal{P}_{\rm cor}(\Gamma)$ is the unique solution of the Fokker-Planck equation \eqref{FPE} for the operator $L_\Lambda$ defined as follows
    \begin{eqnarray*}
   (L_\Lambda F)(\gamma) = - \sum_{x\in \gamma } m(x) \mathds{1}_\Lambda (x) [F(\gamma) - F(\gamma\setminus x)], \quad F \in B_{\rm b} (\Gamma).
    \end{eqnarray*}
\end{itemize}
For $F= KG$, $G\in B^\star_{\rm bs}$, we then have 
\[
(LF )(\gamma) = - \sum_{x\in \gamma} m(x) \mathds{1}_\Lambda (x) \sum_{\xi \subset \gamma\setminus x} G(\xi\cup x) = - \sum_{\xi \subset \gamma} M_\Lambda (\xi)G(\xi). 
\]
Let $\chi_t$ be the correlation measure of the aforementioned solution $\mu_t$. Then
\[
\chi_t(G) = \chi_\mu (G) - \int_0^t \int_{\Gamma_{\rm fin}} G(\xi) M (\xi) \chi_s (d\xi) ds , \quad M(\xi) = \sum_{x\in \xi }m(x).
\] 
The latter holds for all $G\in B_{\rm bs}^\star$,  the support of which is contained in $\Lambda$. Then 
\[
\int_{\Gamma_{\rm fin}} G(\xi) \prod_{x\in \xi}e^{-tm(x)} \chi_{\mu} (d\xi) = \chi_t (G) \geq 0,
\]
since $\chi_t$ is the correlation measure for $\mu_t\in \mathcal{P}_{\rm cor}(\Gamma)$. To complete the proof one takes $m(x)$ and $t$ such that $q(x)= e^{-tm(x)}$.  

\subsection*{Outline of the proof of Proposition \ref{X6pn}}

The proof of Proposition \ref{X6pn} is essentially based on \cite[Lemma 3.5]{KK}, the complete proof of which is given in that paper. Its basic ingredients can be summarized as follows. Let $\mathcal{K}_\alpha$, $\alpha \in \mathds{R}$, be the Banach spaces introduced in Section 2.4, cf. \eqref{Y2}. Define
\begin{equation}
\label{A9}
\mathcal{K}_\alpha^\star = \{ k \in \mathcal{K}_\alpha: \langle\!\langle k, G \rangle\!\rangle \geq 0 \ \ {\rm for} \ 
{\rm all} \ \ G\in B^\star_{\rm bs}\},
\end{equation}
see \eqref{Y4}. Let also the operators $Q_{\alpha \alpha'}$ be as in Proposition \ref{X4pn}. Then Lemma 3.5 of \cite{KK} states that $Q_{\alpha \alpha'}(t) : \mathcal{K}_\alpha^\star\to \mathcal{K}_{\alpha'}^\star$, that holds for all $t< T(\alpha, \alpha')/2$. The proof of this property is performed with the help of a double approximation of the maps $k_\mu \mapsto k_t= Q_{\alpha \alpha'}(t) k_\mu$. The first approximation consists in constructing the families $
\{Q_{\alpha \alpha'}^\sigma: \sigma\in (0,1]\}$, similar to those used in Section 6 of this work. Here, however, instead of \eqref{z2} one takes $\psi_\sigma (x) = \exp(-\sigma |x|^2)$. To show $\langle\!\langle Q_{\alpha \alpha'}^\sigma k , G  \rangle\!\rangle\geq 0$ for $G\in B^\star_{\rm bs}$, one further approximates the evolution $k_\mu \mapsto k^\sigma_t$, see \cite[sect. 3.2.2]{KK}, by restricting it to compact $\Lambda\subset X$ -- similarly to the approximation used right above. The necessity of this second approximation is caused by the aforementioned presence of the logistic term in \eqref{L}. Thereafter, in Propositions 3.7 -- 3.10, with the help of the Thieme-Voigt theory, it is proved that $Q_{\alpha \alpha'}^\sigma$ has the property in question. Finally, like in Section 7.2 above, one proves $\langle\!\langle k_t^\sigma, G \rangle\!\rangle \to \langle\!\langle k_t, G \rangle\!\rangle$, $\sigma \to 0$, see \cite[Proposition 3.6]{KK}. Then the proof of Proposition \ref{X6pn} goes in a standard way, see Section 3.3 of \cite{KK}.

\end{document}